\documentclass[reqno,11pt]{amsart}
\usepackage{amsmath, amssymb,  psfrag, multicol, float, graphicx, amsthm, mathrsfs, mathtools}
\usepackage[numbers]{natbib}
\usepackage[labelfont=bf]{caption}
\usepackage[all,cmtip]{xy}
\usepackage[dvips]{color}
\usepackage{pdfpages, enumerate}
\usepackage[mathscr]{euscript}
\let\euscr\mathscr

\newtheorem{thm}{Theorem}[section]
\newtheorem{lemma}[thm]{Lemma}

\newtheorem{corollary}[thm]{Corollary}
\newtheorem{propn}[thm]{Proposition}

\newtheorem{rmk}[thm]{Remark}
\newtheorem{Fact}{Fact}

\newtheorem{exmpl}[thm]{Example}
\newtheorem{defn}[thm]{Definition}

\begin{document}

\title[The Quantum Steenrod Squares and Their Algebraic Relations]{A construction of the quantum Steenrod squares and their algebraic relations}

\author{Nicholas Wilkins}
\address{School of Mathematics, University of Bristol, Bristol BS8 1UG, UK, and Heilbronn Institute for Mathematical Research, Bristol, UK}
\email{nicholas.wilkins@bristol.ac.uk}
\date{version: \today}

\begin{abstract}
		We construct a quantum deformation of the Steenrod square construction on closed monotone symplectic manifolds, based on the work of Fukaya, Betz and Cohen. We prove quantum versions of the Cartan and Adem relations. We compute the quantum Steenrod squares for all $\mathbb{CP}^n$ and give the means of computation for all toric varieties. As an application, we also describe two examples of blowups along a subvariety, in which a quantum correction of the Steenrod square on the blowup is determined by the classical Steenrod square on the subvariety. 
\end{abstract}

\maketitle

\section{Introduction}

We begin with the background of the Steenrod squares. We will then mention quantum cohomology and the results in this paper. 
The Steenrod squares are cohomology operations that are uniquely defined by a set of axioms, although this uniqueness does not include a construction. There are multiple ways of constructing the squares, one of which involves constructing the operations on $H^{*}(K(\mathbb{Z}/2,n);\mathbb{Z}/2)$ for Eilenberg-MacLane spaces $K(\mathbb{Z}/2,n)$. 

For a topological space $M$, the Steenrod squares are additive homomorphisms $$ Sq^{i} : H^{n}(M) \rightarrow H^{n+i}(M)$$ 
using $\mathbb{Z}/2$ coefficients. They generalize the squaring operation on cohomology with respect to the cup product, $x \mapsto x \cup x$. The $\mathbb{Z}/2$ coefficients ensure that the $Sq^i$ are additive. These Steenrod squares together determine a degree doubling operation that we call the Steenrod square,
$$ Sq: H^*(M) \rightarrow H^*(M)[h],$$ where $$Sq(x) = \sum Sq^{|x|-i}(x) \, h^{i} .$$
Here $h$ is a formal variable in degree 1 that represents the generator of $$H^{*}(\mathbb{RP}^{\infty}; \mathbb{Z}/2)=\mathbb{Z}/2[h].$$

The Steenrod square satisfies the Cartan relation
\begin{equation} \label{equation:introcartan} Sq(x \cup y) = Sq(x) \cup Sq(y) \end{equation}
which, for example, allows one to inductively compute the Steenrod squares for the cohomology of $\mathbb{CP}^n$ (which we will review in Example \ref{exmpl:classcpn}). The Steenrod square also satisfies the Adem relations, which are relations between compositions of the $Sq^i$. Namely, for all $p,q>0$ such that $q<2p$, 
\begin{equation} \label{equation:ademrel} Sq^{q}Sq^{p} = \sum_{s=0}^{[q/2]} {{p-s-1}\choose{q-2s}}Sq^{p+q-s} Sq^{s} \end{equation} where $[q/2]$ is the integer part of $q/2$. The Adem relations are classically implied by the axioms.

This paper begins in Section \ref{sec:preliminaries} with a preliminary section that explains in more detail the relevant background material.

We will then describe two different constructions of the Steenrod square in Section \ref{sec:morssteensqu}: the first construction uses Morse homology and the second uses intersections of cycles. The first construction is based on the definition for Floer theory by Seidel in \cite{seidel}, the origins of which are in the flowlines construction of Betz in \cite{betz}, and Fukaya in \cite{fukaya}, the former of which was extended to a more categorical definition by Betz, Cohen and Norbury in \cite{betzcoh,cohnor}. The second construction we give will be isomorphic to the first construction, using the isomorphism between Morse and singular cohomology.

After considering these constructions of the Steenrod square, in Section \ref{sec:SqQviaMorse} we extend them to define a quantum Steenrod square on the quantum cohomology of a closed monotone symplectic manifold $(M,\omega)$.

In Section \ref{subsec:quantcupprod} we will give details of the quantum cohomology $QH^*(M)$. Briefly, $QH^*(M,\omega)$ is $H^*(M)[[t]]$ as a vector space using a graded formal variable $t$ of degree $2$. However, the cup product is deformed by quantum contributions from counting 3-pointed genus zero Gromov-Witten invariants. That is, by counting certain $J$-holomorphic spheres in $M$ where $J$ is an almost complex structure on $M$ compatible with $\omega$. We often abbreviate by $T= t^N$ for $N$ the minimal Chern number. 

The quantum Steenrod square will be a degree doubling operation, denoted $Q \mathcal{S}$, where \begin{equation} \label{equation:sqstatement} Q\mathcal{S} : QH^{*}(M) = H^{*}(M)[[t]] \rightarrow H^{*}(M)[[t]][h] = QH^{*}(M)[h]. \end{equation} As in the case of the classical Steenrod square, $Q\mathcal{S}$ will be built using additive homomorphisms $Q\mathcal{S}_{i,j}: QH^*(M) \rightarrow QH^{2*-i-2jN}(M)$, so $$Q\mathcal{S}(x) = \sum_{i,j \ge 0} Q\mathcal{S}_{i,j}(x) h^i T^j.$$

The quantum Steenrod square is not necessarily axiomatically defined, but a construction was first suggested by Fukaya in \cite{fukaya} based on his Morse homotopy theory. Our construction is different from Fukaya's, and can be viewed as a Morse theory analogue of the work by Seidel in Floer theory in \cite{seidel}. The first goal of this paper is to solve an open problem posed by Fukaya in \cite[Problem 2.11]{fukaya} as to whether the Adem and Cartan relations hold for quantum Steenrod squares and, if not, what their quantised versions should be. Our second goal is to explore consequences of the solution to this problem, specifically in computations for certain closed monotone symplectic manifolds.

In answer to the first part of Fukaya's problem,  the immediate generalisation of the Cartan and Adem relations fail. In the case of the Cartan relation, this means that it is not in general true that $Q \mathcal{S}(x*y) = Q \mathcal{S}(x) * Q \mathcal{S}(y)$. We will show this in the following example.

	\begin{exmpl}
	\label{exmpl:difficulties}
		In Definition \ref{defn:mqss} of the quantum Steenrod square, we will see that \begin{equation} \label{equation:qssttot2} Q \mathcal{S}(aT) = Q \mathcal{S}(a)T^2 \end{equation} for any $a \in QH^*(M)$. 

		Let $M = \mathbb{P}^{1}$. Let $x$ be the generator of $H^{2}(M)$. Recall that the quantum product is $x*x=T$, where $T$ has degree $4$. Then $$Q\mathcal{S}(x*x) = Q\mathcal{S}(T)=T^{2},$$ using \eqref{equation:qssttot2} and the fact that $Q \mathcal{S}(1) = 1$. Using degree reasons, knowledge of the classical Steenrod square for $\mathbb{P}^1$ and of the quantum cohomology ring, one can show that $Q\mathcal{S}(x) = xh^{2}+T$. Then $$Q\mathcal{S}(x)*Q\mathcal{S}(x) = (xh^{2}+T)*(xh^{2}+T) = Th^{4}+T^{2}.$$

		Hence, in this case $Q \mathcal{S}(x) * Q \mathcal{S}(x) \neq Q \mathcal{S}(x*x)$.
	\end{exmpl}

In Section \ref{sec:quancar} we will prove why the Cartan relation does not immediately generalise and compute the actual quantum Cartan relation. Briefly, the quantum Cartan relation is deformed because the moduli space $\overline{M}_{0,5}$ of genus zero stable curves with 5 marked points $(z_0,z_1,z_2,z_3,z_4)$ has non-trivial $\mathbb{Z}/2$-equivariant cohomology, under the $\mathbb{Z}/2$ action that transposes marked points via the permutation $(12)(34)$. More precisely, the two configurations in Figure \ref{fig:m05elmts} that determine $Q\mathcal{S}(x*y)$ and $Q\mathcal{S}(x) * Q\mathcal{S}(y)$ are not connected by a $\mathbb{Z}/2$-invariant path in $\overline{M}_{0,5}$.

We will prove that a quantum deformation of the Cartan relation holds:

\begin{thm}[Quantum Cartan relation]
\label{thm:quancar}
	$$Q\mathcal{S}(x*y) = Q\mathcal{S}(x)*Q\mathcal{S}(y) + \sum_{i,j} q_{i,j}(W_0 \times D^{i-2,+})(x,y)h^{i}$$ where the correction term is written in terms of linear homomorphisms $$q_{i,j}: H_*^{\mathbb{Z}/2}(\overline{M}_{0,5}) \otimes QH^*(M) \otimes QH^*(M) \rightarrow QH^*(M),$$ such that $q_{i,j}(W_0 \times D^{i-2,+})$ is nonzero only if $i \ge 2$ and $j > 0$. The $q_{i,j}$ will be defined precisely in Definition \ref{defn:qopn}.
\end{thm}

In the correction term, $$W_0 \times D^{i,+} \subset \overline{M}_{0,5} \times_{\mathbb{Z}/2} S^{\infty},$$ where $W_0 \subset \overline{M}_{0,5} \simeq Bl_{\{(0,0),(1,1),(\infty,\infty) \} } (\mathbb{CP}^1 \times \mathbb{CP}^1)$ is the exceptional divisor over $(0,0)$ (compare to Figure \ref{fig:w0}). The notation $D^{i,+}$ means the upper $i$-dimensional hemisphere in $S^i \subset S^{\infty}$. In fact, we are abusing notation as we are really interested in the homology class represented by $W_0 \times D^{i,+}$ in $H_*(\overline{M}_{0,5} \times_{\mathbb{Z}/2} S^{\infty})$, where we are using the singular homology.

In Section \ref{sec:computingqsstoric} we use Theorem \ref{thm:quancar} to calculate the quantum Steenrod squares for a Fano toric variety $M$, as proven in Theorem \ref{thm:SqQtoric}. Here, for $\mu \in H_2(M; \mathbb{Z})$ (which is a free $\mathbb{Z}$-module as a Fano toric variety $M$ is simply connected), let $\mu_2$ be the image of $\mu$ under $H_2(M ; \mathbb{Z}) \rightarrow H_2(M ; \mathbb{Z}/2)$. Denote by $x *_{ \mu, k} y$ the coefficient of $t^{kN}$ in the quantum product $x * y$, using spheres representing $\mu$. Let $N$ be the minimal Chern number, and $|t|=2$.

		\begin{thm}
		\label{thm:SqQtoric}
			Let $M$ be a Fano toric manifold. For $b, x \in H^*(M)$ and $|x| = 2$,
\begin{equation} \label{equation:SqQtoric}  q_{i,j}(W_0 \times D^{i,+})(b,x) =  \sum_{j \ge 1} \sum_{k=1}^{j}  \sum_{c_1(\mu) = 2kN} n(x,\mu_2) \cdot \left( Q\mathcal{S}_{|b|-i+2,j-k}(b) *_{\mu, k} x \right) \cdot t^{jN} \end{equation}
		summing over a basis of $\mu \in H_2(M; \mathbb{Z})$, so $c_1(\mu) = 2kN \in \mathbb{Z}$ and if $\chi$ is some pseudocycle representative of $x$ then $n(x,\mu_2) := \# (\chi \bullet \mu_2) \in \mathbb{Z}/2$. 

	\end{thm}

	For example, if $M = \mathbb{CP}^n$ then setting $b = x^i$ for the generator $x \in H^2(\mathbb{CP}^n)$, we will show in Lemma \ref{lem:qWcpn} that:

$$ q_{4i+2-2n,1}(W_0 \times D^{4i-2n,+})(x^{i},x) = {{i} \choose {n-i}} T, \text{ else } q_{i,j} = 0.$$ Hence

	\begin{equation} \label{equation:qsforcpn} Q\mathcal{S}(x^{i}) = \sum_{j=0}^{i} \left( {{i} \choose {j}}+  \sum_{k=0}^{\lfloor n/2 \rfloor + 1} {{n-k}\choose{k}}\cdot {{i-(n+1-k)} \choose {j-k}} \right) x^{i+j} h^{2(i-j)}, \end{equation} where $x^p$ denotes the $p$-th quantum power of $x$. In particular, if $i+j \ge n+1$ in \eqref{equation:qsforcpn} then $x^{i+j}$ refers to $x^{i+j - n - 1} T$. Omitting the inner summation would give the classical Steenrod square.

	\begin{corollary}
\label{corollary:fanotoricdecided}
		Let $M$ be a Fano toric manifold. Then if we can compute $QH^*(M)$ (over the Novikov ring as in \cite[Section 9.2]{jhols}) then we can compute $Q\mathcal{S}$ through recursive calculations.
	\end{corollary}

In Section \ref{sec:QAR} we extend the Adem relations from Equation \eqref{equation:ademrel} to the quantum Steenrod square. In order to state the quantum Adem relation, we next introduce operations $Q\mathcal{S}^{a,b}: QH^*(M) \rightarrow QH^*(M)$, such that the sum of these $Q\mathcal{S}^{a,b}$ is the total quantum Steenrod square $Q \mathcal{S}$. The index $a$ is the change in homological degree, and the index $b$ is the change in the index of $T$. 

\begin{defn}
 Define $Q\mathcal{S}^{a,b}$ by $$Q\mathcal{S} (x T^i) = \sum_{a,b \in \mathbb{Z}} Q\mathcal{S}^{a,b} (x T^i) \cdot h^{|x| - 2N(b+i) - a} \quad \textrm{where} \quad Q\mathcal{S}^{a,b}(xT^i) \in T^{b+i} H^{|x|+a}(M),$$ for any $x \in H^*(M)$.
\end{defn}

Computing $Q\mathcal{S}^{a,b}$ for $\mathbb{CP}^2$, we find that the naive generalisation of the Adem relation (Equation \eqref{equation:ademrel}), namely  $$\sum_{b,d} \left( Q\mathcal{S}^{q-2bN,b} \circ Q\mathcal{S}^{p-2dN,d}(\alpha) - \sum_{s=0}^{q/2} {{p-s-1}\choose{q-2s}} Q\mathcal{S}^{p+q-s-2bN,b} \circ Q\mathcal{S}^{s-2dN,d}(\alpha) \right) = 0,$$ does not hold, as in the example below.

\begin{exmpl}

Let $M = \mathbb{CP}^2$, so $2N = 6$. Then $Q\mathcal{S}^{2-2N,1} \circ Q\mathcal{S}^{2-0N,0}(x) = T$, but $$\sum_{s=0}^{s = 1} {{1-s}\choose{2-2s}} Q\mathcal{S}^{2+2-s-2iN,i} \circ Q\mathcal{S}^{s-2jN,j}(x) = 0$$ for all $i,j$. 
\end{exmpl}

In order to prove the quantum Adem relation, we begin with the technical Theorem \ref{thm:QAR}. The terminology used in Theorem \ref{thm:QAR} will be fully defined in Section \ref{subsec:QAR}. 

\begin{thm}
\label{thm:QAR}
			For $M$ a closed monotone symplectic manifold, with $\alpha \in QH^*(M)$, and for $p,q>0$ such that $q<2p$:

			$$qq_{|\alpha|+p-q,|\alpha|-p}(\alpha) = \sum_{s=0}^{q/2} {{p-s-1}\choose{q-2s}}  qq_{|\alpha|+2s-p-q,|\alpha|-s}(\alpha).$$
			\end{thm}

The homomorphism $$qq: H^*(M) \rightarrow QH^*(M) \otimes H^*(BD_8),$$ where $D_8$ is the dihedral group. This $qq$ operation will include the data of the composition $Q \mathcal{S} \circ Q \mathcal{S}$. The ring $H^*(BD_8)$ has three generators, labelled $e, \sigma_1, \sigma_2$ (of which we only need to consider $e$ and $\sigma_2$). We then denote by $qq_{i,j}(\alpha)$ the coefficient of $e^i \sigma_2^j$ in $qq(\alpha)$, defined in Equation \eqref{equation:qqijdef}. 

This should be compared to equation \eqref{equation:ademrel}. The above theorem leads to a Corollary in more familiar terms:

\begin{corollary}[Quantum Adem Relations]
			\label{corollary:QAR}
			For $p,q > 0$ such that $q < 2p$, and $\alpha \in QH^*(M)$,

\begin{equation} \label{equation:QAR} \sum_{b,d} \left( Q\mathcal{S}^{q,b} \circ Q\mathcal{S}^{p,d}(\alpha) - \sum_{s=0}^{q/2} {{p-s-1}\choose{q-2s}} Q\mathcal{S}^{p+q-s,b} \circ Q\mathcal{S}^{s,d}(\alpha) \right) = T \cdot Q(\alpha) \end{equation} for the correction term 

$$
\begin{array}{rcl}
T \cdot Q(\alpha) &=&   q_{D_8}((g m_1 + g^2 m_1) \otimes \Psi (e^{|\alpha| + p - q} \sigma_2^{|\alpha|-p}))(\alpha) 
\\[0.5em] &&- \sum_{s=0}^{[q/2]} {{p-s-1}\choose{q-2s}} q_{D_{8}}((g m_1 + g^2 m_1) \times \Psi (e^{|\alpha| +2s - p - q} \sigma_2^{|\alpha|-s}))(\alpha).
\end{array}
$$
		\end{corollary}

In the above corollary, the dihedral group $D_8 = \langle (12), (13)(24) \rangle \subset S_4$ acts on the four incoming marked points $(z_1,z_2,z_3,z_4)$ by permutations. The operation $q_{D_8}$ is a linear homomorphism determined by homology classes in $\overline{M}_{0,5} \times_{D_8} ED_8$, so $q_{D_8}(A) : H^*(M) \rightarrow H^{4*-|A|}(M)$ for $A \in H_*(\overline{M}_{0,5} \times_{D_8} ED_8)$. It is analogous to the $q_{i,j}$ in Theorem \ref{thm:quancar}. Here $m_1 \in \overline{M}_{0,5}$ (see Figure \ref{fig:m05elmts}), $e^i \sigma_2^j \in H^*(BD_8)$ and $\Psi: H^*(\overline{M}_{0,5} \times_{D_8} ED_8) \rightarrow H_*(\overline{M}_{0,5} \times_{D_8} ED_8)$ is the universal coefficients isomorphism. Let $g=(123) \in S_4$, such that the cosets of $D_8$ in $S_4$ are $D_8, gD_8, g^2 D_8$.

In Section \ref{sec:blowups}, we calculate $Q \mathcal{S}$ in the case of the blowups $M= Bl_Y(\mathbb{CP}^3)$ and $M= Bl_Y(\mathbb{CP}^1 \times \mathbb{CP}^1 \times \mathbb{CP}^1)$ where $Y$ is respectively the intersection of two quadrics and the intersection of two linear hypersurfaces. The setup here is similar to Blaier \cite{blaier}. Most of the squares can be computed using the methods from Section \ref{sec:computingqsstoric}. The new computation is of $Q \mathcal{S}_{1,1}$, which is given in the following theorem.

\begin{thm}
\label{thm:blowupID}
	
\begin{equation} \label{equation:blowupID1}  \qquad Q\mathcal{S}_{1,1} = id : H^3(Bl_Y(\mathbb{CP}^3)) \rightarrow H^3(Bl_Y(\mathbb{CP}^3))  \end{equation}	

and

\begin{equation} \label{equation:blowupID2} \qquad Q\mathcal{S}_{1,1} = id : H^3(Bl_Y(\mathbb{CP}^1 \times \mathbb{CP}^1 \times \mathbb{CP}^1)) \rightarrow H^3(Bl_Y(\mathbb{CP}^1 \times \mathbb{CP}^1 \times \mathbb{CP}^1)). \end{equation}

\end{thm}

Observe that $Q\mathcal{S}_{1,1}$ are quantum correction terms to the classical Steenrod square on the blowup $M$. They are determined by lifts of contributions to the classical Steenrod square on $Y$.

\subsection*{Acknowledgements}
I thank my supervisor Alexander Ritter for his guidance and support throughout my Ph.D. and beyond. I thank Paul Seidel for suggesting this project, and for helpful conversations. I thank the anonymous referee for all of the comments and points of clarification, which have been crucial both to correct errors and to improve the exposition. I also thank Dominic Joyce and Ivan Smith, Frances Kirwan and Ulrike Tillmann for helpful comments and corrections on my thesis, and therefore this paper. This work was supported by an EPSRC grant, reference: EP/M508111/1, and partially supported by the Simons Foundation, through a Simons Investigator grant (PI: Paul Seidel).

This forms a chapter of my Ph.D. thesis.

\section{Preliminaries}
\label{sec:preliminaries}
	Henceforth we always work with coefficients in $\mathbb{Z}/2$, unless otherwise stated. For example $H^{*}(M)$ means $H^{*}(M;\mathbb{Z}/2)$.

	\subsection{Equivariant Cohomology}
	\label{subsec:equivcohom}
		We follow \cite[Section 2]{seidel}.

		\begin{defn}[Equivariant cohomology of a chain complex]
		\label{defn:equcohom}
		Let $(C^{\bullet},d)$ be a cochain complex over $\mathbb{Z}/2$. Suppose $(C^{\bullet},d)$ has a chain involution $\iota$, so $\iota: C^{\bullet} \rightarrow C^{\bullet}$ is a chain map with $\iota^2 = id_{C^{\bullet}}$. Let $h$ be a formal variable in grading $1$. The equivariant chain complex is $$(C^{\bullet}_{\mathbb{Z}/2} ,\delta) =  (C^{\bullet}[h],d + h(id_{C} + \iota)).$$ 

		Define $H^{*}_{\mathbb{Z}/2}(C) := H^*(C^{\bullet}_{\mathbb{Z}/2} ,\delta) $, the equivariant cohomology of $(C,d,\iota)$. 
		\end{defn}

		\begin{defn}[Equivariant Cohomology of a manifold]
		\label{defn:equivcohomman}
		Let $N$ be a topological space with a continuous involution $\iota: N \rightarrow N$. Let $C = C^{*}(N)$ be the singular cochain complex of the topological space $N$. There is a $\mathbb{Z}/2$ action on $C^* (N)$ induced by $\iota$. As in Definition \ref{defn:equcohom}, the \textit{equivariant cohomology of $N$} is $H^* _{\mathbb{Z}/2}(N) := H^* _{\mathbb{Z}/2}(C^* (N))$. 
		\end{defn}

		The important examples of this will be $M \times M$ with the involution swapping the factors and $M$ with the trivial involution. We will respectively denote the equivariant chains in this case by $C^{\bullet}_{\mathbb{Z}/2}(M \times M)$ (or in that case of Morse cohomology by $CM^{\bullet}_{\mathbb{Z}/2}(M \times M)$) and by $C^{\bullet}_{\mathbb{Z}/2}(M)$. Similarly, equivariant cohomology will be denoted respectively by $H^{\bullet}_{\mathbb{Z}/2}(M \times M)$ and by $H^{\bullet}_{\mathbb{Z}/2}(M)$

		\begin{rmk} 
			There is another description of $H^*_{\mathbb{Z}/2}(N)$ for a manifold $N$ with a continuous involution. Recall that $E \mathbb{Z}/2$ is the classifying space of $ \mathbb{Z}/2$: a contractible space with a free $\mathbb{Z}/2$ action, for example $E \mathbb{Z}/2 = S^{\infty}$ with the involution being the antipodal map. Then $$H^{*}_{ \mathbb{Z}/2}(N) := H^{*}(N \times_{\mathbb{Z}/2} E  \mathbb{Z}/2).$$

			This definition is equivalent to Definition \ref{defn:equivcohomman}. If we let $N = \{ pt \}$ then we obtain $pt \times_{\mathbb{Z}/2} S^{\infty} = S^{\infty} / (\mathbb{Z}/2) = \mathbb{RP}^{\infty}$, hence $$H^{*}_{\mathbb{Z}/2}(pt) = H^{*}(\mathbb{RP}^{\infty}) = \mathbb{Z}/2[h].$$
		\end{rmk}

	\subsection{The Steenrod Squares}
	\label{subsec:theSqs}
		For a reference, see \cite[Section 4.L]{algtop}.

		The Steenrod square operations $\{ Sq^{i} \}$ are the unique collection of additive homomorphisms such that:
		\begin{enumerate}
			\item $Sq^{i} : H^{n}(M) \rightarrow H^{n+i}(M)$ for each $n \ge 0$ and topological space $M$,
			\item Each $Sq^{i}$ is natural in $M$,
			\item $Sq^{0}$ is the identity,
			\item $Sq^{n}$ acts as the cup square on $H^{n}$, so $Sq^{|x|}(x) = x \cup x$,
			\item If $n > |x|$ or $n < 0$ then $Sq^{n}(x) = 0$,
			\item (Cartan relation) For each $n$, $$Sq^{n}(x \cup y) = \sum_{i+j=n} Sq^{i}(x) \cup Sq^{j}(y).$$ 
		\end{enumerate}
		
		Here $|x|$ is the cohomological grading of $x \in H^*(M)$. Recall that we use $\mathbb{Z}/2$ coefficients to ensure additivity: $(x+y) \cup (x+y) = x \cup x + y \cup y$ modulo 2. These $Sq^{i}$ together define a single operator, the ``total Steenrod square" $$Sq:H^{*}(M) \rightarrow (H^{\bullet}(M)[h])^{2*},$$ where $Sq^{i}$ is the coefficient of $h^{n-i}$, so $Sq(x) = \sum_{i} Sq^{|x|-i}(x) \cdot h^i$. The cup product on $H^{n}(M)[h]$ is $(a \cdot h^i) \cup (b \cdot h^j) = (a \cup b) \cdot  h^{i+j}$, so the Cartan relation becomes $Sq(x \cup y) = Sq(x) \cup Sq(y)$ and thus $Sq$ is a unital ring homomorphism. We will henceforth call $Sq$ the ``Steenrod square" when there is no ambiguity, noting that it contains the same information as $\{ Sq^i \}$.

		One must note that although these axioms imply that there is a unique Steenrod square, there are many different approaches to constructing them.

		\begin{exmpl}[The classical Steenrod square for $\mathbb{CP}^{n}$]
		\label{exmpl:classcpn}
			$$H^{*}(\mathbb{CP}^{n}) \cong \mathbb{Z}/2 [x] / (x^{n+1})$$ where $|x| = 2$. We see that $Sq^{0}(x) = x$ and $Sq^{2}(x) = x^{2}$ using axioms $3$ and $4$, and these are all of the nonzero terms by axiom $5$. Hence $Sq(x) = xh^{2} + x^{2}$. By the Cartan relation (axiom 6), $$Sq(x^{i}) = Sq(x)^{i} = (xh^{2} + x^{2})^{i} = x^{i} \sum_{j=0}^{i} {{i}\choose{j}} x^{j} h^{2(i-j)}.$$ Looking at the coefficient of $h^{2i-k}$, $Sq^k(x) = 0$ for $k$ odd and $Sq^{2j}(x^i) = {{i}\choose{j}} x^{i+j}$.
		\end{exmpl}
		
\subsection{The Betz-Cohen Construction}
\label{subsec:prelimbcncon}
The details are relevant for Section \ref{subsec:tmssissq}. 

Fix a Morse-Smale function $f$ on $M$, and pick a small convex neighbourhood $U_f$ of $f$ in $C^{\infty}(M)$ consisting of Morse-Smale functions. Let $\Gamma$ be the $Y$-shaped graph, oriented and parametrised as $(-\infty,0] \vee_0 [0,\infty) \vee_0 [0,\infty)$. We denote $S$ to be the set of triples $\sigma = (f_{1,s},f_{2,s},f_{3,s})$ such that $f_{1,s} \in U_f$ for each $s \in (-\infty,0]$ and $f_{2,s}, f_{3,s} \in U_f$ for $s \in [0,\infty)$, subject to:
\begin{enumerate}
\item $f_{1,0}, f_{2,0}, f_{3,0}$ are pairwise distinct. 
\item $f_{i,s} = \beta(|s|) f_{i,0} + (1-\beta(|s|)) f$, where $\beta: [0,\infty) \rightarrow [0,1]$ is a fixed monotone bump function such that $\beta(s) =1$ for $s \le 1/2$ and $\beta(s) = 0$ for $s \ge 1$.
\end{enumerate}
We define $\mathcal{M}_{\sigma}$ to be the set of continuous maps $\gamma: \Gamma \rightarrow M$ that are smooth on the edges, such that for each edge $E_i$ of $\Gamma$ we denote $\gamma_i = \gamma|_{E_i}$, and require $$d \gamma_i / dt (s) + \nabla f_{i,s}(\gamma_i(s)) = 0.$$ This is actually slightly different to the construction in \cite{betzcoh}, in which the $f_1,f_2,f_3$ were pairwise distinct and had no $s$ dependence. The construction due to Betz-Cohen is equivalent to that given here, as we are simply using a deformation retraction of their moduli space of metric Morse flows.

Let $\mathcal{M}_{BC}= \sqcup_{\sigma \in S} \mathcal{M}_{\sigma}$, topologised so that $\mathcal{M}_{BC} \rightarrow S$ is continuous. Observe that there is a $\mathbb{Z}/2$-action $\iota_S$ on $S$, induced by the permutation $(23)$. This induces a $\mathbb{Z}/2$-action on $\mathcal{M}_{BC}$, via $(\sigma, \gamma) \mapsto (\iota_S \circ  \sigma, \gamma \circ R_{\Gamma})$. Here, $R_{\Gamma}$ is the involution on $\Gamma$ that swaps the two positive half-lines and fixes the negative half-line.

For $a_1,a_2, a_3 \in \text{crit}(f)$, define $\mathcal{M}_{BC}(a_1,a_2,a_3)$ to consist of equivalence classes of pairs $[\sigma, \gamma] \in \mathcal{M}_{BC}/(\mathbb{Z}/2)$ such that $${\displaystyle \lim_{i \rightarrow -\infty}} \gamma_1(s) = a_1, \ {\displaystyle \lim_{i \rightarrow \infty}} \gamma_2(s) = a_2, \ {\displaystyle \lim_{i \rightarrow \infty}} \gamma_3(s) = a_3.$$

The space $S$ is contractible and has a free $\mathbb{Z}/2$-action, so $SB := S /( \mathbb{Z}/2)$ is homotopy equivalent to $\mathbb{RP}^{\infty}$. Thus, there are representatives $\delta_i$ of the nontrivial generator of $H_i(SB) \cong \mathbb{Z}/2$ for each $i$. Strictly, we consider some $\delta_i = \sum_{j} \tau_{i,j}$ where $\tau_{i,j}: \Delta_j \rightarrow S$ is a simplex. For each $i \ge 0$, let $$\mathcal{M}_{BC,i}(a_1,a_2,a_3) = {\displaystyle \bigcup_j} \tau_{i,j}^* \mathcal{M}_{BC}(a_1,a_2,a_3),$$ the union of the pullback of $\mathcal{M}_{BC}(a_1,a_2,a_3)$ along $\tau_{i,j}: \Delta_i \rightarrow S$, glued along faces.

Recal that in Morse theory, $CM^*(M \times M, f \oplus f)$ and $CM^*(M,f) \otimes CM^*(M,f)$ are identified via the K\"unneth isomorphism, where $$f \oplus f: M \times M \rightarrow \mathbb{R}, \quad (f \oplus f)(x,y) = f(x) + f(y).$$ One uses the correspondence between critical points of $f \oplus f$ and formal pairs of critical points of $f$, denoted $a \otimes b$ for $a,b \in \text{crit}(f)$. The isomorphism between Morse and singular cohomology respects the involution that swaps the factors, and hence we may replace the equivariant cohomology of $C^*(M) \otimes C^*(M)$, denoted $H^*_{\mathbb{Z}/2}(M \times M)$, with the equivariant cohomology of $CM^*(M,f) \otimes CM^*(M,f)$. One can think of this in terms of equivariant Morse cohomology, as detailed in \cite[Section 2]{seidelsmith}, where for the given $\mathbb{Z}/2$-action all of the necessary transversality conditions are satisfied.

Using the fact that the equivariant chains are $C^*_{\mathbb{Z}/2}(M \times M) = C^*(M \times M)[h]$, as well as using the previous paragraph, elements of $C^*_{\mathbb{Z}/2}(M \times M)$ may be written as a finite sum of $(b \otimes c)h^j$ for some $j \ge 0$ and $b,c \in \text{crit}(f)$. Let $a \in \text{crit}(f)$ and $\delta_i$ be the generator of $H_i(SB) \cong \mathbb{Z}/2$ (because $S$ is an $E \mathbb{Z}/2$). 

We define $$q: H_i(SB) \otimes H^j_{\mathbb{Z}/2}(M \times M) \rightarrow H^{j-i}(M),$$ at the chain level, such that the coefficient of $a$ in $q(\delta_i \otimes (b \otimes c)h^k)$ is $\# \mathcal{M}_{BC,i-k}(a,b,c)$, when $\mathcal{M}_{BC,i-k}(a,b,c)$ is a collection of points. Let $q_i(b \otimes c) = q(\delta_i \otimes b \otimes c)$. One defines $$Sq^{|x|-i}(x) = q_i (x \otimes x).$$

	\subsection{The Quantum Cup Product}
	\label{subsec:quantcupprod}
		For more details on the quantum cup product, see \cite[Chapter 8]{jhols}. Throughout this paper, for $M$ a closed $n$-manifold, we denote by $PD: H^*(M) \rightarrow H_{n-*}(M)$ and $PD: H_*(M) \rightarrow H^{n-*}(M)$ the Poincar\'e duality operation over $\mathbb{Z}/2$ coefficients. 

		Let $(M, \omega)$ be a monotone symplectic manifold of dimension $n$, with a fixed almost complex structure $J$ compatible with $\omega$.

		\begin{defn}
			A symplectic manifold $(M,\omega)$ is \textit{monotone} if the restriction to spherical homology classes of the cohomology class of $\omega$ is positively proportional to the first Chern class of $TM$. In other words, there exists a constant $\lambda > 0$ such that $$[\omega]|_{\pi_2(M)} = \lambda \cdot c_1(TM)|_{\pi_2(M)}.$$
		\end{defn}

		 As an abelian group, $QH^{*}(M) = H^{*}(M)[[t]]$ where $t$ is a formal variable of degree $2$. Let $T = t^N$, where $N \geq 0$ is the minimal Chern number of $M$, determined by $c_{1}(\pi_{2}(M)) = N \mathbb{Z}$. By rescaling our symplectic form if necessary, we will assume that $\lambda = 1/N$, and so referring to a {\it $J$-holomorphic map $u$ of energy $k$} means that $c_1(u_*[S^2]) = N \cdot [\omega](u)= kN$.

		As an important note, we define the quantum cochains $$QC^*(M) := C^*(M) \otimes_{\mathbb{Z}/2} \mathbb{Z}/2[[T]].$$ Then $QH^*(M) = H^*(QC^*(M), d \otimes id)$, where $d$ is the differential on $C^*(M)$. Most of the operations that we consider are defined at the chain level, and then descend to maps on (co)homology.

		We pick a basis $\mathcal{B}$ for $H^*(M)$ and a dual basis with respect to the nondegenerate cup product pairing $(e,f) \mapsto \langle e \cup f, [M] \rangle$. There is a dual basis $\mathcal{B}^{\vee}$ with respect to this pairing. Let $\alpha^{\vee} \in H^{n-|\alpha|}(M)$ denote the dual of the cohomology class $\alpha \in H^{|\alpha|}(M)$. Our operations on cohomology will not depend on this choice of basis, although they may affect the chain level description.

		Given $A \in H_{2}(M)$, let $\mathcal{M}_{A}(J)$ be the moduli space of $J$-holomorphic spheres $u: S^2 \rightarrow M$ such that $u_{*}([S^2]) = A$, up to reparametrisation by $PSL(2,\mathbb{C})$. For a generic choice of $J$, this moduli space is a smooth manifold with $$\dim \mathcal{M}_{A}(J) = 2c_{1}(A) + \dim(M).$$ For each $z \in S^{2}$, there is an evaluation map $ev_{A,z} : \mathcal{M}_{A}(J) \rightarrow M$ with $ev_{A,z}(u) =  u(z)$. Pick three distinct points, $z_{1}, z_{2},z_{3} \in S^2$. We use $0,1,\infty$ throughout, and denote $ev_{A} = ev_{A,0} \times ev_{A,1} \times ev_{A,\infty} : \mathcal{M}_{A}(J) \rightarrow M \times M \times M$.

		\begin{defn}[Quantum Product]
\label{defn:quantumproduct} 
			Let $\alpha, \beta \in H^{*}(M) \subset QH^{*}(M)$. Pick generic pseudocycle representatives $a, b$ of the classes $PD(\alpha)$ and $PD(\beta)$ (so that they are transverse to the evaluation maps in the previous paragraph). Similarly, for each $\gamma \in \mathcal{B}$, we pick a representative $c^{\vee}$ of $PD(\gamma^{\vee})$. Denote by $a \times b \times c^{\vee}$ the product of these cycles, landing in $M \times M \times M$.

			Then we define $$\alpha * \beta = \sum_{j \in \mathbb{Z}, \ \gamma \in \mathcal{B} :  |\gamma| = |\beta| + |\alpha| - 2jN} n(\gamma, \alpha, \beta, j) \cdot \gamma \cdot T^j,$$  $$n(\gamma, \alpha, \beta, j) = \sum_{A \in H_2 (M) : c_{1}(A) = jN} ev_A \bullet (a \times b \times c^{\vee}).$$ Here $\bullet$ is the intersection number of pseudocycles of complementary dimension. Extending $\mathbb{Z}/2[[t]]$-linearly defines $*$ on $QH^*(M)$.
		\end{defn}

Observe that for generic $J$, the evaluation map is a pseudocycle, \cite{jhols}. In order to show that this is well defined, one must prove that the outcome is independent of choice of pseudocycle representatives that we choose, such as in \cite[Lemma 7.1.4]{jhols}. The degree condition ensures that the pseudocycles are of complementary dimension. Notice that $|a*b| = |a|+|b|$, using that $|T|=2N$. If $A=0$ (so $E(u)=0$ and $u$ is constant), this recovers the classical intersection product. 

\begin{rmk}
		In concrete terms, in Definition \ref{defn:quantumproduct} we count the number of $J$-holomorphic spheres in $M$ intersecting some choice of pseudocycle representatives of the $PD(\alpha)$, $PD(\beta)$ and $PD(\gamma^{\vee})$. This can be thought of as the intersection $$ev_{A,0}^{-1}(a) \cap ev_{A,1}^{-1}(b) \cap ev_{A,\infty}^{-1}(c^{\vee})$$ in the space of $J$-holomorphic stable maps representing $A$.
\end{rmk}
		
\section{Two constructions of the Steenrod Squares}
\label{sec:morssteensqu}
	The first construction will use Morse theory, and will be based on that given in \cite{seidel}, \cite{fukaya} and \cite{betzcoh}. The second construction is a generalisation of the first, involving pseudocycles. In this section $\Gamma$ is the Y-shaped graph with incoming edge $e_1$ and outgoing edges $e_2$ and $e_3$. Let $e_1$ be parametrised by $(-\infty, 0]$ and $e_2,e_3$ by $[0,\infty)$. This is illustrated in Figure \ref{fig:classmorsqu}.

	Throughout this section, $M$ will be a smooth closed manifold. We recall the Morse theoretic cup product: given a Morse function $f$, pick three generic perturbations $f^1_{s}$ for $s \in (- \infty,0]$ and $f^2_{s}$, $f^3_{s}$ for $s \in [0,\infty)$ (so that they are ``transverse at $0$"). Making a generic choice ensures that the moduli space in Definition \ref{defn:morsecupprod} is cut out transversely: specifically, the genericity condition ensures that the moduli space is a smooth manifold. This is discussed in \cite[Chapter 5.2, Chapter 2]{schwarz} by Schwarz. It should be pointed out that the construction of the Morse theoretic cup product in the cited work uses three distinct fixed Morse functions $f^1,f^2,f^3$, rather than using perturbations $f^1_s,f^2_s,f^3_s$. Combining the case of Schwarz with the standard notion of continuation maps from $f^i$ to our fixed Morse function $f$, and applying a gluing argument, means that we can instead consider $s$-dependent functions on the edges. After applying such a gluing of continuation maps, the requirement that $f^1,f^2,f^3$ be chosen generically translates to requiring that $f^1_0,f^2_0, f^3_0$ be chosen generically, which is what we meant by ``transverse at $0$" above. This idea is made precise in \cite[Section 2]{morsetrajectories}. With this in mind, we choose the $f^i_s$ such that there is an $R > 0$ with $f^i_{s} = f$ if $|s| \ge R$, so that we can apply Morse theoretic arguments outside of a compact neighbourhood of the vertex in $\Gamma$. Denote by $\text{crit}_k(f)$ the critical points of $f$ of Morse index $k$. Write $|x|$ for the Morse index of $x \in \text{crit}(f)$.

	\begin{defn}[Morse cup product]
	\label{defn:morsecupprod}
	Let $a_{2},a_{3}$ be critical points of $f$, with respective Morse indices $|a_{2}|, |a_{3}|$ and let $k = |a_{2}|+|a_{3}|$, then $$a_{2} \cdot a_{3} := \sum_{a_{1} \in \text{crit}_k(f)} n_{a_1, a_2, a_3} a_1$$ where $n_{a_1, a_2, a_3}$ is the number of elements in the $0-$dimensional moduli space $\mathcal{M}(f^i_{s},a_{1},a_{2},a_{3})$ of continuous maps $u: \Gamma \rightarrow M$, smooth on the edges, such that: 

	\begin{enumerate}
		\item $d (u|_{e_i})/ds = -\nabla f^i_{s}$,
		\item $u|_{e_1}(x) \rightarrow a_1$ as $x \rightarrow -\infty$,
		\item $u|_{e_i}(x) \rightarrow a_i$ as $x \rightarrow \infty$ for $i=2,3$.
	\end{enumerate}

	\end{defn}

	\subsection{Morse Steenrod square}
	\label{subsec:msss}

Henceforth, we will consider a nested sequence of spheres $S^{0} \subset S^{1} \subset ... \subset S^{\infty}$, consisting of equators that exhaust $S^{\infty}$ and are preserved under the involution $v \mapsto -v$. Denote $$S^{\infty} = \{ (x_0,x_1,x_2, \ldots) \subset \bigoplus_{i \ge 0} \mathbb{R}^{i} : \sum_i x_i^2 = 1 \},$$ the subset $S^i \subset S^{\infty}$ consists of those elements of $S^{\infty}$ of the form $(x_0, \ldots, x_i, 0 ,\ldots)$.

	We refine the choice of $f^i_{s}$ by picking a collection of smooth functions $f^i_{v,s}: M \rightarrow \mathbb{R}$, smoothly parametrised by $v \in S^{\infty}$ and $s \in (-\infty,0]$ for $i=1$, respectively $s \in [0,\infty)$ for $i=2,3$, satisfying the following conditions:
	\begin{enumerate}
		\item $f^2_{v,s} = f^3_{-v,s}$,
		\item For each $i$, the smooth map $f^2_{\cdot,0}: S^{i} \times M \rightarrow \mathbb{R}$ must be chosen generically, with more details provided in Appendix \ref{subsec:mssrmks}. 
		\item There is an $R>0$ such that $f^i_{v,s} = f$ for all $|s| \ge R$ and $v \in S^{\infty}$.
		\item $f^1_{v,s} = f^1_{-v,s}$ .
	\end{enumerate}

	Given $a_1, a_2, a_3 \in \text{crit}(f)$, and $v \in S^{\infty}$, we define $\mathcal{M}'_{v}(a_1, a_{2}, a_{3})$ to be the set of pairs $(u: \Gamma \rightarrow M, v)$ such that:
	\begin{enumerate}
		\item $d(u|_{e_i})/ds = -\nabla f^i_{v,s}$.
		\item $u|_{e_1} (s) \rightarrow a_1$ as $s \rightarrow -\infty$ and $u|_{e_i} (s) \rightarrow a_{i}$ for $i=2,3$ as $s \rightarrow \infty$. 
	\end{enumerate}

Let $$\mathcal{M}'_{i}(a_1, a_{2},a_{3}) = \bigsqcup_{v \in S^{i}} \mathcal{M}'_{v}(a_1,a_{2}, a_3),$$ topologised as a subset of $C(\Gamma, M) \times S^i$ (where $C(\Gamma, M)$ is the space of continuous maps from $\Gamma$ to $M$ that are smooth on the edges). The projection to $S^i$ is then continuous for all $i$. Indeed, $\mathcal{M}'_{i}(a_1, a_{2},a_{3})$ is a smooth manifold for each $i$, by the genericity conditions as given in Appendix \ref{subsec:mssrmks}.

	Let $r : \Gamma \rightarrow \Gamma$ be the reflection that swaps $e_2$ and $e_3$ (preserving parametrisations) and fixes $e_1$. If $a_2 = a_3$ as in Figure \ref{fig:classmorsqu}, there is a free $\mathbb{Z}/2$ action on the moduli space $\mathcal{M}'_{i}(a_1,a_2, a_2)$, via $$(u,v) \mapsto (u \circ r, -v).$$ 

Let $\mathcal{M}_{i}(a_1,a_2, a_2) = \mathcal{M}'_{i}(a_1, a_2,a_2)/(\mathbb{Z}/2)$, the quotient by the $\mathbb{Z}/2$ action. If $a_2 \neq a_3$, $$\mathcal{M}_{i}(a_1,a_2, a_3) = \bigsqcup_{v \in D^{i,+}} \mathcal{M}'_{v}(a_1,a_{2}, a_3) = \bigsqcup_{v \in D^{i,-}} \mathcal{M}'_{v}(a_1,a_{3}, a_2),$$ where $D^{i,\pm}$ is the upper/lower $i-$dimensional hemisphere in $S^i \subset S^{\infty}$. Observe that when $v \in \partial D^{i,\pm}$, there is no overcounting of solutions (when $a_2 \neq a_3$). This is because a solution for $v \in \partial D^{i,+}$, with asymptotics $a_1, a_2, a_3$, does not correspond to a solution for $-v \in \partial D^{i,+}$ with asymptotics $a_1, a_2, a_3$: the action $u \mapsto u \circ r$ swaps the $a_2$ and $a_3$ asymptotics. Indeed, when $a_2 \neq a_3$ the number of solutions for $v \in \partial D^{i,\pm}$ exactly corresponds to the $Sq'((a_2 \otimes a_3 + a_3 \otimes a_2)h)$ term in Equation \eqref{equation:Sq'chainmap}.

Consider the natural projection $\mathcal{M}_{i}(a_1, a_2,a_3) \rightarrow \mathbb{RP}^{i}$. Over a generic $v \in \mathbb{RP}^{i}$ there is a smooth manifold of degree $|a_{1}| - |a_{2}| - |a_3|$, so the dimension of the moduli space is $$\text{dim}\mathcal{M}_{i}(a_{1},a_{2}, a_3) = |a_{1}| - |a_{2}| - |a_3|+ i.$$ This is an example of genericity in family Morse theory, as in \cite[Theorem 3.4]{hutchingsfamilies}, and as used in \cite[Equations (4.26), (4.95)]{seidel}.

%We recall from Definition \ref{defn:equcohom} that the equivariant chain complex for $CM^{\bullet}(M,f) \otimes CM^{\bullet}(M,f)$, whose involution swaps the two factors, is $CM^{\bullet}(M,f) \otimes CM^{\bullet}(M,f)[h]$, and it has as its differential $$d_{\mathbb{Z}/2}(a_2 \otimes a_3) = d(a_2 \otimes a_3) + (a_2 \otimes a_3 + a_3 \otimes a_2)h.$$ By inspection, the closed elements are generated by elements of the form $a otimes a h^j$, where $a$ is closed in $C^{\bullet}(M)$, and (a_2 \otimes a_3 + a_3 \otimes a_2)h^j$, where $a_2, a_3$ are closed in $C^{\bullet}(M)$, and $j \ge 0$. Denote the closed elements of 

\begin{figure}
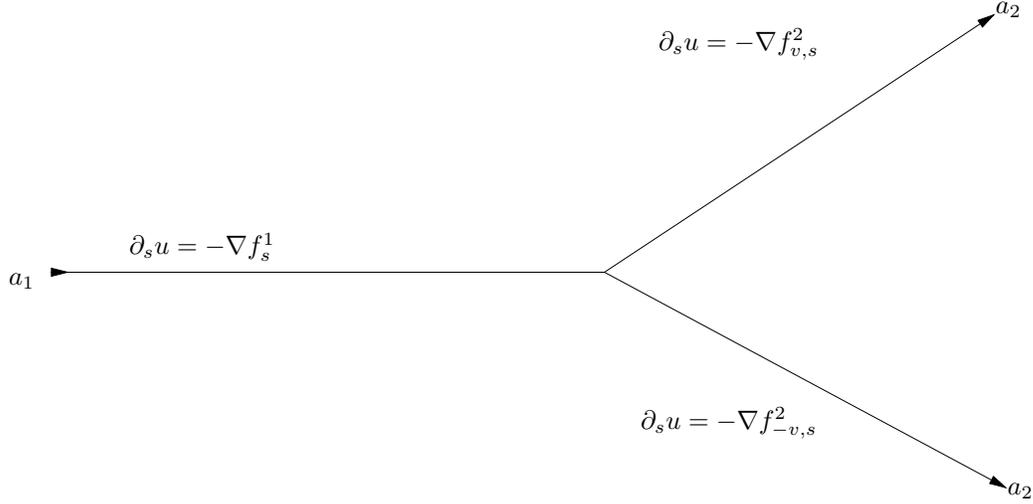
		\caption{Morse flowline configurations for the Steenrod square.}
		\label{fig:classmorsqu}
\end{figure}

Before giving the definition, we recall the notation of Section \ref{subsec:equivcohom}. Specifically, given the chain complex $CM^{\bullet}(M,f)$ with the trivial action of $\mathbb{Z}/2$, one defines the $\mathbb{Z}/2$-equivariant Morse cohomology using the equivariant chain complex $CM_{\mathbb{Z}/2}^{\bullet}(M,f)$. Similarly, given the chain complex $CM^{\bullet}(M,f) \otimes CM^{\bullet}(M,f)$ (which we identify with $CM^{\bullet}(M \times M ,f \oplus f)$ via the K\"unneth isomorphism), there is the action of $\mathbb{Z}/2$ that swaps the two factors, and we denote the $\mathbb{Z}/2$-equivariant chain complex in this case $CM_{\mathbb{Z}/2}^{\bullet}(M \times M)$.  
	
	\begin{defn}[The Morse Steenrod Square]
	\label{defn:mss}
		Let $a_2, a_3 \in \text{crit}(f)$. This determines $a_2 \otimes a_3 \in CM_{\mathbb{Z}/2}^{\bullet}(M \times M)$. Define $$Sq': CM^{\bullet}_{\mathbb{Z}/2}(M \times M) \rightarrow CM_{\mathbb{Z}/2}^{\bullet}(M),$$ by $$Sq'(a_2 \otimes a_3) = \sum_{i =0}^{|a_2| + |a_3|} \sum_{a_1 \in \text{crit}_{|a_2| + |a_3|-i}(f)} n_{a_1, a_2, a_3,i} \cdot a_1 \cdot h^i$$ where $n_{a_1, a_2,a_3,i} = \# \mathcal{M}_{i}(a_1,a_2,a_3)$ for $\#$ the number of points modulo 2. Then extend as a $h$-module.

We then need to prove that $Sq'$ descends to a map on equivariant cohomology. To do this, we use a standard argument involving a $1$-dimensional moduli space (see for example \cite[Section 2.4, Section 5.3]{schwarz}, applied as in \cite[Proposition 1.9, Lemma 1.10]{fukaya}). We then consider its compactification, as covered in detail in Appendix \ref{sec:equivariantcompact}. This in turn shows that $Sq'$ is a chain map, i.e.:  \begin{equation} \label{equation:Sq'chainmap} Sq'( (a_2 \otimes a_3 + a_3 \otimes a_2) h + (d a_2) \otimes a_3 + a_2 \otimes (d a_3)) = d Sq'(a_2 \otimes a_3). \end{equation} 

Further, post-composing with the doubling operation $$\text{double}: CM^{*}(M) \rightarrow CM^{2*}_{\mathbb{Z}/2}(M \times M), \ a \mapsto a \otimes a,$$ which also descends to a map on equivariant cohomology, we define $$Sq := [Sq'] \circ [\text{double}].$$ Here $[-]$ denotes the cohomology level operation of the respective map of chains. This definition is independent of the choice of parametrised Morse functions by a standard continuation argument, such as in \cite[Section 3.4]{salamonfloer}.

		The coefficient of $h^{|a|-i}$ is denoted by $Sq^{i}(a) \in H^{|a|+i}(M)$.
	\end{defn}

\begin{propn}
\label{propn:propositionftw}
The homomorphism $Sq$ is additive, and satisfies axioms $1,2,4$ and $5$ from Section \ref{subsec:theSqs}.
\end{propn}
\begin{proof}
To prove additivity, observe first $Sq(x+y) = Sq(x) + Sq(y) + Sq'(x \otimes y + y \otimes x)$. Hence we must show that $[Sq'(x \otimes y + y \otimes x)] =0 $ when $dx = dy = 0$. In such a case, we see $d(x \otimes y) = (x \otimes y + y \otimes x) h$. Using that $Sq'$ is a chain map, it follows that $Sq'((x \otimes y + y \otimes x)h) =d Sq'(x \otimes y)$. As multiplication by $h$ is injective on $H^*_{\mathbb{Z}/2}(M) = H^*(M) \otimes H^*(B \mathbb{Z}/2)$, this shows that $[Sq'(x \otimes y + y \otimes x)]$ is exact, as required.

\textit{Axiom 1} is immediate from the definition of $Sq^{i}$

\textit{Axiom 2} and naturality is true for the same reason as for the Morse cup product: see for example \cite[Section 2.1]{rot}. 

For \textit{Axiom 4}, for $|y| = 2 |x|$ the coefficient of $y$ in $Sq^{|x|}(x)$ is the number of elements of the $0-$dimensional moduli space $\mathcal{M}_{0}(y,x,x)$. From the definition of $\mathcal{M}_{0}(y,x,x)$, and Definition \ref{defn:morsecupprod}, this number is the same as the coefficient of $y$ in $x^2$. 

For \textit{Axiom 5}(1), $Sq^{i}(x) = 0$ for $i > |x|$ by definition, as only non-negative powers of $h$ are counted in Definition \ref{defn:mss}.

For \textit{Axiom 5}(2), $f_{v,s}$ is a perturbation of $f$. The perturbation may be chosen arbitrarily small in the $C^{2}$ topology. For generic $f$ there is no $-\nabla f$ flowline from $b$ to $a$ if $|b| < |a|$. As $f_{v,s}$ is close to $f$, this means that generically for any $v$ there is no `flowline' from $b$ to $a$ that has gradient $- \nabla f$ for $s < 0$ and $- \nabla f_{v,s}$ for $s > 0$. Hence $Sq^{i}(x) = 0$ for $i < 0$. 

We verify Axiom $3$ in Section \ref{subsec:propofSq} and Axiom $6$ in Section \ref{subsec:Cartan}.
\end{proof}

\begin{rmk}
\label{rmk:msqaxioms}

 Note that showing $Sq$ satisfies these axioms is not sufficient to show that it is indeed the Steenrod square, because we have not shown naturality under all continuous maps: this definition is only applicable for closed smooth manifolds. Nonetheless it provides a sanity check.

\end{rmk}

		\begin{rmk}
		\label{rmk:mssrmks} It is not straightforward to prove $Sq^{0} = id$ without a specific choice of Morse functions. We prove it in Section \ref{subsec:intss} using a different approach.
		\end{rmk}

\subsection{The Morse Steenrod square is the Steenrod square}
\label{subsec:tmssissq}
Recall from Section \ref{subsec:prelimbcncon} the Steenrod square due to Betz and Cohen. Recall in Section \ref{subsec:msss} Definition \ref{defn:mss} of the Morse Steenrod square. We will show that these are the same.

In the previous section we chose $f^i_{v,s}$ for $(v,s) \in S^{\infty} \times [0,\infty)$ and $i=1,2,3$, such that $f^2_{v,s} = f^3_{-v,s}$. We abbreviate $f_{v,s} = f^2_{v,s}$ where appropriate, and observe we may choose $f_{v,0}$ distinct from $\{ f_{-v,0}, f \}$ for each $v$ (as $\text{Conf}_3(C^{\infty}(M))$ is open and dense in $(C^{\infty}(M))^3$, hence the condition is generic). Recall, from Section \ref{subsec:prelimbcncon}, that $S$ was a space consisting of triples $(f^1_{s},f^2_{s},f^3_{s})$, with each $f^p_{s} \in U_f$, a small neighbourhood of the Morse function $f$. Observe $S \xrightarrow{\simeq} \text{Conf}_3(C^{\infty}(M))$ is a $\mathbb{Z}/2$-equivariant homotopy equivalence, using the map $(f^1_{s},f^2_{s}, f^3_{s}) \mapsto (f^1_{0},f^2_{0}, f^3_{0})$, with the obvious homotopy inverse. Henceforth, assume $S = \text{Conf}_3(C^{\infty}(M))$. Let $SB = S / \langle (23) \rangle$, where the transposition $(23)$ acts on $S$ by permutation of the components. As remarked previously, $SB $ is homotopy equivalent to $\mathbb{RP}^{\infty}$.

There is a natural $\mathbb{Z}/2$-equivariant map $i: S^{\infty} \xhookrightarrow{} S$ induced by $v \mapsto (f,f_{v,0},f_{-v,0})$, which descends to $i: \mathbb{RP}^{\infty} \rightarrow SB$. 

\begin{lemma}
$i_*: H_*(\mathbb{RP}^{\infty}) \rightarrow H_*(SB)$ is an isomorphism.
\end{lemma}
\begin{proof}
If $i$ is a weak homotopy equivalence then it is a quasi-isomorphism, see \cite[Proposition 4.21]{algtop}. As the two spaces are both homotopy equivalent to $\mathbb{RP}^{\infty}$ (which is a $K(\mathbb{Z}/2,1)$), it is sufficient to show that $$i_*: \pi_1(\mathbb{RP}^{\infty}) \cong \mathbb{Z}/2 \rightarrow \pi_1(SB) \cong \mathbb{Z}/2$$ is nontrivial.

Identify $S^1 \subset S^{\infty}$ with $\mathbb{R}/(2 \pi \mathbb{Z})$, parametrised by $\theta \in [0,2 \pi)$. Denote $f_v = f_{e^{iv},0}$. We wish to show that $\theta \mapsto [(f, f_{\theta/2}, f_{\theta / 2 + \pi})]$ determines a nontrivial loop, where $[\cdot]$ denotes the $\mathbb{Z}/2$-equivalence class. Observe that $\theta \mapsto (f, f_{\theta/2}, f_{\theta / 2 + \pi})$ is a path in $\text{Conf}_3(C^{\infty}(M))$ with different endpoints, hence the loop is not contractible.
\end{proof}

Consider Diagram \eqref{cohnorsquareandmorse}:

\begin{equation}\label{cohnorsquareandmorse}
\xymatrix{
H_*(\mathbb{RP}^\infty) \otimes H^*(M) 
\ar@{->}_-{i_* \otimes id}^-{\cong}[d]
\ar@{->}^-{MSq}[rr]
&
&
H^*(M)
\ar@{->}^{=}[d]
\\
H_*(SB) \otimes H^*(M)
\ar@{->}^-{s}[r]
\ar@/_2.0pc/@{->}_{Sq}[rr]
&
H_*(SB) \otimes H^*_{\mathbb{Z}/2}(M \times M)
\ar@{->}^-{q}[r]
&
H^*(M)
}
\end{equation}

Here $s(A \otimes x) = A \otimes x \otimes x$, and the map $q$ is as in Section \ref{subsec:prelimbcncon}. We have reinterpreted the Morse Steenrod square from the previous section, here denoted $MSq$, to be a map $MSq: H_*(\mathbb{RP}^\infty) \otimes H^*(M)  \rightarrow H^*(M)$, which we can do canonically as there is a unique graded basis of the homology of $\mathbb{RP}^{\infty}$. Observe that if we use the pushforward of the generator of $H_i(\mathbb{RP}^{\infty})$ by $i_*$ as the generator of $H_i(SB)$, then it is immediate that Diagram \eqref{cohnorsquareandmorse} commutes. Hence, Definition \ref{defn:mss} yields the Steenrod square.

	\subsection{The Cartan Relation}
	\label{subsec:Cartan}
		Let $T$ be a family of graphs as in Figure \ref{fig:modspatree}, parametrised by $t \in (0, \infty)$. Edge $e_1$ is a negative half-line and edges $e_3,e_4,e_6,e_7$ are positive half-lines. Edges $e_2,e_5$ are parametrised by $[0,t]$. Compactify $T$ by adding the graphs at $0$ and $\infty$ as in the figure, to obtain the compactification $T^{c} \cong [0,1]$. Use edge labels as given in Figure \ref{fig:modspatree}. Fix a Morse function $f$ on $M$. The edge parameter in each case will be denoted by $s$.

		\begin{figure}
			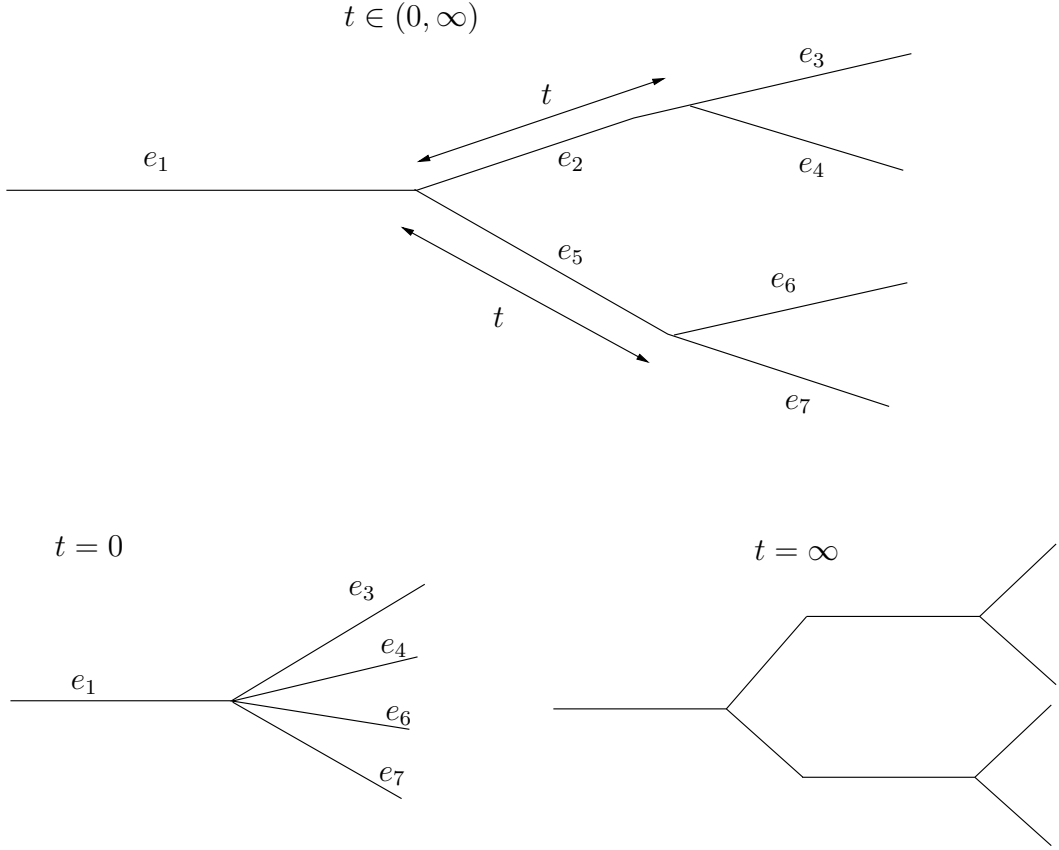
			\caption{Elements of $T^{c}$.}
			\label{fig:modspatree}
		\end{figure}

Pick 5 perturbations of $f$ corresponding to the 5 tree edges in $t=0 \in T^{c}$ in figure \ref{fig:modspatree}.  These are $f^{p}_{v,s,0}$ for $p$ the edge label, $s \in \mathbb{R}^{\pm}$ and $v \in S^{\infty}$. We choose $f^{1}_{v,s,0}=f$ for all $s, v$. We ensure that $f^{3}_{v,s,0} = f^{4}_{-v,s,0}$ and $f^{6}_{v,s,0} = f^{7}_{-v,s,0}$ for all $v,s$. The choice of $f^p_{v,s,0}$ is made along with an $S_0 \in \mathbb{R}$ such that $f^p_{v,s,0} = f$ for $|s| \ge S_0$ and for all edge labels $p$.

		Choose 7 perturbations of $f$ labelled $f^{p}_{v,s,t}$ for $p=1,...,7$ corresponding to the edge labels in Figure \ref{fig:modspatree}, where $t \in T^{c}$, $v \in S^{\infty}$ and $s \in \mathbb{R}^{+}$ for $p=3,4,6,7$, $s \in \mathbb{R}^{-}$ for $p=1$ and $s \in [0,t]$ for $p=2,5$. Choose $f^{1}$ to be independent of $s,v,t$ in this case. Choose Morse functions $f^{2}_{s,2},f^{5}_{s,2}$ for $s \in [0,2]$ such that $f^{p}_{s} = f$ for $s > 1$ and $p=2,5$. The $f^{p}$ must be chosen ``generically at each vertex of $\Gamma$", which is discussed in Appendix \ref{subsec:appendcartrel}. This ensures the transversality of the moduli spaces. The $f^{p}_{v,s,t}$ satisfy the following conditions: 
		
		\begin{enumerate}
			\item $f^{p}_{v,s,t} = f^{p}_{v,s,0}$ as picked previously for $p = 1,3,4,6,7$ and for all $t$.
			\item $f^2_{v,s,t},f^5_{v,s,t}$ are independent of $v$.
			\item For $t \ge 2$ and $p= 2,5$: \ $\begin{cases} \begin{array}{l} f^p_{s,t} = f^p_{s,2}  \text{ for } s \le 2, \\ f^p_{s,t} = f \text{ for } s \ge 2. \end{array} \end{cases}$ In particular, $f^p_{2,2} = f$.
		\end{enumerate}

		Fix $i \in \mathbb{N}$ and $x,y \in \text{crit}(f)$. Let $\overline{T} \xrightarrow{\cong} T^c$ consist of pairs $(|t|,t)$ where $t \in T^c \cong [0,\infty]$ and $|t|$ is the metric tree represented by $t$ as a topological space. The metric structure for $t \in [0,\infty)$ is that the outer edges are semi-infinite and parametrised by respectively $(-\infty,0]$ for the incoming edge and $[0,\infty)$ for the outgoing edges. The inner edges are of length $t$, parametrised by $[0,t]$. For the $t=\infty$ boundary, the metric structure on $|\infty|$ is that the edges attached to bivalent vertices are semi-infinite with the infinite end at the bivalent vertex. 

For $z \in \text{crit}_{2|x|+2|y|-i}(f)$ consider the space $\tilde{\mathcal{M}}_{1}(x,y,z)$ of triples $(t,u,v)$ with $t \in T^c$, $u: |t| \rightarrow M$ a map and $v \in S^{|x| + |y| - i}$, such that $u$ satisfies: 
 $$\partial_{s}u_{s,t} = -\nabla f^{p}_{v,s,t}$$ along edge $p$, with asymptotic conditions  $(z,x,x,y,y)$ on the exterior edges $(1,3,4,6,7)$. One needs to use an equivariant gluing theorem at the $t=\infty$ boundary, as discussed in Appendix \ref{sec:equivariantgluing}.

		For generic $t \in T^{c}$ there is a $0$-dimensional subset of pairs $(u: |t| \rightarrow M ,v \in S^i)$ satisfying the conditions. So $\tilde{\mathcal{M}}_{1}(x,y,z)$ is 1-dimensional. Observe that $\tilde{\mathcal{M}}_{1}(x,y,z)$ has a free $\mathbb{Z}/2$ action, $(t,u,v) \mapsto (t,u \circ \overline{r},-v)$ for $\overline{r}$ acting on $|t|$ by the permutation of edges $(34)(67)$. Let $\mathcal{M}_{1}(x,y,z) = \tilde{\mathcal{M}}_{1}(x,y,z) / (\mathbb{Z}/2)$, which is still 1-dimensional.

		We also define a moduli space $\tilde{\mathcal{M}}_{2}(x,y,z)$ by choosing another 7 Morse functions, labelled $f^{p}_{v,s,t}$ as above, but now with the conditions: 

		\begin{enumerate}
			\item $f^{p}_{v,s,t}= f^{q}_{-v,s,t}$ for $(p,q)=(3,4),(6,7),(2,5)$,
			\item $f^{p}_{v,s,t}$ is independent of $(v,t)$ for large enough $t$ and for $p = 1,3,4,6,7$,
			\item $f^p_{v,s,t} = f$ for $p=1,3,4,6,7$ and $|s| \ge 1$.
			\item For large enough $t$ and $s \in [1,t]$, $f^{2}_{v,s,t} = f^5_{v,s,t}= f$.  
		\end{enumerate}

		\begin{figure}
			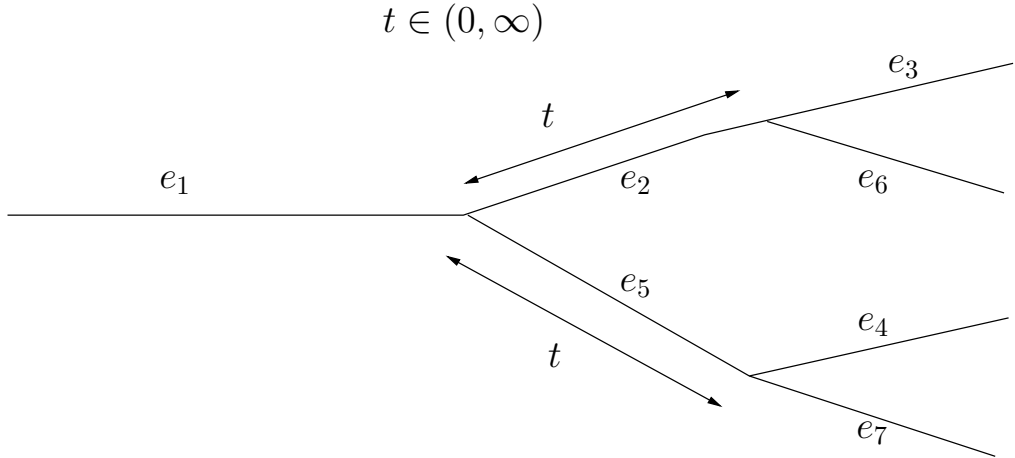
			\caption{Tree labelling for $\mathcal{M}_{2}$.}
			\label{fig:modspatree2}
		\end{figure}

In defining equations for pairs $(t,u,v) \in \tilde{\mathcal{M}}_{2}(x,y,z)$, use the edge labellings in Figure \ref{fig:modspatree2}, i.e. the edge labels $4$ and $6$ from Figure \ref{fig:modspatree} have been swapped. For each edge label the equations and asymptotic conditions are the same as in the $\tilde{\mathcal{M}}_{1}$ case. Further, there is a free $\mathbb{Z}/2$ action on $\tilde{\mathcal{M}}_{2}(x,y,z)$ similarly to $\tilde{\mathcal{M}}_{1}$ but with edge permutation $(25)(34)(67)$ (using the new edge labels in Figure \ref{fig:modspatree2}). Taking the quotient defines $\mathcal{M}_{2}(x,y,z) = \tilde{\mathcal{M}}_{2}(x,y,z)/ (\mathbb{Z}/2)$. 

The following theorem is classical, and the following proof is a modification of \cite[Section 2, Example 2]{betzcoh} for our definition of the Steenrod square. The modification uses a cobordism argument as in \cite[Section 3.4]{salamonfloer}.

		\begin{thm}[The Cartan Relation]
		\label{thm:classicalcartan}
		 	$$Sq^{i}(x \cup y) = \sum_{j+k=i} Sq^{j}(x) \cup Sq^{k}(y).$$ 
		\end{thm}
		\begin{proof}

The moduli space $\mathcal{M}_{1}(x,y,z)$ is a $1$-dimensional cobordism, corresponding to $[0,\infty]$, so $\# \partial \mathcal{M}_{1}(x,y,z) = 0$. Simliarly $\# \partial \mathcal{M}_{2}(x,y,z) = 0$. The $t = \infty$ boundary of $\mathcal{M}_{1}(x,y,z)$ is the count of the contribution of $z$ in $$\sum_{j+k=i} Sq^{j}(x) \cup Sq^{k}(y)$$ (see Figure \ref{fig:cartanSqSq} and Lemma \ref{lemma:lemmaSqSq}). The number of points in the boundary at $t=0$ for $\mathcal{M}_{2}(x,y,z)$ is the same as for $\mathcal{M}_{1}(x,y,z)$, as follows: suppose that $(0,u, v)$ is a point in the $t=0$ boundary of $\mathcal{M}_{2}(x,y,z)$. The domain of $u$ consists of a parametrised graph $\Gamma'$ with an incoming edge labelled $1$, and four outgoing edges labelled $3,4,6,7$. Consider the automorphism $r': \Gamma' \rightarrow \Gamma'$ that acts by the permutation $(46)$ on the edges (without changing the parametrisation). Then $(0, u \circ r', v)$ is a point in the $t=0$ boundary of $\mathcal{M}_{1}(x,y,z)$, and as $r'$ is an involution we see that this is a bijective correspondence. Notice that as we are working with $\mathbb{Z}/2$-coefficients, we do not need to worry about changing the orientation of the moduli space.

The number of points in the $t=\infty$ boundary component of $\mathcal{M}_{2}(x,y,z)$ is the count of the contribution of $z$ in $Sq^{i}(x \cup y)$, by Lemma \ref{lemma:lemmaSqcup}. Hence, the bijection between the $t=0$ boundaries of the moduli spaces, along with the $1$-cobordisms assigned to $\mathcal{M}_{1}(x,y,z)$ and $\mathcal{M}_{2}(x,y,z)$, yield that $$\sum_{j+k=i} Sq^j(x) \cup Sq^k(y) = Sq^i(x \cup y),$$ as required.
		\end{proof}

		\begin{figure}
			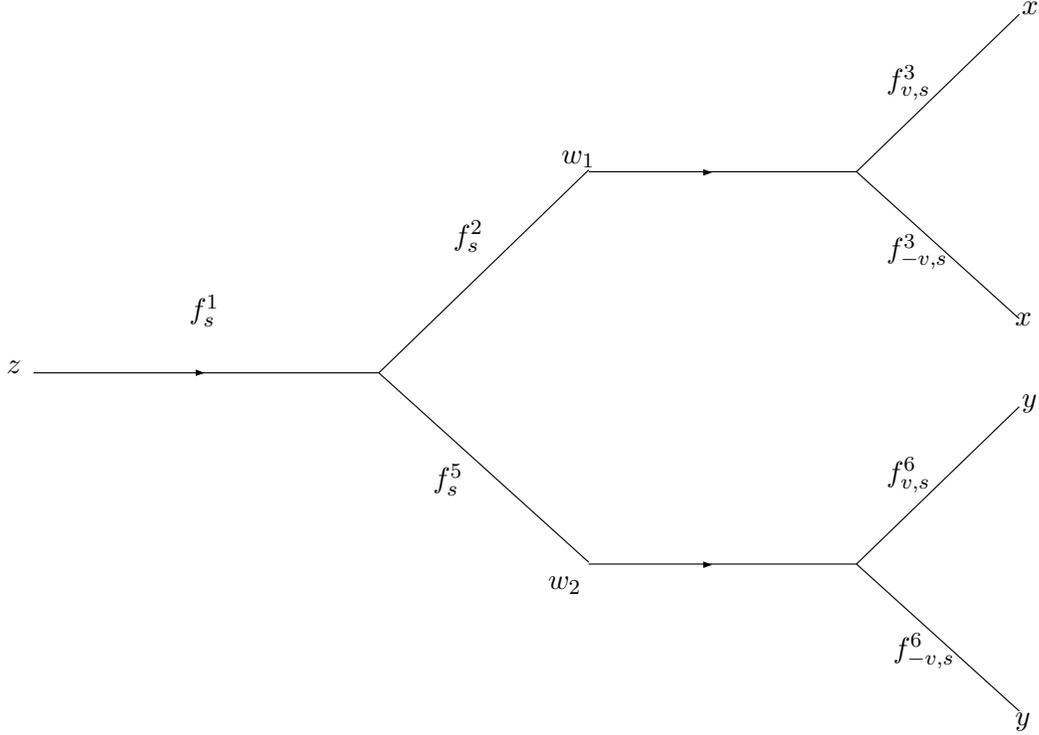
			\caption{Flowline configurations for $Sq(x) \cup Sq(y)$. }
			\label{fig:cartanSqSq}
		\end{figure}

		\begin{lemma}
		\label{lemma:lemmaSqSq}
			Summing over all choices of $w_1,w_2 \in \text{crit}(f)$, counting equivalence classes $[(u,v)] \in \mathcal{M}_1(x,y,z)$ satisfying the asymptotic conditions as shown in Figure \ref{fig:cartanSqSq}, yields the coefficient of $z$ in $\sum_{j+k=i} Sq^{j}(x) \cup Sq^{k}(y)$.
		\end{lemma}
		\begin{proof}
			We have that $|w_1| + |w_2| = |z| = |x|+|y| + i$. Hence if $|w_1| = |x|+j$ and $|w_2| = |y| + k$ then $j+k = i$. Throughout fix $w_1$,$w_2$ for the configuration, as outputs of $Sq^{j}(x), Sq^{k}(y)$ respectively.

			Restrict attention to the upper right-hand Y-shaped graph of Figure \ref{fig:cartanSqSq}. Suppose that we restrict the $v$ parameter space to $\mathbb{RP}^{|x|-j} \subset \mathbb{RP}^{|x|+|y|-i}$: in this case, counting $[(u,v)]$ satisfying the configuration conditions would be exactly the count of the coefficient of $w_1 \cdot h^{|x|-j}$ in $Sq^{j}(x)$, which we denote $n_{w_1}$. In our case, $v$ varies in the entirety of $\mathbb{RP}^{|x|+|y|-i}$, we call the set of such pairs $$\mathcal{U}_x = \left\{ [v,u] \biggr\vert \begin{array}{l} v \in S^{|x|+|y|-i} \text{ and } u: \Gamma \rightarrow M \text{ satisfies conditions as} \\ \text{illustrated in the upper right-hand graph of Figure } \ref{fig:cartanSqSq} \end{array}\right\}.$$ Here $[v,u]$ refers to taking the quotient by the $\mathbb{Z}/2$-action $(v,u) \rightarrow (-v,u \circ r)$ (where $r$ is the involution on the $Y$-shaped graph as seen previously). Similarly for the lower right-hand branch, for each $\mathbb{RP}^{|y|-k} \subset \mathbb{RP}^{|x|+|y|-i}$ there is a count of $n_{w_2}$, the coefficient of $w_2 \cdot h^{|y|-k}$ in $Sq^{k}(y)$. Define similarly $$\mathcal{U}_y = \left\{ [v,u] \biggr\vert \begin{array}{l} v \in \mathbb{RP}^{|x|+|y|-i} \text{ and } u: \Gamma \rightarrow M \text{ satisfies conditions as} \\ \text{illustrated in the lower right-hand graph of Figure } \ref{fig:cartanSqSq} \end{array}\right\}.$$ 

			Let $n_{z, w_1, w_2}$ be the coefficient of $z$ in $w_1 \cup w_2$ (the chain level Morse cup product obtained by using the perturbed Morse functions $f^1_s, f^2_s, f^5_s$). This is obtained by counting elements of the zero dimensional set corresponding to configurations as in the left hand $Y$-shaped graph of Figure \ref{fig:cartanSqSq}. We will show that the contribution of configurations as in Figure \ref{fig:cartanSqSq} to the coefficient of $z$ is $n_{z, w_1,w_2} \cdot n_{w_1} \cdot n_{w_2}$. 

Following \cite[Lemmas 4.2-4.5]{schwarzmorsesingiso}, suppose in fact that $x$ is a Morse cycle (specifically some sum of critical points, $\sum_i a_i \cdot x_i$ where $x_i \in \text{crit}(f)$). Then we may modify $\mathcal{U}_{x}$ to $\overline{\mathcal{U}_{x}},$ obtained by first taking the disjoint union of $a_i$ copies of $\mathcal{U}_{x_i}$ for each $i$ (defined as for $x$ above) and then adding in codimension $1$ strata, and identifying them in pairs (this can be done exactly because $dx = 0$). Here, the codimension $1$ strata correspond to the case when the $Y$-shaped graph undergoes a ``breaking" at one end. The outcome is then the union of a $Y$-shaped graph and an unparametrised flowline, such that:
\begin{itemize}
\item one of the $Y$-shaped graph's positively/negatively asymptotic critical points coincides with the flowline's negatively/positively asymptotic critical points. The other three asymptotic critical points are $w_1, x_i, x_i$ for some $i$.
\item the index difference between the asymptotic critical points of the unparametrised flowline is $1$.
\end{itemize}
Then observe that $\overline{\mathcal{U}_{x}}$ is a smooth manifold, \cite[Lemma 4.4]{schwarzmorsesingiso}. Let $$\pi_x : \overline{\mathcal{U}_{x}} \rightarrow \mathbb{RP}^{|x|+|y|-i}$$ be the projection onto the first coordinate. Then $\pi_x$ is a pseudocycle: specifically, consider $[v_n, u_n]$ such that $v_n$ converges in $\mathbb{RP}^{|x|+|y|-i}$ but $[v_n, u_n]$ has no convergent subsequence in $\overline{\mathcal{U}_{x}}$. By parametrised compactness of Morse flowlines we know that $[v_n, u_n]$ must have a convergent subsequence in the full compactification of $\mathcal{U}_{x}$: but that convergent subsequence must be in the codimension $2$ strata, as $\overline{\mathcal{U}_{x}}$ contains its codimension $1$ strata.

Observe also that by the second paragraph of the proof (i.e. knowing the intersection of $\pi_x$ with $\mathbb{RP}^{|x|-j}$, and in fact with any perturbation of $\mathbb{RP}^{|x|-j}$, is $n_{w_1}$) we deduce that $$\pi_x \bullet [\mathbb{RP}^{|x|-j}] = n_{w_1},$$ where $\bullet$ is the intersection number, hence $\pi_x$ is a weak representative of $n_{w_1} \cdot [\mathbb{RP}^{|y|-k}]$ (by which we mean that the intersection number of any cycle with $\pi_x$ is the same as with $n_{w_1} \cdot [\mathbb{RP}^{|y|-k}]$). Similarly the first projection $\pi_y :\overline{\mathcal{U}_y}  \rightarrow \mathbb{RP}^{|x|+|y|-i}$ is a weak representative of $n_{w_2} \cdot [\mathbb{RP}^{|x|-j}]$. 

The count of all solutions $[(u,v)]$ satisfying the configuration in Figure \ref{fig:cartanSqSq} is now $$n_{z, w_1,w_2} \cdot ( \pi_x \bullet \pi_{y}) = n_{z,w_1,w_2} \cdot n_{w_1} \cdot n_{w_2}.$$ 

Now recall from the definitions of $n_{w_1}, n_{w_2}$ that $$Sq^j(x) = \sum_{w_1 \in \text{crit}_{|x|+j}(f)} n_{w_1} w_1 h^{|x|-j}$$ and $$Sq^k(y) = \sum_{w_2 \in \text{crit}_{|y|+k}(f)} n_{w_2} w_2 h^{|y|-k}.$$ Then $$Sq^j(x) \cup Sq^k(y) = \sum_{w_1,w_2} n_{w_1} \cdot n_{w_2} \cdot w_1 \cup w_2,$$ and recalling that $n_{z,w_1,w_2}$ is the coefficient of $z$ in $w_1 \cup w_2$, the lemma is proved.
		\end{proof}

\begin{lemma}
\label{lemma:lemmaSqcup}
The count for the $t=\infty$ boundary component of $\mathcal{M}_{2}(x,y,z)$ is the count of the contribution of $z$ in $Sq^{i}(x \cup y)$.
\end{lemma}
\begin{proof}
The edge and asymptotic conditions are as shown in Figure \ref{fig:cartanSqcup}. The edges attached to bivalent vertices are semi-infinite with the infinite end at the bivalent vertex, which is a critical point of $f$. For this operation, the $t=\infty$ boundary, we choose the perturbed Morse functions so that the two right-hand Y-shaped graphs use the same perturbations $f^3,f^6$ of $f$. Specifically, we may assume that $f^3$ and $f^6$ are independent of $v$. The number of such setups is then immediately the coefficient of $z$ in $Sq^{i}(x \cup y)$.
\end{proof}

		\begin{figure}
			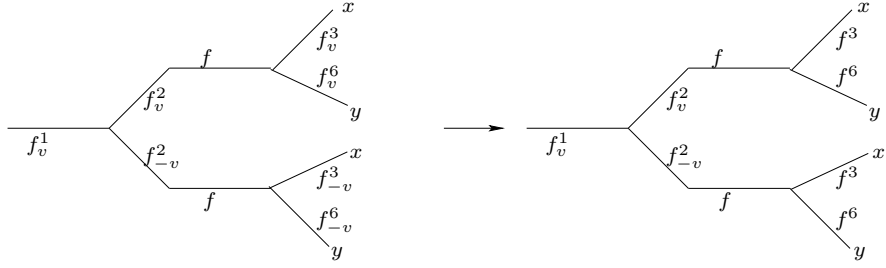
			\caption{Flowline configurations for $Sq(x \cup y)$.}
			\label{fig:cartanSqcup}
		\end{figure}

	\subsection{Steenrod Squares via intersections of cycles}
	\label{subsec:intss}

		Recall that there are nested equators $S^{i} \subset S^{\infty}$, invariant under the antipodal action. Let $a \in H^{|a|}(M)$. Let $\mathcal{B}$ be a basis of $H^* (M)$. 

In practice, we would like to work with representatives. A representative of a homology class $A$ is a pair $(X,\alpha)$, often denoted simply $\alpha$, where $X$ is a smooth compact manifold and $\alpha: X \rightarrow M$ is smooth such that $\alpha_*[X] = A$. We recall that over $\mathbb{Z}/2$-coefficients every homology class has a representative (see e.g. \cite[Theorem B]{buonhacon}). For notation, we will denote a homology class by $A$, $a$ will denote its Poincar\'e dual cohomology class, and $\alpha$ will be a representative as above. We will say that $\alpha$ represents a cohomology class $a$ if $\alpha$ represents its Poincar\'e dual homology class. Similarly for $b \in \mathcal{B}$, we denote $B = PD(b)$. As previously, we denote by $b^{\vee}$ the dual basis element to $b$ in the dual basis $\mathcal{B}^{\vee}$ of $H^*(M)$. 

In order to link this definition to the previous definition, we will weaken our requirements below: in fact we only ask that $\alpha: X \rightarrow M$ is a pseudocycle representative of $a$. Note however that the definition will proceed identically in the cases where we can instead use either representatives or embeddings. Denote by $\beta^{\vee}: Y_b \rightarrow M$ a pseudocycle representative of $PD(b^{\vee})$. 
%should we be slightly less strict conditions..?
We will choose some smooth manifold $X$, along with a sequence of smooth maps $\alpha_i: X \times S^i \rightarrow M \times S^i$ (for brevity we shorten $X_i := X \times S^i$) such that:
		\begin{enumerate}
		\item For $\pi_{2}: M \times S^i \rightarrow S^i$ the second projection, $\pi_{2}(\alpha_i(x,v)) = v$ for all $(x,v) \in X \times S^{i}$.
		\item The restriction $\alpha_i |_{X_j} = \alpha_j$ for $j \le i$.
		\item For $\pi_{1}: M \times S^i \rightarrow M$ the first projection, for any $v \in S^{i}$ then \begin{equation} \label{equation:alphav} \alpha_v:= \pi_1 \circ \alpha|_{X \times \{ v \}} : X_v := X \times \{ v \} \rightarrow M \end{equation} is a pseudocycle representative of $A$ in $M$ (and is well defined by (2) above).
		\item For $b \in \mathcal{B}$, in $M \times M \times M \times S^{i}$ we require: \begin{equation} \label{tripleintersection} (\Delta \times id) \pitchfork \euscr{W} \end{equation}  where \begin{equation} \label{euscrw} \euscr{W}: Y_b \times X \times X \times S^i \rightarrow M \times M \times M \times S^i \end{equation} is defined by $(y, x,x',v) \mapsto (\beta^{\vee}(y),\alpha_i(x,v), \alpha_i(x',-v),  v)$ and $$\Delta \times id: M \times S^i \rightarrow M \times M \times M \times S^i$$ is defined by $(z,v) \mapsto (z,z,z, v)$.
		\end{enumerate}

		The pseudocycles $\Delta \times id \text{ and } \euscr{W}$ in \eqref{tripleintersection} descend to pseudocycles $$[\Delta \times id]: M \times \mathbb{RP}^i \rightarrow M \times ((M \times M) \times_{\mathbb{Z}/2} S^i),$$ and $$[\euscr{W}]: Y_b \times ((X \times X) \times_{\mathbb{Z}/2} S^i)  \rightarrow M \times ((M \times M) \times_{\mathbb{Z}/2} S^i),$$ respectively. Provided $|b| = 2 |a|- i$, define $n_{i,b,a} = [\Delta \times id] \bullet [\euscr{W}]$, the intersection of these two pseudocycles (of complementary dimension).

		\begin{defn}[Steenrod Square]
		\label{defn:miss}
			Define $$Sq(a) = \sum_{i \in \mathbb{Z}, \ b \in \mathcal{B}, \ |b| = 2|a|-i} n_{i,b,a} b h^{i},$$ where $\#$ is the count modulo $2$.
		\end{defn}

			\begin{rmk}
To see that Definition \ref{defn:miss} is a good one, i.e. independent of the choice of $\alpha_i$ (all the other choices are immediately covered by pseudocycle theory, e.g. \cite{zinger}), observe that the given number of points $n_{i,b,a}$ in any given degree (by which we mean for any fixed choice of $S^i \subset S^{\infty}$) is obtained as the number of intersection points of two pseudocycles. A construction as in \cite[Lemma 3.2]{zinger} for two different choices of $\alpha_i$ yields a bordism of pseudocycles, meaning that the intersection number $n_{i,b,a}$ is independent of this choice.
\end{rmk}

		\begin{rmk}
		\label{rmk:ourdefinitionsthesame}
	The Morse Steenrod square of Definition \ref{defn:mss} is the same as Definition \ref{defn:miss} using the isomorphism $HM^*(M,f) \cong H^*(M)$ that intertwines the Morse product and the cup product, in particular as described in \cite{schwarzmorsesingiso}. 

Recall that for each $v$, and Morse cocycle $a = \sum n_i \cdot a_i$ ($n_i \in \mathbb{Z}$ and $a_i \in \text{crit}(f)$) there is a pseudocycle associated to the $s$-dependent Morse function $f_{v,s}$. The domain of this pseudocycle is constructed first by taking the spaces $W^s(a_i,f_{v,s})$ of smooth $u: [0,\infty) \rightarrow M$ such that $\partial u/ \partial t(s) = - \nabla f_{v,s}(u(s))$ and $u(\infty) = a_i$, the stable manifold under $f_{v,s}$. One then adds in the codimension $1$ strata of the standard Morse compactification, and then glues together the disjoint union of $n_i$ copies of each $W^s(a_i,f_{v,s})$, along the codimension $1$ strata, which one knows can be done because $da  = 0$. We call this space $\overline{W}(a,f_{v,s})$. The map of this pseudocycle is (on the codimension $0$ strata) evaluation at $0$, denoted $E_v: \overline{W}(a,f_{v,s}) \rightarrow M$. Details are in \cite[Lemma 4.5]{schwarzmorsesingiso}, for the pseudocycle $\overline{W}(a,f)$ associated to the fixed Morse function $f$. 

Recall that we chose $f_{v,s}$ in Section \ref{subsec:tmssissq}, based on Section \ref{subsec:prelimbcncon}. Specifically, they satisfy $f_{v,s} = \beta(s) f_{v} + (1-\beta(s))f$ (where $f_v$ is confined to a small contractible neighbourhood of Morse functions $U_f$ containing $f$). Observe that in this instance $f_{v,s} = f$ for $s \ge 1$. Recall that for each $v \in S^i$ there is a $1$-parameter family of diffeomorphisms $\phi_{v,s}: M \rightarrow M$ for $s \in [0,1]$, defined by $\phi_{v,0} = id$ and $$\partial \phi_{v,t} /\partial t |_{t=s}(x) = - \nabla f_{v,s}(x).$$ Then $W(a,f_{v,s}) = \phi_{v,1}^{-1}(W(a,f))$.

Hence, for each $i \in \mathbb{Z}_{\ge 0}$ we obtain an $\alpha_i: \overline{W}(a,f) \times S^i \rightarrow M \times S^i$, defined by $$\alpha_i(u,v) = (E_v \phi_{v,1}^{-1} u, v).$$ Then recalling the conditions we required from $\alpha$, earlier in Section \ref{subsec:intss}, we see that:
\begin{itemize} 
\item condition $(1)$ holds, 
\item condition $(2)$ is immediate because we define our map fibrewise for each $v$, 
\item condition $(3)$ holds because of \cite[Lemma 4.5]{schwarzmorsesingiso}, 
\item condition $(4)$ holds because of condition $(2)$ at the beginning of Section \ref{subsec:msss}.
\end{itemize}

		\end{rmk}

\begin{rmk}
\label{rmk:embeddedsubs}
As the definition in this section will be used as a computational tool for our purposes, for simplicity we will assume in certain places that our homology classes in $\mathcal{B}$ can be represented as embedded submanifolds: in this instance, we may replace a $\alpha_i: X \times S^i \rightarrow M \times S^i$ (which in such a case satisfies that $\pi_1 \alpha_i(\cdot,v): X \rightarrow M$ is an embedding for each $v \in S^i$) by $X_v := \pi_1 \alpha_i(X,v)$. 
\end{rmk}

\begin{rmk}
		Suppose that $\alpha$ is represented by an embedded submanifold $\mathcal{A} \subset M$. Then each $\alpha_i(X_i)$ cannot simply be $\{ (p,v) | p \in \mathcal{A}, v \in S^{i} \}$, because then transversality would not hold. More generally, we cannot assume that the pseudocycles $\alpha_i$ are independent of $v$. In the next section we construct a family of admissible choices of $\alpha_i$. However, we may take $B^{\vee} \times S^{i}$ to be such a ``standard representative". This is analogous to how, in the Morse definition, $f^1_s$ is chosen to be independent of $v$.
\end{rmk}

		\subsection{Properties of the Steenrod Square}
		\label{subsec:propofSq}

		As promised in Section \ref{subsec:msss} we now check Axiom 3 from Section \ref{subsec:theSqs}.

				\begin{lemma}
				\label{lemma:sq0baby}
				$Sq^{0}(PD(pt))=PD(pt)$. 
				\end{lemma}
				\begin{proof}
				Let $n=\text{dim}(M)$. Write $a = PD(pt)$. We construct a representative of $\{ pt \} \times S^n$ in $M \times S^n$:

				The submanifold $pt \subset M$ has trivialisable normal bundle, so the disc subbundle $D(pt)$ of the normal bundle $N(pt)$ embeds into $M$ as a small disc around $pt$. Let $S^{n}, D^n \subset \mathbb{R}^{n+1}$, where $$S^n = \biggr\{ (x_1,\ldots x_{n+1} ) \in \mathbb{R}^{n+1} \biggr\vert \sum_{i} x_i^2 = 1 \biggr\}$$ and $$D^n = \biggr\{ (x_1,\ldots x_{n+1}) \in \mathbb{R}^{n+1} \biggr\vert x_{n+1} = 0, \sum_{i} x_i^2 \le 1 \biggr\}$$ is the $n$-disc with the $n+1^{th}$ coordinate $0$. There is a natural flattening map denoted $\phi':S^{n} \rightarrow D^{n}$, where $\phi'(x_1,\ldots, x_n, x_{n+1}) = (x_1,\ldots x_n, 0)$ is projection of $S^n$ onto the first $n$ coordinates. Note that $\phi'$ is a double cover except on the equator, which is $\partial D^{n} \cong S^{n-1}$. 

				There is a diffeomorphism $D^{n} \cong D(pt) \subset M$. Composing $\phi'$ with this diffeomorphism defines $\phi : S^{n} \rightarrow M$. The map $\phi$ is homotopic to a constant map hence $\bigsqcup_{v \in S^n} (\phi(v), v)$, the graph of $\phi$ in $M \times S^n$, is cobordant to $\{ pt \} \times S^{n} \subset M \times S^n$. Specifically, we denote $\alpha_n : \{ pt \} \times S^n \rightarrow M \times S^n$ by $\alpha_n(pt, v) = (\phi(v),v)$, and this immediately satisfies most of the relevant properties of $\alpha_n$ from Section \ref{subsec:intss} (we will verify transversality after computing the points of intersection). 

Observe that for $b \neq [M]^*$ (hence $PD(b^{\vee}) \neq [M]$), and for a general choice of the dual basis pseudocycles $\beta^{\vee}$, there is no intersection as in Statement \eqref{tripleintersection} (hence transversality holds trivially). To check transversality in the case where $b = [M]^*$, we pick $\beta^{\vee} = id_M : M \rightarrow M$. Then for $\Delta \subset M \times M$ the diagonal, $\Delta \times S^i$ intersects $\Phi := \sqcup_{v \in S^i} \{ \phi(v) \} \times \{ \phi(-v) \} \times \{v \}$ exactly when $\phi(v) = \phi(-v)$. We know there is exactly one such pair $\{ \pm v_0 \}$, where $v_0 = (0,...,0,1) \subset S^{n} \subset \mathbb{R}^{n+1}$. 

To verify transversality, consider the tangent directions at $(\beta^{\vee}(x), \phi(v_0),\phi(-v_0), v_0)$ in $T(M \times M \times M \times S^i) = TM \oplus TM \oplus TM \oplus TS^i$. Those tangent directions in $0 \oplus 0 \oplus 0 \oplus TS^i$ and $T \Delta \oplus 0 \subset T(M \times M \times M) \oplus TS^i$ are all contained in $T(\Delta \times S^i) = T \Delta \oplus TS^i$. Similarly, as $\beta^{\vee} = id_M$ we obtain all tangent vectors in $TM \oplus 0 \oplus 0 \oplus 0$. It remains to show that we may obtain the rest of the tangent vectors of $T(M \times M \times M \times S^i)$. Observe that $d \phi(v_0) = - d \phi(-v_0)$ is nondegenerate,  because $v_0 \not\in \phi'(\partial D^n)$. Hence in particular $\{ (\phi(v), \phi(-v)) \}_{v \in S^n} \subset M \times M$ intersects $\{(x,x)\}_{x \in M} \subset M \times M$ transversely at $(\phi(v_0),\phi(-v_0))$. This immediately implies transversality. 

To calculate the coefficient of $a$ in $Sq^{0}(a)$, count the number of (pairs of) solutions to $\phi(v) = \phi(-v)$ modulo $\mathbb{Z}/2$. Recall from above there exists exactly one such (pair of) solutions $v = \pm v_0$. Taking this modulo the $\mathbb{Z}/2$ action gives $Sq^0(a) = a + ...$. The cycle $\{ a \}$ generates $H^n (M)$ so there are no more contributions to $Sq^0(a)$ for degree reasons.
				\end{proof} 

			An easy generalisation of the above proof shows:

			\begin{lemma}
\label{lemma:sq0}
				For $x \in H^*(M)$, when $PD(x)$ is represented by an embedded submanifold $\chi$ then $Sq^{0}(x)=x$.
			\end{lemma} 
			\begin{proof}

				Let $x$ be as given in the statement. Proceed as in the previous lemma, but now $x = PD(X)$ for some cycle $X$. It is convenient to assume that $X$ is in a basis for the homology of $M$, with $PD(X^{\vee}) = x^{\vee}$ being the corresponding member of the dual basis under the intersection product. Similarly to above, for general pseudocycle representatives $\alpha: \chi \rightarrow M$ and $\alpha^{\vee}: Y \rightarrow M$ of $X, X^{\vee}$, we detemine that $\alpha \cdot \alpha^{\vee}$ consists of a finite, odd number of points $\{ p_{i} \}$ (since $x \cdot x^{\vee} = 1$ mod $2$ by definition). In particular, this is true when $\alpha$ is the embedding of $\chi$. Moreover, this is true of any generic sufficiently small perturbation of $\alpha$, such as when the image of $\alpha$ is contained in a sufficiently small normal disc bundle of $\chi$.

Each of these $p_i$ has a small neighbourhood $U_{i} \subset M$ such that the normal bundle $N(\chi)$ of $\chi$ is trivial on $U_{i} \cap \chi$, with the $U_{i}$ being pairwise disjoint. Pick a bump function $\beta_{i}$ for each neighbourhood $U_{i}$. On the neighbourhood $U_{i}$, there is a diffeomorphism between the disc bundle and the trivial bundle $D(U_{i}) \cong  (U_{i} \cap \chi) \times D^{n-\dim(\chi)}$. Using the tubular neighbourhood theorem, $N(\chi)$ and hence $(U_{i} \cap \chi) \times D^{n-\dim(\chi)}$ embeds into $M$ via a map $e$. 

				Hence there is a smooth map $\phi: \chi \times S^{n- \dim(\chi)} \rightarrow M$, such that if $x \in \chi$ is not in any $U_{i}$, then $\phi(x,v) = x$. Otherwise $x$ is in exactly one $U_{i}$ and we define $\phi(x,v): = e(x, \beta_{i}(x) \phi'(v))$, where $\phi': S^{n-\dim (\chi)} \rightarrow D^{n- \dim (\chi)}$ is the flattening map as in the previous lemma. This yields $\alpha_{n-\dim(\chi)}(x,v) :=\phi(x,v)$, recalling that $n-dim(\chi) = |x|$. Consider the intersection modulo $\mathbb{Z}/2$, whose transversality is verified as in Lemma \ref{lemma:sq0baby}. The coefficient of $xh^{|x|}$ is obtained by using as the output cycle $\alpha^{\vee}(Y) \times S^{n-dim(\chi)}$. By construction, such intersections only occur when the first coordinate is one of the $p_{i}$. At $p_{i}$, there is exactly one pair of solutions corresponding to the two solutions as in the previous claim: i.e. $\phi'$ is $2$ to $1$ on a dense open subset. 
				
				Take the quotient by the $\mathbb{Z}/2$ action to deduce that the number of contributions is an odd number (the number of $p_{i}$) multiplied by an odd number (the number of pairs of solutions at each $p_{i}$), hence is odd. Therefore $Sq^{0}(x) = xh^{|x|} + ...$. To show that there are no more terms in $Sq^0(x)$, repeat this with $S^{n-dim(\chi)} \times B$ as the output cycle, for $B$ representing another element of the dual basis of homology. Strictly, to cover all cases at once we must choose pseudocycle representatives for every $B \in \mathcal{B}^{\vee}$. Then instead of considering $\{ p_i \}$, we now have $\{p_{B,i} \}$, where $B$ varies in $\mathcal{B}^{\vee}$, which are pairwise distinct. Define similarly pairwise disjoint $U_{B,i} \ni p_{B,i}$, and a map $\phi$ as previously. Then as $B \neq X^{\vee}$ is in the dual basis, a general pseudocycle representative of $B$ intersects $\chi$ with an even number of points. We count exactly as in the previous case, except the number of contributions is an even number (the intersection number of $B \cdot \chi$) multiplied by an odd number. Hence the count is even and the contributions due to other $B$ are $0$.
			\end{proof}

\begin{corollary}
\label{corollary:trivialisable}
Let $A$ be a closed submanifold of $M$, with trivialisable normal bundle. Then $Sq^{i}(PD([A])) = 0$ for $i \neq 0$.
\end{corollary}	
\begin{proof}
		Use the embedding $e:  A \times D^{n-\text{dim}(A)}  \rightarrow M$ by inclusion of the unit disc bundle of $A$, which exists because $A$ has trivialisable normal bundle, to define $A_v$ for $v \in S^{n-i}$ for $i > 0$. Immediately no intersections occur for $i \neq 0$, as $A_v \cap A_{-v} = \emptyset$ for all $v \in S^{n-i}$.
\end{proof}

\begin{rmk}
	More generally, for any immersed submanifold $A \xrightarrow{} X$, consider the homology class $[A] \in H_*(X)$. Then $Sq^i (PD([A]))$ is the Stiefel-Whitney class $w_i(N_A X)$ (where $N_A X$ is the normal bundle of $A$ in $X$). An account of Stiefel-Whitney classes is given in \cite{stiefelwhitney}.
\end{rmk}

\section{Quantum Steenrod Square via Morse theory}
\label{sec:SqQviaMorse}
	Let $M$ be a closed monotone symplectic manifold. The definition of the quantum Steenrod square uses a $Y$-shaped graph as with the Morse Steenrod square, but  now allows for a $J$-holomorphic sphere at the trivalent vertex in the Y-shaped graph in the definition. This is a $J$-holomorphic sphere with $2+1$ marked points, and 2 incoming and 1 outgoing Morse flowlines from the respective marked points.

	Make a choice of $f^p_{s,v}$ as in Subsection \ref{subsec:msss}, for $p=1,2,3$. Let $N$ be the minimal Chern number of $M$. Fix $i,j \in  \mathbb{Z}_{\ge 0}$ and $a,b \in H^{*}(M)$ with $$|b| - 2 |a| + i + 2jN = 0.$$ 

	Let $\mathcal{M}'_{i,j}(b,a)$ be the moduli space of pairs $(u,v)$, such that:
\begin{itemize}
\item $v \in S^{i}$,
\item $u: S^2 \rightarrow M$ is a simple $J$-holomorphic map of Chern number $2jN$, i.e. \begin{equation} \label{equation:jvsholoc} du(z) = J(u(z)) \circ du(z) \circ j_{S^2}(z), \end{equation} where $j_{S^2}$ is the standard almost complex structure on $S^2$, 
\item the $-\nabla f^1_{s,v}$ flowline from $u(0)$ converges to $b$ as $s \rightarrow -\infty$ and the $-\nabla f^p_{s,v}$ flowline from $u(1), u(\infty)$ converge to $a$ as $s \rightarrow \infty$ for $p=2,3$ respectively.
\end{itemize}

There is a free $\mathbb{Z}/2$-action on this moduli space: $$\iota_{\mathcal{M}} (u,v) = (u \circ R, -v)$$ where $R$ is the unique M\"obius map in $PSL(2,\mathbb{C})$ swapping $1 \text{ and } \infty$ and fixing $0$. Let $$\mathcal{M}_{i,j}(a,b) = \mathcal{M}'_{i,j}(a,b) / \iota_{\mathcal{M}}.$$ The space $\mathcal{M}_{i,j}(a,b)$ is a smooth manifold of dimension $|b| - 2|a| + i + 2jN$. See Appendix \ref{subsection:transvholspheres} for a discussion of transversality for the equivariant case in the presence of pseudoholomorphic spheres.

	\begin{defn}[Morse Quantum Steenrod Square]
	\label{defn:mqss}
		Pick a basis $\mathcal{B}$ of $H^*(M)$. Let $a \in H^*(M)$. For each $i,j$, let  $$Q\mathcal{S}_{i,j}(a) = \sum_{b \in \mathcal{B} : |b| + i + 2jN = 2 |a|} \# \mathcal{M}_{i,j}(b,a) \cdot b,$$ 

$$Q\mathcal{S}(a) = \sum_{i,j} Q\mathcal{S}_{i,j}(a) \cdot h^{i}T^{j}.$$

		Extend to a general element of $QH^{*}(M)$ by $Q\mathcal{S}(aT^{j}) = Q\mathcal{S}(a)T^{2j}$. 
	\end{defn}

The proof that $Q\mathcal{S}$ is an additive homomorphism is identical to Proposition \ref{propn:propositionftw}. First define $Q \mathcal{S}' : H^*_{\mathbb{Z}/2}(M \times M) \rightarrow QH^*(M) \otimes H^*(B \mathbb{Z}/2)$. This is identical to $Sq'$ from Definition \ref{defn:mss}, but one uses moduli spaces $\mathcal{M}_{i,j}(a_1,a_2,a_3)$ that have a $J$-holomorphic map $u:S^2 \rightarrow M$ in place of the intersection of the Morse flowlines. Then $Q \mathcal{S} = Q \mathcal{S}' \circ \text{double}$, and observe that $Q \mathcal{S}'(x_1 \otimes x_2 + x_2 \otimes x_1) = 0$.

	\begin{rmk}
\label{rmk:propertiesqs}
For $a \in H^*(M)$, $$Q\mathcal{S}_{i,0}(a) = Sq^{|a|-i}(a)$$ as it counts constant spheres. Further, $$\sum_{j \ge 0} Q\mathcal{S}_{0,j}(a) T^j = a * a$$ is the usual quantum product.
	\end{rmk}

	\subsection{Quantum Steenrod Squares via intersections of cycles}
	\label{subsec:qssintcyc}
		Let $a \in H^{|a|}(M)$, and we pick a basis $\mathcal{B}$ of $H^*(M)$. Denote $\alpha = PD(a), \beta = PD(b)$ for $b \in \mathcal{B}$. We define a moduli space and evaluation maps analogously to Section \ref{subsec:quantcupprod}: given $j \in \mathbb{Z}_{\ge 0},$ consider $\mathcal{M}_{j}(J) \times S^{i}$ consisting of pairs $(u,v)$ where $u$ is a $J$-holomorphic map such that $u_*[S^2]$ has Chern number $jN$ and $v \in S^i$. Fixing $q \in \mathbb{CP}^1$, the evaluation maps are $ev_{q} \times id_{S^i}:\mathcal{M}_j (J) \times S^i \rightarrow M \times S^i$, which we abusively denote $ev_q$. Choose a sequence of maps $(\alpha_i)_{i=0}^{\infty}: X \times S^i \rightarrow M \times S^{i}$ as in Section \ref{subsec:intss}, satisfying conditions (1), (2) and (3) but we will modify (4). Firstly, let
$\mathcal{M}(j,J)$ be the space of $J$-holomorphic spheres of Chern number $jN$, with a $\mathbb{Z}/2$-action acting by $u \mapsto u \circ R$, where as in Section \ref{sec:SqQviaMorse} $R: S^2 \rightarrow S^2, \ R(z) = z/(z-1)$. Further, for $b \in \mathcal{B}$, and $i \in \mathbb{Z}_{\ge 0}$ we define:
$$\euscr{Y}_Q: Y_b \times ((X \times X) \times_{\mathbb{Z}/2} S^i) \rightarrow M \times ((M \times M) \times_{\mathbb{Z}/2} S^i)$$ by $$(y,((x,x'),[v])) \mapsto (\beta^{\vee}(y),[\alpha_i(x,v), \alpha_i(x',-v), v]),$$ and
$$ev: \mathcal{M}(j,J) \times_{\mathbb{Z}/2} S^i \rightarrow M \times ((M \times M) \times_{\mathbb{Z}/2} S^i)$$ is defined by $$[u,v] \mapsto (u(0), [u(1), u(\infty), v]).$$ The required condition (4) is then:

		\begin{enumerate}
  \setcounter{enumi}{3}
		\item For $b \in \mathcal{B}$, and $i \in \mathbb{Z}_{\ge 0}$, the intersection of pseudocycles \begin{equation} \label{tripleintersectionquant}  ev(\mathcal{M}(j,J) \times_{\mathbb{Z}/2} S^i) \cap \euscr{Y}_Q(X \times \mathbb{RP}^{i}) \end{equation} is transverse in $M \times ((M \times M) \times_{\mathbb{Z}/2} S^{i})$.
		\end{enumerate}

Given $i,j \in \mathbb{Z}_{\ge 0}$, for $|b| = 2 |a| - i -2j$, the pseudocycles are of complementary dimension. Define $n_{i,j}(a,b)$ to be the intersection number of these pseudocycles. 

		\begin{defn}[Quantum Steenrod Square]
		\label{defn:singqss}
			For $a \in H^*(M)$ define $$Q\mathcal{S} : QH^{*}(M) \rightarrow QH^{*}(M)[h],$$ such that 

			$$Q\mathcal{S}(a):= \sum_{i,j \in \mathbb{Z}_{\ge 0}, \ b \in \mathcal{B}, \ |b| = 2|a|-i-2jN} n_{i,j}(a,b) \cdot b T^{j} h^i$$
			with $Q\mathcal{S}$ a linear homomorphism. Then extend $Q\mathcal{S}$ linearly to $QH^*$ by requiring that $Q\mathcal{S}(a T^k) = Q\mathcal{S}(a) T^{2k}$. Also define $Q\mathcal{S}_{i,j}(a)$ as previously.
		\end{defn}

		As in the classical case this is equivalent to Definition \ref{defn:mqss}.

\subsection{Quantum Stiefel Whitney Class}
		For a smooth compact manifold $M$, the classical Stiefel-Whitney class of $TM$, $w(TM)$, is constructed as in \cite[Section 5.3]{cohnor}, using a certain graph operation. We will not go into details. A more classical treatment is found in \cite{stiefelwhitney}.

		Using the convention that $\langle ah,A \rangle = \langle a,A \rangle h$ for $a \in H^*(M), A \in H_*(M)$, one can use a gluing theorem as in \cite[Theorem 20]{cohnor}, or a direct argument to prove that:

		\begin{lemma}
			$$w(TM) = \sum_{y \in \mathcal{B}} Sq(y) \cdot \langle Sq(y^{\vee}),[M] \rangle.$$
		\end{lemma}
		\begin{proof}
			Recalling that $w(TM) = Sq(v)$, where $v$ is the Wu class of $M$, it is sufficient to prove that \begin{equation} \label{equation:wTM} v = \sum_{y \in \mathcal{B}} y \cdot \langle Sq(y^{\vee}),[M] \rangle. \end{equation} Suppose that we write $v$ as an element of $H^*(M)[h]$, i.e. $$v = \sum_{y \in \mathcal{B}, \ i \ge 0} n_{y,i} \cdot y h^i.$$ Substituting this into the definition of $v$, i.e. $\langle Sq(b), [M] \rangle = \langle b \cup v, [M] \rangle$ for any $b \in H^*(M)$, we obtain that $$\langle Sq(b), [M] \rangle = \sum_{y \in \mathcal{B}} n_{y,i} \cdot \langle b \cup y h^i, [M] \rangle.$$ For each $y \in \mathcal{B}$, let $b = y^{\vee}$. Hence $$\langle Sq(y^{\vee}), [M] \rangle = n_{y,i} \cdot \langle y^{\vee} \cup y h^i, [M] \rangle = n_{y,i} \cdot h^i \langle [M]^{*}, [M] \rangle = n_{y,i} \cdot h^i,$$ and \eqref{equation:wTM} follows. 
		\end{proof}

		Let $M$ be a closed monotone symplectic manifold.

		\begin{defn}[Quantum Stiefel-Whitney Class]
			The Quantum Stiefel-Whitney class is $$w_Q(TM) := \sum_{y \in \mathcal{B}} Q\mathcal{S}(y) \langle Sq(y^{\vee}),[M] \rangle.$$
		\end{defn}

	It follows from this definition and a grading argument that:	

	\begin{lemma}
	\label{propn:quantumstiefel4Ng2n}
		If the minimal Chern number $N > (\dim M)/2$ then $w_Q(TM) = w(TM)$.
	\end{lemma}
\begin{proof}
	We will show that given the assumptions of this lemma, for every $y \in H^*(M)$, either $Q \mathcal{S}(y) = Sq(y)$ or $\langle Sq(y^{\vee}),[M] \rangle = 0$.

	Suppose that $Q \mathcal{S}(y)$, which is of degree $2|y|$, has a summand containing some nontrivial power of $T$, which is of degree $2N$. This implies that $2|y| \ge 2N > \dim M$. Hence $|y| > (\dim M)/2$. Hence $|y^{\vee}| < (\dim M)/2$, and therefore for degree reasons there can be no summand of the form $[M]^* h^j$ in the expansion of $Sq(y^{\vee})$. Therefore $\langle Sq(y^{\vee}),[M] \rangle = 0$.
\end{proof}

	\begin{corollary}
		Let $M = \mathbb{CP}^n$. Then $w_Q(TM) = w(TM)$.
	\end{corollary}
	\begin{proof}
		The minimal Chern number for $\mathbb{CP}^n$ is $N = n+1 > n = (\dim \mathbb{CP}^n) /2$. Now apply Lemma \ref{propn:quantumstiefel4Ng2n}.
	\end{proof}

\section{The Quantum Cartan relation}
\label{sec:quancar}
	
	We continue the discussion from Example \ref{exmpl:difficulties}. Consider the space $M^{\#}_{0,5}$ of 5 distinct marked points on the $2$-sphere, and let $$M_{0,5} = M^{\#}_{0,5} / PSL(2,\mathbb{C})$$ where the M\"obius group $G = PSL(2,\mathbb{C})$ acts diagonally on the 5 marked points. There are two different descriptions of $M_{0,5}$ that will be useful: 
	\begin{enumerate}
		\item $\{ (z_{0},z_{1},z_{2},z_{3},z_{4}) \} / G$ of five distinct points modulo the action of $G$, reparametrising M\"obius maps.
		\item $\{ (0,1, \infty, z_{3},z_{4}) \}$ with $z_{3},z_{4}$ distinct from each other and from $0,1,\infty$. 
	\end{enumerate}

		The former description gives a simpler definition of the compactification, but the latter description is more useful when describing homology classes. Letting $z_{3},z_{4}$ vary in the description (2) yields a third description:
	\begin{enumerate}
  \setcounter{enumi}{2}
		\item $$M_{0,5} \cong ((\mathbb{CP}^{1} - \{ 0,1,\infty \}) \times (\mathbb{CP}^{1} - \{0,1,\infty \})) - \Delta,$$ where $\Delta$ is the diagonal. 
	\end{enumerate}
One compactifies this space, adding stable genus $0$ nodal curves with $5$ marked points (there are 10 copies of $\mathbb{CP}^{1} - \{0,1,\infty \}$ and 15 points to add), one obtains a space $$\overline{M}_{0,5} \simeq Bl_{\{ (0,0),(1,1),(\infty, \infty) \}}(\mathbb{CP}^{1} \times \mathbb{CP}^{1}).$$ See \cite[Section D.7.]{jholssympl}. Then $\overline{M}_{0,5}$  is homotopy equivalent to $(\mathbb{CP}^{1} \times \mathbb{CP}^{1}) \# 3 (\overline{\mathbb{CP}^{2}})$, which means:
	\begin{equation} \label{equation:m05coh} \begin{array}{l} H^{*}(\overline{M}_{0,5}) = \mathbb{F}_{2}[\delta_1, \delta_2,w_{0},w_{1},w_{\infty}] / I \\ I = (\delta_1^{2},\ \delta_2^{2}, \ w_{i}^{3},w_{i}^{2}+\delta_1 \delta_2 \text{ for all } i, \text{ and } \delta_i \omega_j \text{ for all } i,j, \text{ and } \omega_i \omega_j \text{ for } i \neq j) \end{array} \end{equation}
where $w_{i}$ corresponds to the exceptional divisor at $(i,i)$ and $\delta_1, \delta_2$ correspond to the spheres $\mathbb{CP}^{1} \times \{ pt \}$ and $\{ pt \} \times \mathbb{CP}^{1}$ respectively: thus all the generators have degree $2$. A treatment of this is \cite[Section D.7]{jholssympl}. Henceforth $W_i = PD(w_i) \text{ and } \Delta_j = PD(\delta_j)$ for $i=0,1,\infty$ and $j=1,2$.

Let $x,y,z$ be cohomology classes in $H^*(M)$. Let $\zeta: Z^{\vee} \rightarrow M$ be a pseudocycle representative of $PD(z^{\vee})$. There is a natural $\mathbb{Z}/2$ action on $\overline{M}_{0,5}$, induced by $(12)(34)$. Specifically, $$\iota : (z_{0},z_{1},z_{2},z_{3},z_{4}) \mapsto (z_{0},z_{2},z_{1},z_{4},z_{3}).$$ Then $\iota \times -\text{id}$ defines a free diagonal $\mathbb{Z}/2$ action on $\overline{M}_{0,5} \times S^{i}$ for each $i$. Define $$P_i := ( \overline{M}_{0,5} \times S^i) / (\iota \times -\text{id}).$$ 

Pick smooth maps $\chi_i: X_{i} \rightarrow M$ and $\gamma_i: Y_{i} \rightarrow M$, as in Section \ref{subsec:qssintcyc}, for the cohomology classes $x,y$. Then $\mathcal{M}'_{i,j}(x,y,z)$ consists of triples $(u,m,v)$, where $m$ is a $5$-pointed genus $0$ holomorphic nodal curve, and $u: |m| \rightarrow M$ is a smooth stable (nodal) $J$-holomorphic map representing a homology class of Chern number $jN$ (here $|m|$ refers to forgetting the marked points of $m$). The parameter space is $v \in S^{i}$. The map $u$ satisfies $u(z_{0}) \in \zeta(Z^{\vee})$, $u(z_{1}) \in \chi_v(X_{v})$, $u(z_{2}) \in \chi_{-v}(X_{-v})$, $u(z_{3}) \in \gamma_v(Y_{v})$ and $u(z_{4}) \in \gamma_{-v}(Y_{-v})$. 

There is a $\mathbb{Z}/2$-action on $\mathcal{M}'_{i,j}(x,y,z)$, acting by:
 \begin{equation} \label{equation:actiononmoduli} (u,m,v) \mapsto (u,\iota m,-v), \end{equation} recalling that $\iota$ acts, as on $\overline{M}_{0,5}$, by the permutation of marked points $(12)(34)$. Then the action \eqref{equation:actiononmoduli} is well defined because $|\iota m| = |m|$. There is also an action induced by reparametrisation: specifically, if $g \in PSL(2,\mathbb{C})$ acts on some holomorphic sphere $m^a$ of $m$, with corresponding $J$-holomorphic map $u^a: |m^a| \rightarrow M$, then  \begin{equation} \label{equation:actionGonmoduli} g \cdot (u^a,m^a,v) = (u^a \cdot g^{-1}, g \cdot m^a, v). \end{equation} We denote $$\mathcal{M}_{i,j}(x,y,z)$$ the moduli space obtained after quotienting $\mathcal{M}'_{i,j}(x,y,z)$ by the actions in Equation \eqref{equation:actiononmoduli} and \eqref{equation:actionGonmoduli}.

There is a natural map $$\pi_{x,y,z}: \mathcal{M}_{i,j}(x,y,z) \rightarrow P_{i}, \ [u,m,v] \mapsto [\text{stab}(m),v],$$ where $\text{stab}(m)$ denotes taking the stabilisation of the $5$-pointed genus $0$ nodal curve $m$, which corresponds to an element of $\overline{M}_{0,5}$. The square brackets denote taking equivalences classes with respect to the actions of $PSL(2,\mathbb{C})$ and $\mathbb{Z}/2$. 

\begin{rmk}
\label{rmk:evalmapquancar}
We can think of $\mathcal{M}_{i,j}(x,y,z)$ as the inverse image under the evaluation map $$ev: \overline{\mathcal{M}_{0,5}(J,j)} \times_{\mathbb{Z}/2} S^{i} \rightarrow M \times (M^{4} \times_{\mathbb{Z}/2} S^{i}),$$ $$ev([[u,(p_0, p_1...,p_4)],v]) = (u(p_0),[(u(p_1),...,u(p_4)),v])$$ where $\mathcal{M}_{0,5}(J,j)$ is the set of $5$-pointed $J$-holomorphic maps $u: \mathbb{CP}^{1} \rightarrow M$ of Chern number $jN$, where $m$ combines the information of $\mathbb{CP}^1$ along with the $5$ marked points. The space $\overline{\mathcal{M}_{0,5}(J,j)}$ is a partial compactification using stable nodal genus $0$ holomorphic curves that contain no repeated or multiply covered components. The $\mathbb{Z}/2$-action acts on the marked points by the permutation $(12)(34)$. Observe that one uses some large machinery, namely the gluing theorem for $J$-holomorphic curves (see \cite[Chapter 10]{jholssympl}), to show that this partial compactification of the space of simple maps has a fundamental class. Then given $x,y,z$ as previous to this remark, it is immediate from the definition that $ \mathcal{M}_{i,j}(x,y,z)$ is obtained by \begin{equation} \label{equation:lotsofpseudo} ev^{-1} \left( \zeta(Z^{\vee}) \times \left[ \bigcup_{[v] \in \mathbb{RP}^{i} } \alpha_v(X_{v}) \times \alpha_{-v}(X_{-v}) \times \gamma_v(Y_{v}) \times \gamma_{-v}(Y_{-v}) \times \{ v \} \right] \right). \end{equation} As in Appendix \ref{subsec:mssrmks}, we may interpret the expression between $\left(, \right)$ in Equation \eqref{equation:lotsofpseudo} as a pseudocycle. The square brackets $\left[, \right]$ in Equation \eqref{equation:lotsofpseudo} denote the equivalence class under the $\mathbb{Z}/2$ action. This allows us to calculate $\dim \mathcal{M}_{i,j}(x,y,z)$ as follows. 
\end{rmk}

Let $Q$ be a closed submanifold of the parameter space $P_{i}$. Then  $Q$ represents a cycle in $H_{*}(P_{i})$, such that: \begin{equation} \label{eq:dimension} \dim \pi_{x,y,z}^{-1}(Q) = |z| - 2|x| - 2|y| + \dim(Q) + 2jN. \end{equation} In particular, for $Q = P_{i}$, using $\dim(P_{i}) = 4+i$, $$\dim \mathcal{M}_{i,j}(x,y,z) = \dim \pi_{x,y,z}^{-1}(P_{i}) = |z| - 2|x| - 2|y| + i+4 + 2jN.$$

\begin{defn}
\label{defn:qopn}
	Let $W$ be a cycle in $H_*(P_{i})$ with $i,j$ fixed, represented by a union of embedded closed submanifolds $\bigcup_{a \in A} Q_{a} \subset P_i$. Let $x,y \in H^*(M)$. Define $$q_{i,j}(W)(x,y) = \sum_{z : \dim \pi_{x,y,z}^{-1}(Q) = 0} \left( \sum_{a \in A} \# (\pi_{x,y,z}^{-1}(Q_a)) \right) \cdot zT^{j}$$ where the first sum is taken over a basis of $z$ for $H^{|z|}(M)$ such that Equation (\ref{eq:dimension}) is $0$. Extending bilinearly over $\mathbb{Z}/2 [T]$, this defines a bilinear map $$q_{i,j}(W) : QH^{k}(M) \otimes QH^{l}(M) \rightarrow QH^{k+l-|W|}(M).$$
\end{defn}

\begin{lemma}
\label{lemma:additiveandindep}
The homomorphism $$q_{i,j}(W) : QH^{k}(M) \otimes QH^{l}(M) \rightarrow QH^{k+l-|W|}(M)$$ does not depend on the representative of $W$, and is additive.
\end{lemma}
\begin{proof}
	Represent $W$ by a pseudocycle $\omega: U \rightarrow P_i$ (in the case of Definition \ref{defn:qopn}, we chose a union of embedded submanifolds). Observe that the coefficient of $z$ in $q_{i,j}(W)(x,y)$ is the intersection number of two pseudocycles. Using notation as previously, these are $\pi_{x,y,z}: \mathcal{M}_{i,j}(x,y,z) \rightarrow P_i$ and $\omega: U \rightarrow P_i$. We know that intersection numbers are independent of the choice of pseudocycle representative, i.e. $q_{i,j}(W)$ only depends on the homology class of $W$. 

	Then it is immediate that $q_{i,j}(W+W') = q_{i,j}(W) + q_{i,j}(W')$, as if $\omega: U \rightarrow P_i$ represents $W$ and $\omega': U' \rightarrow P_i$ represents $W'$ then consider $\omega'': U \sqcup U' \rightarrow P_i$, defined by $\omega''|_{U^{\lambda}} = \omega^{\lambda}$ for $\lambda = ' \text{or } ''$. Then $\omega ''$ represents $W+W'$. The intersection numbers from the previous paragraph are additive.
\end{proof}

		\begin{figure}
			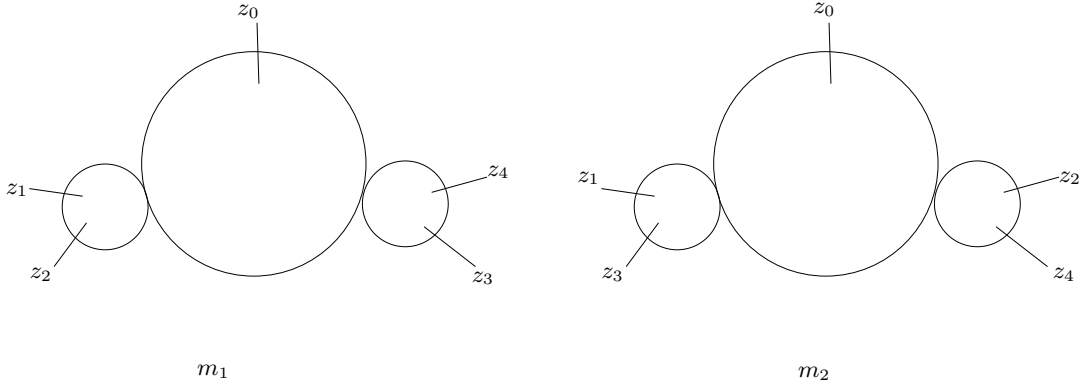
			\caption{$m_{1}$, $m_{2}$ $\in$ $\overline{M}_{0,5}$}
			\label{fig:m05elmts}
		\end{figure}

\begin{rmk}
	One likewise calculates the coefficients of $z T^j$ in $q_{i,j}(W)(x,y)$ (with notation as in Definition \ref{defn:qopn}) in the following way. Take the cup product of the cycle $\pi^* \rho$, where $\rho = PD(Q) \in H^* (P_i)$, with the pullback of $z^{\vee} \times x \times x \times y \times y$ under the evaluation map (specifically, the cup product takes place in $H^*(\overline{\mathcal{M}_{0,5}(j,J)} \times_{\mathbb{Z}/2} S^i)$). Integrate this over the equivariant fundamental class of $\overline{\mathcal{M}_{0,5}(j,J)}$. 
\end{rmk}

	In the following, use a cell decomposition for $S^{i}$ with cells $D^{i,\pm}$ in degree $i$, corresponding to the two hemispheres of dimension $i$. For $d$ the differential on cellular chains, $d(D^{i,\pm}) = D^{i-1,+} + D^{i-1,-}$.

The class of cases we consider are when $\text{dim}(Q) = i$. If $m_{1}, m_{2} \in \overline{M}_{0,5}$ are as given in Figure \ref{fig:m05elmts}, then $m_{1}$ and $m_{2}$ are invariant under the $\mathbb{Z}/2$ action on $\overline{M}_{0,5}$. Hence $\{ m_{1} \} \times D^{i,+}$ and $\{ m_{2} \} \times D^{i,+}$, which are embedded submanifolds of $P_i$, represent well defined cycles in $H_*(\overline{M}_{0,5} \times_{\mathbb{Z}/2} S^{i})$. For $p=1,2$ call these cycles $Q_{p}^{i}$. To see that these cycles are indeed closed, observe that (for example using singular homology), if $X \subset \overline{M}_{0,5}$ is an embedded submanifold, then $X \times D^{i,+}$ represents some chain in $H_*(P_i)$. Then abusing notation (applying the K\"unneth isomorphism, and writing the submanifold $X$ instead of a sum of the simplices representing $X$), $$d([X \times D^{i,+}]) = [(dX) \times D^{i,+}] + [X \times (D^{i-1,+} + D^{i-1,-})] = [(X + \iota X) \times D^{i-1,+}].$$ The brackets $[,]$ represent that we have taken the quotient by $\mathbb{Z}/2$ of the chain complex $C_*(\overline{M}_{0,5} \times S^i)$. The last equality uses the $\mathbb{Z}/2$-action on $\overline{M}_{0,5} \times S^{i}$. Hence, if $X$ is a $\mathbb{Z}/2$-invariant closed submanifold, such as $\{ m_1 \}$ and $\{ m_2 \}$, then the chain represented by $X \times D^{i,+}$ is closed in equivariant homology.

Indeed, by the previous, for $i > 0$ the chain ``$\{ pt \} \times D^{i,+}$" is only a cycle when $pt$ is a fixed point of the $\mathbb{Z}/2$ action on $\overline{M}_{0,5}$. The space of fixed points $(\overline{M}_{0,5})^{\mathbb{Z}/2}$ is the disjoint union of a sphere containing $m_{2}$ and the single point $m_{1}$. See Remark \ref{rmk:z2actioncompact} at the end of this section for more details on this $\mathbb{Z}/2$-action. 

\begin{lemma}
\label{lemma:lem1}
	$$\sum_{i,j} q_{i,j}(Q_{1}^{i})(x,y)h^i = Q\mathcal{S}(x)*Q\mathcal{S}(y) \text{ and } \sum_{i,j} q_{i,j}(Q_{2}^{i})(x,y)h^i = Q\mathcal{S}(x*y).$$
\end{lemma}
\begin{proof}
	For the rest of this proof, we fix $i,j$, and we show that $$q_{i,j}(Q_{1}^{i})(x,y) = [Q\mathcal{S}(x)*Q\mathcal{S}(y)]_{i,j}T^{j}.$$ To do this we proceed as in Lemma \ref{lemma:lemmaSqSq}, using $1$-dimensional moduli spaces, the ends of which count e.g. $\sum_{i,j} q_{i,j}(Q_{1}^{i})(x,y) \cdot h^i$ and $Q\mathcal{S}(x)*Q\mathcal{S}(y)$ respectively. This yields a bordism between the endpoints, with more details provided in Appendix \ref{subsec:bordismquantumcartantrans}.

Fixing some $i \in \mathbb{N}$ (the dimension of the sphere in which $v$ will vary), we consider the $1$-dimensional moduli spaces from Section \ref{subsec:Cartan}, denoted $\tilde{\mathcal{M}}_1(x,y,z)$ and $\tilde{\mathcal{M}}_2(x,y,z)$. Recall the spaces $T^c \cong [0,\infty]$, and $\overline{T} \rightarrow T^c$. We define quantum analogues $\tilde{\mathcal{M}}^Q_{p}(x,y,z)$ of the $\tilde{\mathcal{M}}_p(x,y,z)$ from Section \ref{sec:quancar} for $p=1,2$, where now each element of $\tilde{\mathcal{M}}^Q_{p}(x,y,z)$ is a triple $(t, u, v)$ where $t \in T^c$, $v \in S^i$ and $u: (|t|_Q,t) \rightarrow M$ is continous, and smooth away from nodes. Here, $|t|_Q$ is obtained by taking the graph associated to $|t|$, and adding a sphere at each trivalent vertex (in such a way that the incoming edge of $t$ is attached at $0$ on the sphere, and the outgoing vertices are attached at $1,\infty$ respectively). We then require that $u$ is $J$-holomorphic on each sphere, satisfies the edge and asymptotic equations as in Section \ref{subsec:Cartan}, and the sum of the Chern numbers of the three spheres is $jN$. The $\mathbb{Z}/2$-action acts by $(u, v) \mapsto (u \circ \overline{r}, -v)$, where $\overline{r}$ acts on $|t|$ as in Section \ref{sec:quancar} and extends to the holomorphic spheres in the following ways:
\begin{itemize}
\item for $\tilde{\mathcal{M}}^Q_{1}(x,y,z)$, the involution $\overline{r}$ acts by $z \mapsto z / (z-1)$ on the two right holomorphic spheres and the identity on the left holomorphic sphere.
\item for $\tilde{\mathcal{M}}^Q_{2}(x,y,z)$, the involution $\overline{r}$ acts by $z \mapsto z / (z-1)$ on the left holomorphic sphere and the identity on the two right holomorphic spheres.
\end{itemize}

For the $t=0$ end of the moduli spaces, we instead consider for $\tilde{\mathcal{M}}^Q_{p}(x,y,z)$ continuous maps $u: |m_p|' \rightarrow M$. Here, $m_p$ are the elements of $\overline{M}_{0,5}$ as in Figure \ref{fig:m05elmts}, and $|m_p|'$ is obtained from the nodal sphere configuration $|m_p|$ associated to $m_p$, by attaching to $z_0$ the negative half-line $(-\infty,0]$ and to $z_1,z_2,z_3,z_4$ the positive half-line $[0,\infty)$. We then require the same conditions on the edges, the asymptotics and the energy of the nodal spheres. The $\mathbb{Z}/2$-action extends continuously from the $t \in (0,\infty]$ action. We let $\mathcal{M}^Q_{p}(x,y,z) = \tilde{\mathcal{M}}^Q_{p}(x,y,z) / (\mathbb{Z}/2)$ for $p=1,2$.

It is then immediate from the definition that counting the setups corresponding to the $t=0$ end of $\mathcal{M}_{p,j}(x,y,z)$ is the coefficient of $z$ in $q_{i,j}(Q_{p}^{i})(x,y)h^i$, for each $i$. Hence, it remains to show that counting the $t=\infty$ ends of $\mathcal{M}^Q_{p}(x,y,z)$ corresponds to the coefficient of $Q\mathcal{S}(x)*Q\mathcal{S}(y)$ and $Q\mathcal{S}(x*y)$ respectively for $p=1,2$. The proof of this is identical to Lemmas \ref{lemma:lemmaSqSq} and \ref{lemma:lemmaSqcup} respectively.

We do not need to worry about the bubbling off of extra spheres because $M$ is monotone. Specifically, any $J$-holomorphic bubble must have strictly less than $2$ marked points. Introducing ``phantom" marked points, we may consider this to be a $3$-pointed Gromov-Witten invariant corresponding to intersections with the Poincar\'e dual of $1 \in H^0(M)$. We know that, for a general choice of $J$, such Gromov-Witten invariants only contain contributions from constant spheres, as in \cite[Proposition 11.1.11(ii)]{jholssympl}.
	\end{proof}

We now prove a slightly more general lemma for $\mathbb{Z}/2$-equivariant homology.

\begin{lemma}
\label{lemma:lem111}
Suppose that $M$ is a smooth connected manifold with a smooth $\mathbb{Z}/2$-action $\iota:M \rightarrow M$. Suppose that $W^n, L^{n-1} \subset M$ are submanifolds, fixed set-wise by $\iota$, of dimensions $n, n-1$ respectively, and representing respective homology classes $[W]$, $[L]$. Suppose further that $W = L \cup U \cup \iota U$ for some open submanifold $U \subset W$ such that $\partial \overline{U} = L$, where $\overline{U}$ is the closure of $U$ in $W$.

Then denoting $D^{i,+}$ for the upper $i$-dimensional hemisphere in $S^j$ ($i \le j$), the submanifolds $W \times D^{i-1,+}$ and $L \times D^{i,+}$ of $M \times S^j$ represent homologous elements of $H_*(M \times_{\mathbb{Z}/2} S^j)$ i.e. $$[W \times D^{i-1,+}] = [L \times D^{i,+}].$$ 
\end{lemma}
\begin{proof}
By the K\"unneth isomorphism, using singular homology, there is a quasi-isomorphism between  $C_{\bullet}(M) \otimes C_{\bullet}(S^j)$ and $C_{\bullet}(M \times S^j)$, here using singular homology. In fact we may replace singular homology of $C_{\bullet}(S^j)$ by cellular homology, as the K\"unneth isomorphism is natural on chain complexes. In this cellular decomposition, there are two $i$-cells for each $0 \le i \le j$, such that one obtains a decomposition of $S^{j+1}$ from $S^j$ by attaching two $j+1$-cells along their boundaries at $S^j$.

Observe then that there is an involution on $C_{\bullet}(M) \otimes C_{\bullet}(S^j)$, which is the chain map $\phi := \iota_* \otimes (-id)_*$. We consider the homology of the complex $(C_{\bullet}(M) \otimes C_{\bullet}(S^j))/\phi$, the quotient of the complex $C_{\bullet}(M) \otimes C_{\bullet}(S^j)$ by $\phi$, with the differential being induced by the differential on the tensor product. Then we know that the K\"unneth isomorphism is natural (in particular, with respect to the action of $\phi$), hence $$H_*((C_{\bullet}(M) \otimes C_{\bullet}(S^j))/\phi) \cong H_*(C_{\bullet}(M \times S^j) / \phi).$$ The homology of the latter complex is isomorphic to the homology of $M \times_{\mathbb{Z}/2} S^j$, but we will represent chains using the former complex.

We will (abusively) denote by $W \times D^{i-1,+}$ the $\phi$-equivalence class of the chain (i.e. sum of simplices) corresponding to the submanifold $W \times D^{i-1,+} \subset M \times S^j$. Observe that $$d(\overline{U} \times D^{i,+}) = (d \overline{U}) \times D^{i,+} + \overline{U} \times (d D^{i,+}),$$ abusively also denoting by $d$ the differentials on all possible complexes. We know that $d \overline{U} = L$ by assumption. Further, $d D^{i,+} = D^{i-1,+} + D^{i-1,-}$. Hence $$d(\overline{U} \times D^{i,+}) = L \times D^{i,+} + \overline{U} \times (D^{i-1,+} + D^{i-1,-}).$$ Note that $$\begin{array}{lll} \overline{U} \times (D^{i-1,+} + D^{i-1,-}) &=& \overline{U} \times D^{i-1,+} + \overline{U} \times D^{i-1,-} \\ & =& \overline{U} \times D^{i-1,+} + \iota \overline{U} \times D^{i-1,+} \\ &=& (\overline{U} + \iota \overline{U}) \times D^{i-1,+}, \end{array}$$ using for the second equality that the chains represent elements of the complex quotiented by the involution $\phi$. But note that by definition the chains $\overline{U} + \iota \overline{U} = W$ (summing simplices, the boundaries match and cancel along $L$). Hence $$d(\overline{U} \times D^{i,+}) = L \times D^{i,+} + W \times D^{i-1,+},$$ as required.
\end{proof}

Let $A^{i} = Q_{1}^{i} - Q_{2}^{i}$. Let $W_{q}$ be the pullback under the blowdown $$Bl_{(0,0), (1,1), (\infty,\infty)}(\mathbb{CP}^{1} \times \mathbb{CP}^{1}) \rightarrow Bl_{(q,q)}(\mathbb{CP}^{1} \times \mathbb{CP}^{1}),$$ of the exceptional $\mathbb{CP}^{1}$ divisor in $Bl_{(q,q)}(\mathbb{CP}^{1} \times \mathbb{CP}^{1})$, for $q=0,1,\infty$. The elements of $W_0$ are given in Figure \ref{fig:w0}.

		\begin{figure}
			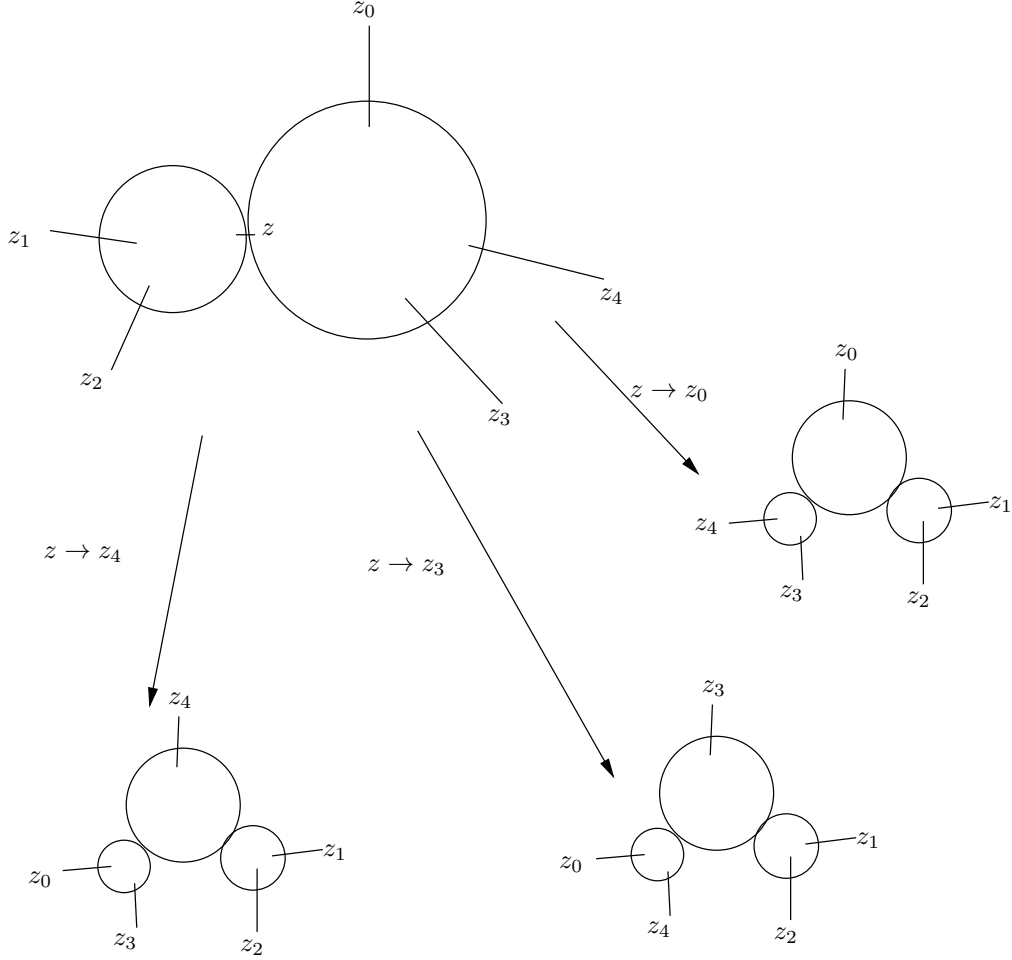
			\caption{Elements of $W_0$}
			\label{fig:w0}
		\end{figure}

\begin{lemma}
\label{lemma:lem222}
$[W_0 \times D^{i-2,+}] = [\{ m_1 \} \times D^{i,+}] + [\{ m_2 \} \times D^{i,+}]$.
\end{lemma}
\begin{proof}
We use $M = \overline{M}_{0,5}$ and $W = W_0$, and $\iota = (12)(34)$. Observe that we may identify $W_0 \cong S^2$ with the extended complex plane, fixing $(z_0,z_3,z_4) = (0,1,\infty)$. The $z \in \mathbb{C} \cup \{ \infty \}$ corresponds to the freely moving point on the four-pointed component of an element $m$ of $W_0$: specifically, the node connecting together the two components (i.e. copies of $S^2$) that together comprise $m$. The $\mathbb{Z}/2$-action on $\mathbb{C} \cup \{ \infty \}$ is then $z \mapsto z/(z-1)$. Let $L = \mathbb{R} \subset \mathbb{C} \cup \{ \infty \}$. By Lemma \ref{lemma:lem111}, we know that $$[W_0 \times D^{i-2,+}] = [L \times D^{i-1,+}].$$ Now observe that $L$ contains two fixed points, $\{ m_1, m_3 \} \in W_0$ corresponding to the points $\{ 0, 2 \} \in \mathbb{R} \subset \mathbb{C} \cup \{ \infty \}$. Applying Lemma \ref{lemma:lem111} again, we obtain that $$[W_0 \times D^{i-2,+}] = [L \times D^{i-1,+}] = [\{m_1, m_3 \} \times D^{i,+}] = [\{m_1\} \times D^{i,+}] + [\{m_3\} \times D^{i,+}].$$

Hence it remains to prove that $[\{m_3\} \times D^{i,+}] = [\{m_2\} \times D^{i,+}]$. Recall we stated earlier (and will elaborate in Remark \ref{rmk:z2actioncompact}) that the fixed point set of $\overline{M}_{0,5}$ corresponds to the union of $\{ m_1 \}$ and a $2$-dimensional sphere. In particular, the points $m_3$ and $m_2$ can be joined by a path of invariant points, which we denote $l$. Then $d (l \times D^{i,+}) = \{m_3\} \times D^{i,+} + \{m_2\} \times D^{i,+}$, as required.
\end{proof}

\begin{proof}[Proof of Theorem \ref{thm:quancar}]
	By Lemmas \ref{lemma:lem222} and \ref{lemma:additiveandindep}, $q_{i,j}(\{ m_{1} \} \times D^{i,+}) = q_{i,j}(\{ m_{2} \} \times D^{i,+}) + q_{i,j}(W_{0} \times D^{i-2,+})$. Multiply by $h^{i}$ and sum over all $i,j$ and apply Lemma \ref{lemma:lem1}.
\end{proof}

\begin{lemma}
\label{lemma:lemqWtermscor}
	The homomorphism $q_{i,j}(W_0 \times D^{i-2,+})$ is only nonzero for $i > 0, j >0$.
\end{lemma}
\begin{proof}
	The $j=0$ case corresponds to $J$-holomorphic maps that are constant. Suppose that we have chosen input cocycles $x,y$ and a test output cocycle $z^{\vee}$, such that the moduli space used to calculate the coefficient of $z$ in $q_{i,j}(W_0 \times D^{i-2,+})(x,y)$ is $0$-dimensional. In particular, we require that $|z| = 2|x| + 2|y| - (i-2)$, using Equation \eqref{eq:dimension} and recalling that $\dim W_0 \times D^{i-2,+} = i$ and $j=0$. However, such setups in fact consists of pairs $(m \in W_0, (v,u))$ such that $(v,u)$ is a configuration as in the $t=0$ end of Figure \ref{fig:modspatree}. For generic choices of data, there do not exist any such $(v,u)$ (the space of such pairs is of virtual dimension $|z| - 2|x| - 2|y| + i = -2$), hence the moduli space is empty and therefore trivially transverse. 

For the vanishing for $i=0$, observe that the $h^{i}$ terms correspond to calculating $q_{i,j}(W_0 \times D^{i-2,+})$ which vanishes for $i < 2$, as $S^{\infty}$ has no cells of negative dimension.
\end{proof}

We will verify Theorem \ref{thm:quancar} in the case of $\mathbb{CP}^{1}$.

\begin{exmpl}[$\mathbb{CP}^{1}$]
\label{exmpl:cp1calc}
Let $x$ be the generator of 2-dimensional cohomology. We verify that $$[Q\mathcal{S}(x)*Q\mathcal{S}(x)]_{i,j} = Q\mathcal{S}_{i,j}(x*x) + q_{i,j}(W_{0} \times D^{i-2,+})(x,x).$$ We know that there can only be contributions from $j>0$ and $i>0$, by Lemma \ref{lemma:lemqWtermscor}. In the cases where $j=0$ or $i=0$, we have already verified this using Example \ref{exmpl:difficulties}. For degree reasons, there cannot be any solutions for $j \ge 2$ or $i \ge 4$, and there cannot be solutions for $i=1,3$ (as $\mathbb{CP}^1$ has only even cohomology). Hence we need to consider only the cases $(i,j) = (4,1), (2,1)$.

For the case $(i,j) = (4,1)$, $$[Q\mathcal{S}(x) * Q\mathcal{S}(x)]_{4,1} = [T^{2} + h^{4}T]_{4,1} = h^{4}T$$ $$Q\mathcal{S}_{4,1}(x*x) = Q\mathcal{S}_{4,1}(T) = [T^{2}]_{4,1} = 0.$$ We then calculate $q_{4,1}(W_{0} \times D^{2,+})(x,x)$.

Pick representatives of $PD(x) \times S^{2}$ as follows: let $\phi : S^{2} \rightarrow D^{2}$ be the ``flattening map" of the sphere, i.e. if $S^{2} \subset \mathbb{R}^{3}$ it is projection onto $\mathbb{R}^{2} \subset \mathbb{R}^{3}$. Pick two disjoint discs in $\mathbb{CP}^{1} = S^{2}$, call them $D$ and $D'$, and pick maps $\eta: D^2 \rightarrow D \xhookrightarrow{} \mathbb{CP}^{1}$ and $\eta': D^2 \rightarrow D \xhookrightarrow{} \mathbb{CP}^{1}$ identifying $D^{2}$ with $D,D'$ respectively. Let $\psi = \eta \circ \phi: S^2 \rightarrow \mathbb{CP}^{1}$, and likewise $\psi' = \eta' \circ \phi: S^2 \rightarrow \mathbb{CP}^{1}$. Then two representatives of $PD(x)_{v} = \{ pt \}_{v}$ are $\psi(v)$ and $\psi'(v)$ where $v$ varies in $S^{2}$. Recall that elements of $W_{0}$ are as in Figure \ref{fig:w0}. 

Every element of $W_0$ consists of two spheres, joined at a point, which in this discussion we call ``components": recall that one has three special points, and one has four. The space $S^2$ has minimal Chern number $N = 2$, so for $J$-holomorphic maps $u$ from elements of $W_0$ to $\mathbb{CP}^1$, the map $u$ may be non-constant on only one of the two components. Further, this $J$-holomorphic map must be degree 1 on the other component. It is immediate that the map must be constant on the component with three marked points (if it were constant on the other component, then the solution cannot be rigid as $z_{4}$ can vary freely), and the other sphere of $u$ has degree $1$. Then $u(z_{1}) = \psi(v)$ and $u(z_{2}) = \psi(-v)$ meet at the unique point on the sphere where $\psi(v) = \psi(-v)$. Hence there is one solution, and this solution gives the correction term $h^{4}T$. 

For the case $(i,j) = (2,1)$, observe that the coefficient of $h^2$ in the correction term corresponds to using as the parameter space $D^{2-2,+} \subset S^2$, which is a single point: thus, we are simply performing a nonequivariant calculation. As we are calculating the coefficient of $T$, as in the previous paragraph we know that any $J$-holomorphic maps must be constant on the component with three special points. This setup then corresponds to deducing the coefficient of $x T$ in $(x*x) * (x \cup x)$, but $x \cup x = 0$ in $\mathbb{CP}^1$. Example \ref{exmpl:difficulties} provides the contributions $Q\mathcal{S}_{2,1}(x*x) = 0$ and $[Q\mathcal{S}(x)*Q\mathcal{S}(x)]_{2,1} = 0$.
\end{exmpl}

Henceforth, for brevity we will denote $$q(W)(x,y) := \sum_{i,j} q_{i,j}(W_0 \times D^{i-2,+})(x,y)h^i$$ for $x,y \in QH^*(M)$.

\begin{rmk}[Quantum Cartan in Classical Case]
Lemma \ref{lemma:lemqWtermscor} gives a sanity check that in the classical case, $Sq(x) \cup Sq(y) = Sq(x \cup y)$.
\end{rmk}

\begin{rmk}
\label{rmk:z2actioncompact}
We recall that the $\mathbb{Z}/2$-action $\iota$ on $M_{0,5}$ acts on the labels of the points by the transposition $(12)(34)$. Hence, suppose that $(0,1,\infty,z_3,z_4) \in M_{0,5}$. Then $\iota (0,1,\infty,z_3,z_4) = [0,\infty,1,z_4,z_3]$. This is no longer from description $(2)$ of $M_{0,5}$. We must apply the element $R \in PSL(2,\mathbb{C})$ such that $R(z) = z/(z-1)$. Then $$[0,\infty,1,z_4,z_3] = [R0,R\infty,R1,Rz_4,Rz_3] = (0,1,\infty,z_4/(z_4 - 1), z_3/(z_3-1)).$$ Such a point in $M_{0,5}$ is fixed exactly when $z_3 = z_4/(z_4-1)$. This provides a $2$-dimensional family $F$ of fixed points of $\iota$, as $z_3$ varies in $S^2 - \{0,1,\infty,2 \}$.

Using description $(2)$ of $M_{0,5}$, the action $\iota$ extends in the obvious way to the compactification, by permuting edge labels of the marked points. Fixed points of the $\iota$-action on the compactification can be found in the limit as $z_3 \rightarrow 0,1,\infty,2$, assuming that $z_4 \rightarrow 0,\infty,1,2$ respectively. These four points compactify $F$ to a $2$-sphere that we denote $\overline{F}$. The point when $z_3 \rightarrow 1$ and $z_4 \rightarrow \infty$ is $m_2$.

We now use description $(1)$ of $M_{0,5}$. It can be deduced by inspection that there are no fixed points if exactly one of the pairs $(z_1, z_3), (z_2,z_4), (z_1,z_4), (z_2,z_3)$ collide. The collisions of $(z_1,z_2)$ are only fixed if they collide at $2$, which is covered above. The collision of $(z_3,z_4)$ is only fixed if it occurs at $2$, which is the single point when $(z_1,z_2)$ collide at $0$, which is counted above. It is an easy check that any point in a collision of $(z_0,z_i)$ is not fixed for $i=1,2,3,4$. Hence, the only other possibility to check is a collision when two pairs collide at the same time, say $(z_i,z_j),(z_k,z_l)$ for $(i,j,k,l)$ all distinct. Checking the cases, we see that there is a single point that has not yet been accounted for in $\overline{F}$, namely $(z_1,z_2)$ and $(z_3,z_4)$. This is the point $m_1$. 
\end{rmk}

\section{Computing the Quantum Steenrod Square for toric varieties}
\label{sec:computingqsstoric}

In this section, we will use the intersection definition of the quantum Steenrod square (Definition \ref{defn:singqss}). We will require that $\alpha_v: X_v \rightarrow M$ is an embedded submanifold for each $v$ (and not just a pseudocycle), and we will abusively replace $\alpha_v(X_v)$ by $X_v$.

		\subsection{Quantum Steenrod squares for $\mathbb{CP}^{n}$}
		\label{subsec:computingqsscpn}
		Let $x^{i}$ generate $H^{2i}(\mathbb{CP}^{n})$. By the quantum Cartan relation, Theorem \ref{thm:quancar}, $$Q\mathcal{S}(x^{i+1}) = Q\mathcal{S}(x^{i}) * Q\mathcal{S}(x) + q(W)(x^{i},x)$$ We can iteratively construct $Q\mathcal{S}(x^{i+1})$ as long as we know $q(W)(x^{i},x)$. Using a combination of degree reasons and Remark \ref{rmk:propertiesqs}, $Q\mathcal{S}(x) = x * x + xh^2$.

		\begin{lemma}
		\label{lem:qWcpn}
			For $2i<n$, $q(W)(x^i,x)=0$. For $n \le 2i \le 2n$,

			$$q(W)(x^{i},x) = {{i} \choose {n-i}} T h^{4i+2-2n}.$$
		\end{lemma}
		\begin{proof}

	Recall that we make a generic choice of $C^{n-1}_v \subset \mathbb{CP}^n$, parametrised by $v \in S^{\infty}$, such that $C^{n-1}_v$ represents $PD(x) \in H_{2(n-1)}(\mathbb{CP}^n)$ for each $v$. Similarly, we choose $C^{n-i}_v \subset \mathbb{CP}^n$ representing $PD(x^i) \in H_{2(n-i)}(\mathbb{CP}^n)$ for each $v \in S^{\infty}$. Observe that $q(W)(x^i,x)$ has degree $4i+4$ so, by Lemma \ref{lemma:lemqWtermscor}, for $i=1,...,n$ we deduce that:
\begin{equation}
\label{equation:qW}
\begin{array}{rcl}
q(W)(x^i, x) & = & \sum_{j=n-i}^{i} m^{i+1}_{j} x^{i+j-n} T h^{2(i+1)-2j}
\\[2em]

& = &
m_{n-i}^{i+1} x^0 T h^{4i+2-2n}
+
m_{n-i-1}^{i+1} x^1 T h^{4i-2n}
+
\cdots
+
m_i^{i+1} x^{2i-n} T h^2
\end{array}
\end{equation}
where $m^{i+1}_{j}$ are coefficients and the degrees are $|x| = 2$, so $|x^{i}| + |x| = 2i+2$ and $|T| = 2(n+1)$. Equation \eqref{equation:qW} follows for grading reasons.

We claim that $m^{i+1}_j$, the coefficient of $x^{i+j-n} T h^{2(i+1)-2j}$, is the number of (unparametrised) $J$-holomorphic spheres that intersect both $\mathbb{CP}^{{i+j}-n}$ and some representative of $PD(Sq^{2j}(x^i))$. We proceed in the following steps:

\begin{enumerate}[i)]
	\item Counting the coefficient of $x^{i+j-n} T h^{2(i+1)-2j}$ in $q(W)(x^i,x)$ is the same as counting setups as in Figure \ref{fig:qWcpn}(1) for $v \in D^{2i-2j, +}$ (recall $D^{2i-2j,+}$ corresponds to the $h^{2i-2j+2}$ term when defining $q(W)$ in Theorem \ref{thm:quancar}). Only $T^{1}$ appears in equation \eqref{equation:qW}, so one of the holomorphic bubbles has degree $0$ and the other degree $1$. For the solutions to be rigid, the sphere with the marked points $z_1, z_2$ must be constant (as in Example \ref{exmpl:cp1calc}). This yields the setup in Figure \ref{fig:qWcpn}(2).  

	\item Let $b$ be an element of the basis of cohomology, $\mathcal{B}$. The intersection of $C^{n-i}_{v}$ and $C^{n-i}_{-v}$ with some representative of $PD(b^{\vee})$, taken over all $v \in D^{2i-2j, +}$, is the coefficient of $b$ in $Sq^{2j}(x^{i})$ (by definition). 

\item Suppose that we neglect the intersections with $C^{n-1}_{\pm v}$ in Figure \ref{fig:qWcpn}(2). Then we count the number of (unparametrised) $J$-holomorphic spheres $u:S^2 \rightarrow M$ that intersect:
\begin{itemize}
\item a representative of $PD((x^{i+j-n})^{\vee}) = PD(x^{2n-(i+j)})$ (an example of which is a copy of $\mathbb{CP}^{i+j-n} \subset \mathbb{CP}^n$) and 
\item a representative of $PD(Sq^{2j}(x^{i}))$. 
\end{itemize} We recall that $Sq^{2j}(x^{i}) = {{i} \choose {j}} x^{i+j}$ (see \cite[Section 4.L.]{algtop}). This implies that $PD(Sq^{2j}(x^{i})) = {{i} \choose {j}} PD(x^{i+j})$. Recall that $PD(x^{i+j})$ is represented by a copy of $\mathbb{CP}^{n-(i+j)}$. Our problem reduces to asking how many lines there are intersecting $\mathbb{CP}^{(i+j)- n }$ and $\mathbb{CP}^{n - (i+j) }$, and multiplying this by the coefficient ${{i} \choose {j}}$.

However, this only makes sense if $i+j-n \ge 0$ and $n-(i+j) \ge 0$ (both of the representatives must be of nonnegative dimension), hence $j=n-i$. In particular, the representative of of $PD((x^{i+j-n})^{\vee})$ is a point, denoted $pt$. Further, there are a finite number (congruent to ${{i} \choose {n-i}} \text{ mod } 2$) of pairs $\{ v_k, -v_k \}$ such that $C^{n-i}_{v_k} \cap C^{n-i}_{-v_k} = pt_k \cong \mathbb{CP}^0$.  For each $pt_k$ there is exactly one line between $pt_k$ and $pt$ (i.e. there is always exactly one line between any two points in $\mathbb{CP}^n$).

\item  The homology class of each of the degree $1$ $J$-holomorphic spheres (the lines from the previous step) is the same homology class as that of $\mathbb{CP}^{1}$. Observe that $\mathbb{CP}^{1} \cap C^{n-1}_v = \{ pt'_v \}$ for each $v$. We make a generic choice of $C^{n-1}_v$ such that the $J-$holomorphic spheres are not contained in $C^{n-1}_{v}$ for generic $v$: specifically, we choose some hypersurface in $\mathbb{C}^{n-1}$ not containing the finite collection of lines from step iii, and then require that $C^{n-1}_v$ is a $C^2$-small perturbation in $v$ from this hypersurface (to ensure transversality). Then for each pair $\{ v_k , -v_k \}$, the intersection of the line at $pt'_{v_k}$ fixes the parametrisation of the $J$-holomorphic map, and the intersection of the line at $pt'_{-v_k}$ fixes which element $m$ of $W_0$ that we are using as the domain.
\end{enumerate}
Hence for each of the ${{i} \choose {n-i}}$ lines from step iii, there is exactly one choice of tuple $(m,u,v_k)$ (up to reparametrisation and the $\mathbb{Z}/2$-action) satisfying the configuration in Figure \ref{fig:qWcpn}(2).
		\end{proof}

		\begin{thm}
		\label{thm:SqQcpn}

For all $i \ge 0$,

\begin{equation}
\label{equation:quant1}
Q\mathcal{S}(x^{i}) = \sum_{j=0}^{i} \left( {{i} \choose {j}}+  \sum_{k=0}^{\lfloor n/2 \rfloor + 1} {{n-k}\choose{k}}\cdot {{i-(n+1-k)} \choose {j-k}} \right) x^{i+j} h^{2(i-j)},
\end{equation}

where $x^{i+j}$ is the $(i+j)$-th quantum power of $x$. 
		\end{thm}

Observe that if $i+j \ge n$ then $x^{i+j} = x^{i+j-n} T$, as this is the quantum power.

		Recall that $Q\mathcal{S}(x^i) = Sq(x^i) + T(...)$ where by Example \ref{exmpl:classcpn}: $$Sq(x^{i}) = \sum_{j=0}^{n-i} {{i} \choose {j}} x^{i+j} h^{2i-2j}.$$

		\begin{figure}
			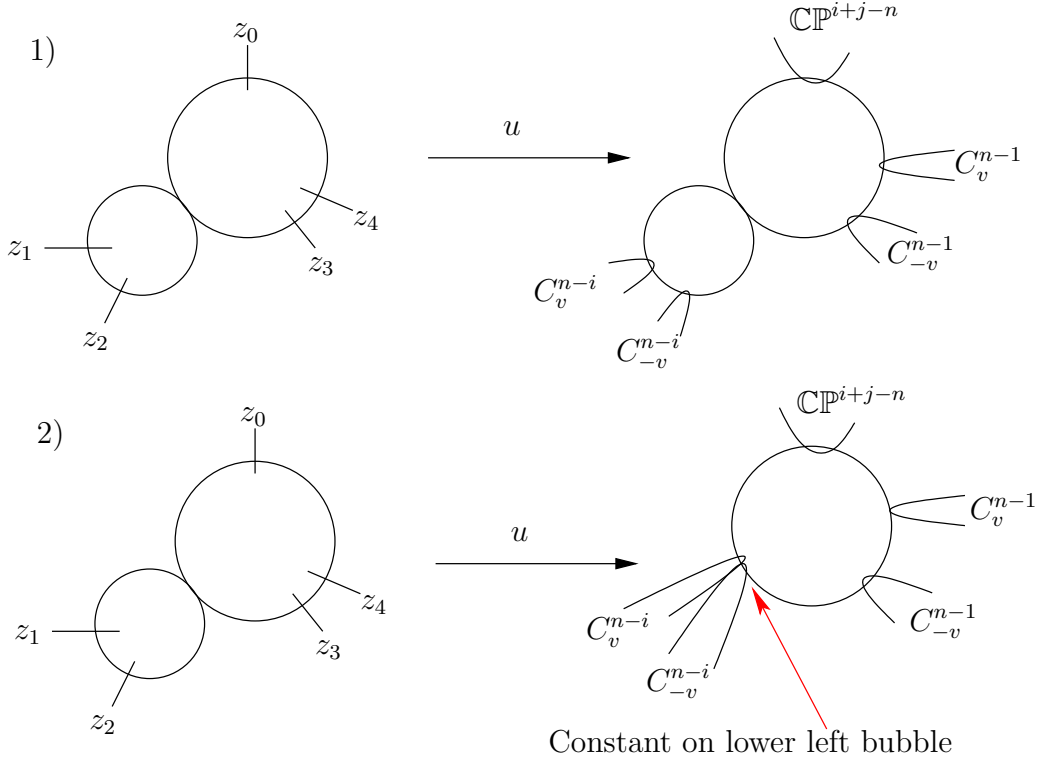
			\caption{Configurations for $q(W)(x^{i},x)$ for $\mathbb{CP}^{n}$}
			\label{fig:qWcpn}
		\end{figure}

	\begin{proof}[Proof of theorem \ref{thm:SqQcpn}]

		Since $T = x^{n+1}$, we can express the square as:

$$
\begin{array}{rcl}
Q\mathcal{S}(x^{i}) & = & \sum_{j=0}^{i} l^{i}_{j} x^{i+j} h^{2i-2j}
\\[2em]

& = &
l_0^i x^i h^{2i} + l_1^i x^{i+1} h^{2i-2} + \cdots
+ l_i^{i} x^{2i} h^{0}
\end{array}
$$
for some $l^i_j \in \mathbb{Z}/2$.

By the Quantum Cartan relation, Theorem \ref{thm:quancar} and Lemma \ref{lem:qWcpn}, the coefficients $l^i_j$ satisfy $l^{i+1}_{j} = l^{i}_{j} + l^{i}_{j-1}$ for $j \neq n-i$ and $l^{i+1}_{n-i} = l^{i}_{n-i-1} + l^i_{n-i} + {{i}\choose{n-i}}$ (the latter term arises from the quantum correction). Using a Pascal Triangle and the iterative formula for the $l^{i}_{j}$, one can write down the closed form solution.		
	\end{proof}

		In particular, truncating the sum in equation \eqref{equation:quant1} to $j \le n-i$ recovers the classical Steenrod square formula for $\mathbb{CP}^{n}$ from Example \ref{exmpl:classcpn}. This is because if $j \le n-i$ then every term in the second summation in Equation \eqref{equation:quant1} vanishes because either
\begin{itemize}
\item $j-k < 0$ 
\item or $i-(n+1-k) < 0$
\end{itemize}

To see that this is true, observe that if $j \ge k$ then $$i-(n+1-k) = k+i-n-1 \le j+i-n-1.$$ Then as $j \le n-i$, we see that $j+i-n-1 \le -1 < 0$

Explicit examples:
		
		\begin{enumerate}
		\item[$\mathbb{CP}^1$] :					
			$q(W)(x,x) = {{1}\choose{1-1}}T h^{4+2-2}$

			$Q\mathcal{S}(x) = xh^{2} + T$

			$Q\mathcal{S}(T) = (x h^2 + T)^2 + {\bf Th^4} = T^2$.

		\item [$\mathbb{CP}^2$] :
			$q(W)(x,x) = {{1}\choose{2-1}}T h^{4+2-4}$ 

			$Q\mathcal{S}(x) = xh^{2} + x^{2}$

			$Q\mathcal{S}(x^{2}) =(xh^{2} + x^{2})^2 + {\bf Th^2} =  x^{2}h^{4} + Th^{2} + xT$.

		\item[$\mathbb{CP}^3$] :
			$q(W)(x,x) = {{1}\choose{3-1}}T h^{4+2-6} = 0$ and 			$q(W)(x^2,x) = {{2}\choose{3-2}}T h^{8+2-6} = 0$ 

			$Q\mathcal{S}(x) = xh^{2} + x^{2}$

			$Q\mathcal{S}(x^{2}) = (xh^{2} + x^{2})^2 = x^{2}h^{4} + T$

			$Q\mathcal{S}(x^{3}) = (xh^{2} + x^{2})(x^{2}h^{4} + T) = x^{3}h^{6} + Th^{4} + xTh^{2} + x^{2}T$.

		\end{enumerate}

\begin{rmk}
Observe that, after one appeals to dimension reasons to rule out the other cases, the proof of Theorem \ref{thm:SqQcpn} only uses $GW(\mathbb{CP}^{n-1},\{pt \}, \{ pt' \} )$.
\end{rmk}

	\subsection{Fano Toric Varieties}
	\label{subsec:gentorvar}
		Let $M$ be a compact monotone toric manifold, with $b \in H^{|b|}(M)$ and $x \in H^{2}(M)$, and let $X = PD(x)$. Then analogously to Theorem \ref{thm:SqQcpn}, one proves Theorem \ref{thm:SqQtoric}.	

		\begin{proof}[Proof of Theorem \ref{thm:SqQtoric}]
		
Consider setups as in Figure \ref{fig:qTORIC}, which we henceforth call setups. These are configurations that, when counted, yield the coefficient of $c T^{c_1(\mu)} h^{i+2}$ in $q(W)(b,x_p)$. Henceforth we fix the dimension of the equivariant parameter space, $i \in \mathbb{Z}_{\ge 0}$ corresponding to $S^{i} \subset S^{\infty}$, and some $\mu \in H_2(M, \mathbb{Z})$ such that the $J$-holomorphic curves we consider represent $\mu$. We also fix $x \in H^2(M)$ and $b \in H^*(M)$, as in the statement of the theorem. We make choices of $X_v, B_v$ for $v \in S^{\infty}$, with the usual conditions for the input cycles used in $q(W)(b,x_p)$. Given a test output cycle $c \in H^*(M)$, we pick an embedded submanifold representing $PD(c^{\vee})$.

We will describe configurations that are related to Figure \ref{fig:qTORIC}, which we call {\it reduced setups}, which arise by neglecting the intersection with $X_v$ and the marked point $z_4$ corresponding to it. The setup as given remains dimension $0$: removing $z_4$ ``removes 2 dimensions", and removing the intersection with $X_v$ ``adds 2 dimensions".

A ``reduced setup" is a pair $(v,u_{red})$ such that $v \in S^{i}$ and $u_{red}: S^2 \vee_{1 \sim 0} S^2 \rightarrow M$. Then $v \in S^{i}$ and $u_{red}$ is $J$-holomorphic, and subject to $(u_{red})_*[S^2 \vee S^2] = \mu$. The map $u_{red}$ satisfies: $$u_{red}(0) \in PD(c^{\vee}), u_{red}(\infty) \in X_{-v}, u_{red}(1) \in B_v, u_{red}(\infty) \in B_{-v}.$$ Note that given a setup, we may obtain a reduced setup by forgetting the point $z_4$ (and the associated intersection condition). Observe that the space of setups and reduced setups is of the same dimension, $|c| + i + 2c_1(\mu) - 2|x| - 2|b|$.

We would like to prove that for a generic choice of $\{ X_v \}$, if $(v, u_{red})$ is a reduced setup, then for every $p \in S^2$ such that $u_{red}(p) \in X_v$:
\begin{itemize}
\item $u_{red}$ and $X_v$ intersect transversely in $M$ at $u_{red}(p)$, and
\item $p$ is an injective point of $u_{red}$.
\end{itemize}

Observe that if we are in the situation where the set of reduced setups is $0$-dimensional, i.e. $|c| + i + 2c_1(\mu) - 2|x| - 2|b|=0$, then we may assume that no intersections occur for $v \in \partial D^{i,+} = S^{i-1}$. Further, counting reduced setups with $v \in S^i$ and then quotienting by the free $\mathbb{Z}/2$-action of Equation \eqref{equation:actiononmoduli} is identical to simply restricting to $v \in \mathring{D}^{i,+}$ (without taking a quotient). With this in mind, we may freely perturb our choice of $X_v$ for $v \in \mathring{D}^{i,+}$, without changing the reduced setups. We make sure that the perturbation is sufficiently small that the moduli spaces of setups remains transverse. Then the argument becomes a classical argument that a generic perturbation of the embedded submanifold/pseudocycle $X_v$ will be transverse to $u_{red}$, and \cite[Proposition 1.3.1]{jhols} implies that the set of injective points of a simple curve $u_{red}$ is open and dense: hence, generically each intersection occurs at an injective point of $u_{red}$.

Now suppose that we are given a reduced setup $(v, u_{red})$. Then there are $\# (X_v \bullet \mu)$ (modulo $2$) setups corresponding to it. Observe that the actual number of corresponding setups is $\# (X_v \cap \mu)$, where $\cap$ is the absolute number of intersection points counted without signs. Generally such a count is not preserved under changes of representatives of $X_v$ and $\mu$, but one immediately sees that $\# (X_v \bullet \mu) = \# (X_v \cap \mu)$ for transversely intersecting pseudocycles in characteristic $2$. This choice of $\# (X_v \bullet \mu)$ setups corresponds to a choice of the marked point $z_4$ on the domain, which we know bijects with a choice of intersection points of $X_v$ and $\text{Im}(u)$ (as it is an injective point). 

In fact, setups and reduced setups are in a $1$ to $\#( X \bullet \mu)$ correspondence (recalling that $X = PD(x)$). This holds because one may pick $X_v$ such that every $X_v$ is a normal perturbation in a $C^2$-small tubular neighbourhood of some fixed submanifold representative $\euscr{X}$ of $X$ (argue likewise for a pseudocycle representative). This is then bordant to having chosen $X_v = \euscr{X}$ for all $v$, by deformation retracting the tubular neighbourhood to $\euscr{X}$. Hence $\# (X_v \bullet \mu) = \#( X \bullet \mu)$.

It is now sufficient to prove that reduced setups count $$\sum_{2i=0}^{|b|} \sum_{j \ge 1} \sum_{k=1}^{j}  \sum_{\mu \in H_2(M) : E(\mu) = k} \left( Q\mathcal{S}_{2i,j-k}(b) *_{\mu, k} x \right) \cdot h^{|b|-2i+2} T^{j}.$$ However, considering reduced setups alone one may choose $X_{-v}$ to be independent of $v \in D^{i,+}$ (again, choose a deformation retraction of a tubular neighbourhood to its core $\euscr{X}$). The result follows immediately from the definitions of $Q\mathcal{S}$ and the quantum product.
		\end{proof}

		\begin{figure}
			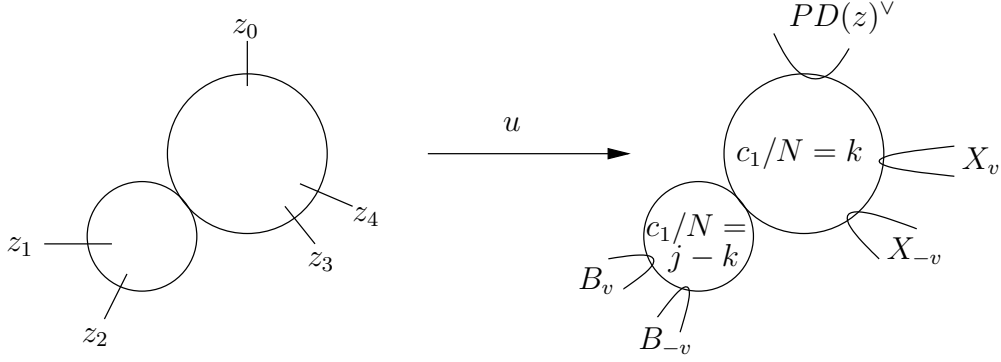
			\caption{Configurations for $q(W)(b,x_{p})$ for toric varieties, where $x$ and $b$ are the inputs and $c$ is the output. Here we are using the notation where $X_v$ and $B_v$ are embedded submanifold, as in Remark \ref{rmk:embeddedsubs}.}
			\label{fig:qTORIC}
		\end{figure}

		As the cohomology of a toric variety $M$ is generated by $H^{2}(M)$, iterated application of \eqref{equation:SqQtoric} yields a general solution, i.e. one can calculate $Q\mathcal{S}(x_{p_{1}}x_{p_{2}}...x_{p_{r}})$ assuming the base cases $Q\mathcal{S}(x_{p_{i}})$ for a basis $\{ x_{p} \}$ of $H^2(M)$ are known. Using a combination of degree reasons and Remark \ref{rmk:propertiesqs}, $Q\mathcal{S} (x_p) = x_p * x_p + x_p \cdot h^2$.

		\begin{proof}[Proof of Corollary \ref{corollary:fanotoricdecided}]
			We induct on degree. The base case is for $|x|=2$, and we know from above that $Q\mathcal{S}(x)= xh^2 + x * x$ is determined by $QH^*(M)$. Given $a \in H^{*}(M)$ for $* > 2$, write $a = b * x$ for $x \in H^{2}(M)$. By Theorem \ref{thm:quancar}, we have $Q\mathcal{S} (a) = Q\mathcal{S} (b) * Q\mathcal{S} (x) + q(W)(b,x)$. By induction $Q\mathcal{S} (b)$ and $Q\mathcal{S} (x)$ are determined by $QH^* (M)$, hence so is $Q\mathcal{S} (b) * Q\mathcal{S} (x)$. By Theorem \ref{thm:SqQtoric}, $q(W)(b,x)$ is determined by $QH^* (M)$ (observing that $\# (X \bullet \mu)$ is determined from singular cohomology). 
		\end{proof}

Let $\beta: QH^*(M) \rightarrow QH^*(M)$ be a ring homomorphism satisfying $\beta(T) = T$. In the notation of Theorem \ref{thm:SqQtoric}, we deduce that for $a,b \in QH^*(M)$, $$\beta(a *_{0,0} b ) = \beta(a) *_{0,0} \beta(b).$$ This is because $\mu =0$ is the only possible element of $H_2(M, \mathbb{Z})$ of Chern number $0$ when $M$ is monotone. Indeed, Theorem \ref{thm:SqQtoric} simplifies to state that if $|x| = 2$, then \begin{equation} \label{equation:toriccartanagain} Q \mathcal{S}( b * x) = Q \mathcal{S}(b) * Q \mathcal{S}(x)  + (Q \mathcal{S}(b)* x - Q \mathcal{S}(b)*_{0,0} x). \end{equation} Thus, any ring homomorphism $\beta$ with the given constraint is compatible with the quantum Steenrod square.

\begin{exmpl}[$\mathbb{CP}^{1} \times \mathbb{CP}^{1}$]
We let $x,y$ be the generators of $H^2(\mathbb{CP}^1 \times \mathbb{CP}^1)$, with $PD(x) = [\{ pt \} \times \mathbb{CP}^1]$ and $PD(y) = [\mathbb{CP}^1 \times \{ pt' \}]$ . Here $q(W_{0} \times D^{i-2,+})(x,y) = 0$ hence $Q\mathcal{S}(x) * Q\mathcal{S}(y) = Q\mathcal{S}(x*y)$. Indeed by equation \eqref{equation:SqQtoric},
$$q(W_{0} \times D^{i-2,+})(x,y) =\sum_{2i=0}^{2} \sum_{j \ge 1} \sum_{k=1}^{j} k \cdot Q\mathcal{S}_{2i,j-k}(x) *_{\mu, k} y h^{4-2i} T^{j},$$ recalling that $x *_{ \mu, k} y$ the coefficient of $T^{k}$ in the quantum product $x * y$, using spheres representing $\mu$. Working from definitions, $Q\mathcal{S} (x) = x h^{2} + T$. Then $\alpha *_{k} y \neq 0 \implies k = 1, \alpha = y$. There are no $i,j,k$ such that $Q\mathcal{S}_{2i,j-k}(x) = y$. Hence the sum on the right hand side is $0$.
\end{exmpl}

\section{The Quantum Adem Relations}
\label{sec:QAR}

		\subsection{Classical Adem Relations}
		\label{subsec:classadem}
		We begin with a discussion of the group cohomology of $S_4$ and $D_8$. This will involve adding details to the argument alluded to by Cohen-Norbury to prove the classical Adem relations in \cite[Section 5.2]{cohnor}.	

		It is proved in \cite[Sections IV.1, VI.1]{ademmilgram} that \begin{equation} \label{equation:HBD8} H^{*}(BD_{8}) = \mathbb{Z}/{2}[e,\sigma_{1},\sigma_{2}]/(e \sigma_{1}) \end{equation} where $e, \sigma_{1}$ are of degree 1 and $\sigma_{2}$ is of degree 2, and \begin{equation} \label{equation:HBS4} H^{*}(BS_{4}) = \mathbb{Z}/{2}[n_{1},n_{2},c_{3}]/(n_{1} c_{3}), \end{equation} where again subscripts denote the degree of the elements. Considering $$D_{8} = \langle (12),(34),(13)(24) \rangle \subset S_{4},$$ there are subgroups $$\mathbb{Z}/2 = \langle (13)(24) \rangle, \qquad \mathbb{Z}/2 \times \mathbb{Z}/2 = \langle (12),(34) \rangle.$$ Then $$H^{*}(B \mathbb{Z}/2) = \mathbb{Z}/{2}[e], \qquad H^{*}(B (\mathbb{Z}/2 \times \mathbb{Z}/2)) = \mathbb{Z}/{2}[x,y].$$ Consider the commutative diagram \eqref{commutativediagramofgroups} induced by the various inclusion maps of groups. As in \cite{ademmilgram}, one shows that:

\begin{equation}\label{commutativediagramofgroups}
\xymatrix{
H^*(B \mathbb{Z}/2)
\\
H^*(BD_8)
\ar@{->}^-{i_1}[u]
\ar@{->}_-{i_2}[d]
& 
H^*(BS_4)
\ar@{->}^-{j_1}[ul]
\ar@{->}^-{j_2}[dl]
\ar@{->}^-{\pi^*}[l]
\\
H^*(B(\mathbb{Z}/2 \times \mathbb{Z}/2))
}
\end{equation}

\begin{tabular}{l }
 $i_{1}(e) = e$\\
$i_{2}(\sigma_{1}) = x+y,  \quad i_{2}(\sigma_{2})=xy$\\
$j_{1}(n_{2}) = e^{2}$\\
$j_{2}(n_{1}) = x+y, \quad j_{2}(n_{2}) = xy$\\
\end{tabular}

All of the other generators map to $0$ via the $i,j$ maps. From this, and the fact that $\pi^*$ is injective, we deduce that $$\pi^{*}(n_{1}) = \sigma_{1} \qquad \pi^{*}(n_{2}) = \sigma_{2} + e^{2} \qquad \pi^{*}(c_{3}) = e \sigma_{2}.$$

			By Cohen-Norbury, \cite{cohnor}, there is a commutative diagram, namely diagram \eqref{classicalademdiagram}, where $qq^0$ satisfies $$qq^0(\alpha) = \sum_{p,q} Sq^{q} \circ Sq^{p}(\alpha) e^{|\alpha| + p - q} \sigma_{2}^{|\alpha| - p}.$$

\begin{equation}\label{classicalademdiagram}
\xymatrix{
H^*(M)
\ar@{->}^-{\hat{qq}^0}[r]
\ar@{->}_-{=}[d]
&
H^*(M) \otimes H^*(BS_4)
\ar@{->}^-{id_{H^*(M)} \otimes \pi^*}[d]
\\
H^*(M)
\ar@{->}^-{qq^0}[r]
&
H^*(M) \otimes H^*(BD_8)
}
\end{equation}

We do not in general know a closed form definition of $\hat{qq}^0$ in terms of compositions of Steenrod squares, but in fact we do not need to: the Adem relation is a purely algebraic relation, only using the fact that $qq^0$ lifts to a homomorphism $\hat{qq}^0$ (and not any information about the homomorphism itself). For a definition of $\hat{qq}^0$ in Diagram \eqref{classicalademdiagram}, use the $T^0$-component of $\hat{qq}$ from Definition \ref{defn:qs4}.

			\begin{Fact} 
			\label{Fact:commuteAdem}
				By Theorem 19 (Invariance) in \cite{cohnor}, the diagram \eqref{classicalademdiagram} commutes. This implies that the image of $qq^0$ lies in the image of $id_{H^{*}(M)} \otimes \pi^{*}$. Hence there are constraints on the image. Specifically, $e^{2i} \sigma_{2}^{j}$ may only appear in $qq^0(\alpha)$ if it arises from some $(e \sigma_{2})^{2k} (e^{2} + \sigma_{2})^{i+j-3k}$ for $k=0,1,...$, with coefficient ${i+j-3k}\choose{i-k}$. This is a special case of Lemma \ref{lemma:quantumademdiagram}.
			\end{Fact}

			\begin{lemma}	
\label{lem:bincoeff}		
				For any $s,m$, 
\begin{equation}
\label{equation:combinatorial}
{{3s+m}\choose{s+m}} = \sum_{l=0}^{\infty} {{m+l-1}\choose{2l}} {{3s+m}\choose{s-l}}
\end{equation}
modulo 2.
			\end{lemma}
			\begin{proof}
				We prove this by induction. Let $c(m,s) = {{m+3s}\choose{m+s}}$. Then modulo 2, $$c(m+2,s) = c(m,s) + c(m+3,s-1).$$

				Define $S(m,s) = \sum_{l}  {{m-1+l}\choose{2l}}{{3s+m}\choose{s-l}}$. Check that $S(m,s) = c(m,s)$ for $s=0,1$ and $m=1,2$. These are the base cases. Hence if $S(m+2,s) = S(m,s) + S(m+3s,s-1)$ for all $m,s$ then the lemma holds by induction. This is an exercise in binomial coefficient algebra modulo $\mathbb{Z}/2$. 
			\end{proof}

			\begin{thm}[Classical Adem Relations]
			\label{thm:car}
				Given $\alpha \in H^{*}(M)$ and $q,p>0$ such that $q<2p$,

				\begin{equation} \label{equation:classicalademrelations} Sq^{q}Sq^{p}(\alpha) = \sum_{k=0}^{[q/2]} {{p-k-1}\choose{q-2k}}Sq^{p+q-k} Sq^{k}(\alpha).\end{equation}
			\end{thm}
			\begin{proof}
				Suppose $q$ is even. Let $l = |\alpha|-p$, $m=p-q/2$, $n=q/2-k$, thus $$ {{p-k-1}\choose{q-2k}} = {{m+n-1}\choose{2n}}.$$ Assume $l=2r$. The cases for $q$ or $l$ odd are proven identically, except for slight modifications in the substitutions and the exponents of the labelled equations. Throughout, for $E \in H^*(B D_8)$, let $\text{cff}(E)$ be the coefficient of $E$ in $qq^0(\alpha)$. By definition of $qq^0$: $$Sq^{q} \circ Sq^{p} (\alpha)= \text{cff}(e^{l+2m} \sigma_{2}^{l}) \qquad \textrm{and} \qquad Sq^{p+1-k} \circ Sq^k (\alpha) =  \text{cff}(e^{l-2n} \sigma_{2}^{l+m-n}).$$ By Fact $1$ (which also ensures that the right hand sides of the following two equations are well defined), \begin{equation} \label{equation:cff1} \text{cff}(e^{l+2m} \sigma_{2}^{l})= \sum_{i=0}^{r}  {{3r+m-3i}\choose{r+m-i}} \cdot \text{cff}((e \sigma_{2})^{2i} (e^{2} + \sigma_{2})^{3r+m-3i})\end{equation} and \begin{equation} \label{equation:cff2} \text{cff}(e^{l-2n} \sigma_{2}^{l+m+n}) = \sum_{i=0}^{r} {{3r+m-3i}\choose{r-n-i}} \cdot \text{cff}((e \sigma_{2})^{2i} (e^{2} + \sigma_{2})^{3r+m-3i}).\end{equation}

The claim now follows since by Lemma \ref{lem:bincoeff}, $${{3r+m-3i}\choose{r+m-i}} = \sum_{n=0}^{\infty}  {{m+n-1}\choose{2n}}{{3r+m-3i}\choose{r-n-i}}.$$ Substitute this into Equation \eqref{equation:cff1}, swap the summation, and then substitute Equation \eqref{equation:cff2}. This yields Equation \eqref{equation:classicalademrelations}, after substituting back for $p,q$ and $k$.

The terms with $n > q/2$ will not appear in the final statement because $n > q/2$ implies $k < 0$, and $Sq^k(\alpha) = 0$ for $k<0$.
			\end{proof}

		\subsection{Quantum Adem Relations}
	\label{subsec:QAR}
			
		In this section we will denote by $\mathcal{B}$ some basis of $H^*(BS_4)$, by $\hat{\mathcal{B}}$ some basis of $H^*(BD_8)$ and by $\mathcal{B}_M$ some basis of $H^*(M)$.

			Recall that in Definition \ref{defn:qopn}, for $W \in H_{*}(\overline{M}_{0,5} \times_{\mathbb{Z}/2} S^i)$ for some $i$, we defined additive homomorphisms $$q_{i,j}(W): QH^a(M) \otimes QH^b(M) \rightarrow QH^{2a+2b-i-2jN}(M).$$ We will define a similar construction of operators that are parametrised by $H_{*}^{D_{8}} (\overline{M}_{0,5})$ and $H_{*}^{S_{4}} (\overline{M}_{0,5})$, where $$D_{8} = \langle (12),(34),(13)(24) \rangle \subset S_{4}$$ acts by permutations on the indices of $[z_{0},z_{1},z_{2},z_{3},z_{4}] \in \mathcal{M}_{0,5}$. We will abbreviate $P^{p,q,r}_{D_8} = \overline{M}_{0,5} \times_{D_8} ES_4^{p,q,r}$ and $P^{p,q,r}_{S_4} = \overline{M}_{0,5} \times_{S_4} ES_4^{p,q,r}$, recalling the constructions in Appendix \ref{sec:ed8es4}, where we expressed $ES_4$ as the union of a countable nested sequence of smooth closed manifolds, $ES_4^{p,q,r}$ of respective dimension $2p-1+3q+6r$. 

We note that for any $M$ with $H^*(M)$ finitely generated in all degrees, there is a map $$\Psi: H^*(M) \rightarrow H_*(M),$$ along with its inverse also denoted $\Psi: H_*(M) \rightarrow H^*(M)$. This is an isomorphism via universal coefficients (as usual working over $\mathbb{Z}/2$): explicitly, one picks a dual basis under the pairing $\langle \alpha, a \rangle \mapsto \alpha (a)$ given by evaluation of cocycles. For brevity we denote $P_{D_8} = \overline{M}_{0,5} \times_{D_8} ES_4$ and $P_{S_4} = \overline{M}_{0,5} \times_{S_4} ES_4$. The homology of $P_{D_8}$ and $P_{S_4}$ satisfy this finite generation condition: this is due to the Cartan Leray spectral sequence.

Pick a pseudocycle representative $\zeta^{\vee}: Z^{\vee} \rightarrow M$ for each $z^{\vee} \in \mathcal{B}_M^{\vee}$. For $\alpha \in H^*(M)$, choose pseudocycles $i_v: A_v \rightarrow M$ for $v \in ES_{4}$ (where $A_v = A \times \{v \} \subset A \times ES_4^{p,q,r}$ for some sufficiently large $p,q,r$). We do this such that $i_v A_v$ is a weak representative of $PD(\alpha)$ for each $v$, by which we mean that $i_v A_v \bullet X = PD(\alpha) \bullet X$ for all $X \in H_*(M)$, where $\bullet$ is the intersection number. We choose $i_v$ with invariance and genericity conditions as follows:
\begin{enumerate}
\item $A_v = A_{(23) \cdot v} = A_{(24) \cdot v}$ and $i_v = i_{(23) \cdot v} = i_{(24) \cdot v}$ for all $v \in ES_{4}$.
\item Let $\mathcal{M}_{0,5}(J,j)$ be the space of $J$-holomorphic maps of Chern number $jN$ from $S^2$ to $M$ with $5$ marked points. Let $\overline{\mathcal{M}_{0,5}(J,j)}$ be its compactification with stable nodal maps. Then the $i_v$ must be chosen sufficiently generically so that the intersection of the $S_4$-equivariant pseudocycles in Equations \eqref{pseudoaaa1} and \eqref{pseudoaaa2} is transverse: \begin{equation} \label{pseudoaaa1} \begin{array}{l} ev: \overline{\mathcal{M}_{0,5}(J,j)} \times ES_4^{p,q,r} \rightarrow M \times M \times M \times M \times M \times ES_4^{p,q,r} \\ (u,v) \mapsto (u(z_0), u(z_1), u(z_2), u(z_3), u(z_4), v) \end{array} \end{equation} and \begin{equation} \label{pseudoaaa2} \begin{array}{l} Z^{\vee} \times A \times A \times A \times A \times ES_4^i \rightarrow M \times M \times M \times M \times M \times ES_4^i \\ (x,a_1, a_2, a_3,a_4, v) \mapsto (\zeta(x), i_v(a_1), i_{(12) \cdot v} (a_2), i_{(13) \cdot v} (a_3), i_{(14) \cdot v} (a_4), v). \end{array} \end{equation} 
\end{enumerate}

Observe that we may restrict to the special case of Morse theory, as we have done throughout this paper. Specifically we choose $f_{v,s}$ for $v \in ES_4$ and $s \in [0,\infty)$. We do this such that $f_{v,s} = f$ for $s \gg 0$, and $f_{(23) \cdot v, s} = f_{(24) \cdot v, s} = f_{v, s} $ for all $v,s$, and we replace the incidence condition with $i_{(1p) \cdot v} (a_p)$ by incidence with a $-\nabla f_{(1 p) \cdot v,s}$-flowline asymptotic to a critical point $\alpha$.

\begin{defn} 
\label{defn:s4operators}
Let $\alpha \in \text{crit}(f)$. For $d \in H^{*}_{S_{4}}(\overline{M}_{0,5})$, we pick a pseudocycle representative $\delta: D \rightarrow P^{p,q,r}_{S_4}$ of $\Psi (d) \in H_* (P^{p,q,r}_{S_4})$ (for some sufficiently large $p,q,r$). Then we define an operation $q_{S_4}(D): H^*(M) \rightarrow QH^*(M),$ by $$q_{S_4}(D)(\alpha) := \sum_{z \in \text{crit}(f), \ j \ge 0} n_{z, \alpha, j} \cdot z \cdot T^j,$$ where $n_{z, \alpha,j}$ counts the number of $S_4$ equivalence classes of triples $(m,u,v)$ with $[m,v] \in \delta(D) \subset P_{S_4}$ and $u: m \rightarrow M$ is $J$-holomorphic and of Chern number $j N$. We also require that $$u_0: (-\infty,0] \rightarrow M, \ u_p: [0,\infty) \rightarrow M,$$ for $p=1,2,3,4$, such that $$\begin{array}{l} \partial_t u_0(s) = -\nabla f(u_0(s)) , \ \partial_t u_p(s) = -\nabla f_{(1 p) \cdot v}(u_p(s)), \\ u(z_p) = u_p(0), \ u_0(-\infty) = z, \text{ and } u_p(\infty) = \alpha. \end{array}$$ 
On cohomology the operation will be independent of the representative $\delta$ of $\Psi (d)$, by the same proof as Lemma \ref{lemma:additiveandindep} (i.e. we express our coefficients as the intersections of pseudocycles, and bordant pseudocycles give the same intersection number).
			\end{defn}			

In order to show that Definition \ref{defn:s4operators} is well defined on cohomology, we must define an operation $q_{S_4}'(D) : C^*_{S_4}(M \times M \times M \times M) \rightarrow C^*(M) \otimes C^*(BS_4)$ as in Definition \ref{defn:mss}. It is then a standard compactification theorem to prove that $q_{S_4}$ is well defined on cohomology (as $D$ is closed), for example as in Equation \ref{equation:Sq'chainmap}. 

			The definition of $q_{D_8}(D)$ is identical to Definition \ref{defn:s4operators}, replacing everywhere $S_4$ by $D_8$ (note specifically that this definition uses $BD_8 = ES_4 / D_8$ as its parameter space. 

Henceforth, we will restrict to the subalgebra $H^*(BS_4)_{red}$ of $H^*(BS_4)$ generated by $\sigma_2$ and $e$, and similarly the subalgebra $H^*(BD_8)_{red}$ of $H^*(BD_8)$ generated by $n_2$ and $c_3$. The map $\pi^*: H^*(BS_4)_{red} \rightarrow H^*(BD_8)_{red}$ is well defined and injective because of Diagram \eqref{commutativediagramofgroups}. Indeed, the only difference to using $H^*(BS_4)$ and $H^*(BD_8)$ is that we we forget all additive generators that include monomials with some nontrivial $\sigma_1$ and $n_1$ exponent respectively.

As in the case of the quantum Cartan relation, we would like to consider cycles in $H_*(P_{S_4})$ parametrised by some basis $\mathcal{B}$ of $H^*(BS_4)_{red}$. Compare this to the proof of the quantum Cartan relations, where the classes $[\{ m_1 \} \times D^{i,+}] \in H_*( P_{\mathbb{Z}/2})$ were parametrized by $[D^{i,+}] \in H_*(B \mathbb{Z}/2) = H_* (\mathbb{RP}^{\infty})$. Further, we will show later that \begin{equation} \label{equation:quantumsquarecomposition} Q\mathcal{S} \circ Q\mathcal{S} (\alpha) = \sum_{i,j} q_{D_8} (\{ m_1 \} \otimes \Psi (e^i \sigma_2^j) )(\alpha) \cdot e^i \sigma_2^j. \end{equation}

Hence, ideally we would like the chains represented by $\{ \{ m_1 \} \otimes B \}$ to be elements of $H_*(P_{S_4})$, for $B \in \mathcal{B}$. This will not work because $m_1$ is not $S_4$-invariant. However, the cycle $m_1 + g m_1 + g^2 m_1 \in H_*(\overline{M}_{0,5})$ is $S_4$ invariant, where $g = (123)$ generates the cosets of $D_8$ in $S_4$ (note that $g m_1 = m_2$). 

\begin{defn}
\label{defn:qs4}
Given a basis $\mathcal{B}$ of $H^*(BS_4)_{red}$, define: $$\hat{qq}: H^*(M) \rightarrow QH^*(M) \otimes H^*(BS_4)_{red},$$ $$\hat{qq}(\alpha) = \sum_{b \in \mathcal{B}} q_{S_4}((m_1 + g m_1 + g^2 m_1) \otimes \Psi (b)) (\alpha) \cdot b.$$
\end{defn}

\begin{defn}
Given a basis $\tilde{\mathcal{B}}$ of $H^*(BD_8)_{red}$, define: $$qq: H^*(M) \rightarrow QH^*(M) \otimes H^*(BD_8)_{red},$$ $$qq(\alpha) := \sum_{\tilde{b} \in \tilde{\mathcal{B}}} q_{D_8}((m_1 + g m_1 + g^2 m_1) \otimes \Psi (\tilde{b})) (\alpha) \cdot \tilde{b}.$$
\end{defn}

		We fix some additive basis $\mathcal{B}$ for $H^{*}(BS_{4})_{red}$, of the form $\{ n_2^a c_3^q \}$ (with notation as in Equation\eqref{equation:HBS4}). Recall from Diagram \eqref{classicalademdiagram} there is $\pi_* : H_*(BD_8)_{red} \rightarrow H_*(BS_4)_{red}$ and $\pi^* : H^*(BS_4)_{red} \rightarrow H^*(BD_8)_{red}$, which are induced by the continuous quotient map $$\pi: ES_4 / D_8 \rightarrow ES_4 / S_4.$$ We also define: $$i_* : H_*(BS_4)_{red} \rightarrow H_*(BD_8)_{red}, \qquad i_*(D) = D + gD + g^2 D$$ and $$i^* : H^* (BD_8)_{red} \rightarrow H^* (BS_4)_{red}, \qquad i^* (d) = d + g d + g^2 d.$$ As we work over $\mathbb{Z}/2$ we see that $\pi_* \circ i_* = id$ and $i^* \circ \pi^* = id$, which also shows that $\pi^*$ is injective. As $\pi^*$ is injective, $\pi^* \mathcal{B}$ is linearly independent in $H^*(BD_8)_{red}$. We extend this to a basis $\hat{\mathcal{B}} = \pi^* \mathcal{B} \cup \mathcal{B}'$ of $H^*(BD_8)_{red}$. 

\begin{lemma}
\label{lemma:adem1}
$$q_{S_4}((m_1 + g m_1 + g^2 m_1) \otimes \pi_* \Psi (b)) = q_{D_8}((m_1 + g m_1 + g^2 m_1) \otimes \Psi (b)).$$
\end{lemma}
\begin{proof}
	Suppose that we pick some pseudocycle representative $f: X \rightarrow BD_8$ of $\Psi(b) \in H_*(BD_8)_{red}$ (or specifically, some stratum $BD_8^{p,q,r}$). To define a pseudocycle representative of $\pi_* \Psi(b) \in H_*(BS_4)_{red}$, we choose $\pi \circ f$. So in particular, there is a pseudocycle representative of $(m_1 + g m_1 + g^2 m_1) \otimes \Psi (b)$ of the form $$f': \{pt_1,pt_g,pt_{g^2} \} \times X \rightarrow PD_8, \ f'(pt_a, x) = (a \cdot m_1, f(x)),$$ which we see descends to a $D_8$-equivariant pseudocycle, and similarly an $S_4$-equivariant pseudocycle: $$\pi \circ f': \{pt_1,pt_g,pt_{g^2} \} \times X \rightarrow PS_4, \ \pi \circ f'(pt_a, x) = (a \cdot m_1,\pi \circ f(x)).$$

Let $z \in \text{crit}(f)$. Let $\overline{\mathcal{M}}(J,j)$ be a partial compactification of the space of genus $0$ stable $J$-holomorphic maps (i.e. excluding repeated or multiply covered components). Recall from Lemma \ref{lemma:additiveandindep} the means by which we determine the coefficient of $z$ in $q_{S_4}((m_1 + g m_1 + g^2 m_1) \otimes \pi_* \Psi (b))(x)$ as an intersection number. One defines a $5$-pointed Gromov-Witten invariant assigned to $\overline{\mathcal{M}}(J,j)$. Push this forwards along the map $$\mathcal{W}: \overline{\mathcal{M}}(J,j) \times ES_4^{p,q,r} \rightarrow M \times ((M \times M \times M \times M \times \overline{M}_{0,5}) \times_{S_4} ES_4^{p,q,r})$$ (for some $p,q,r$), which is induced by the evaluation map on the five marked points, the stabilisation map $\overline{\mathcal{M}}(J,j)  \rightarrow \overline{M}_{0,5}$ and the identity on the $ES_4^{p,q,r}$ factor. This determined a cohomology class in $M \times ((M \times M \times M \times M \times \overline{M}_{0,5}) \times_{S_4} ES_4^{p,q,r})$. There is also a pseudocycle constructed using the evaluation maps on the partially compactified (un)stable manifolds $W^u(z,f)$, $W^s(x,f_{v},s)$, $W^s(x,f_{(12) \cdot v},s)$, $W^s(x,f_{(13) \cdot v},s)$, $W^s(x,f_{(14) \cdot v},s)$, alongside the map $\pi \circ f'$. The intersection of the image of the (equivariant) Gromov-Witten invariant with the pseudocycle provides the coefficient of $z$. A similar argument holds for $q_{D_8}$, this time using the map $f'$. Then $$M \times ((M \times M \times M \times M \times \overline{M}_{0,5}) \times_{D_8} ES_4^{p,q,r}) \rightarrow M \times ((M \times M \times M \times M \times \overline{M}_{0,5}) \times_{S_4} ES_4^{p,q,r})$$ is a $3$-to-$1$ covering, hence the coefficients of this intersection differ by multiplying the $S_4$-coefficient by three. As we work over $\mathbb{Z}/2$-coefficients, multiplication by three is the identity.

Ensuring transversality for pseudocycles in both the base and the cover simulataneously is not an issue, as the property of an intersection being transverse is preserved under a $p$-fold smooth covering map (being a local diffeomorphism), so it suffices to ensure transversality on the cover, $M \times ((M \times M \times M \times M \times \overline{M}_{0,5}) \times_{D_8} ES_4^{p,q,r})$, which we know by Appendix \ref{subsec:quantumademrels}.
\end{proof}

\begin{lemma}
		For $b' \in \mathcal{B}' := \hat{\mathcal{B}} - \pi^* \mathcal{B}$, $$q_{D_8} ((m_1 + g m_1 + g^2 m_1) \otimes \Psi (b')) = 0.$$
\end{lemma}
\begin{proof}
		By Lemma \ref{lemma:adem1}, $$q_{D_8} ((m_1 + g m_1 + g^2 m_1) \otimes \Psi (b')) = q_{S_4} ((m_1 + g m_1 + g^2 m_1) \otimes \pi_* \Psi (b')).$$ If $B' = \Psi (b')$ then for all $b \in \mathcal{B}$, $$\langle b, \pi_* \Psi(b') \rangle  = \langle b, \pi_* B' \rangle = \langle \pi^* b, B' \rangle = \langle \pi^* b, \Psi(b') \rangle = 0$$ by definition of the dualising isomorphism $\Psi$. Hence $\pi_* \Psi (b') = 0$.
\end{proof}

This implies that \begin{equation} \label{equation:qq} qq(\alpha) := \sum_{\pi^* b \in \pi^* \mathcal{B}} q_{D_8}((m_1 + g m_1 + g^2 m_1) \otimes \Psi (\pi^* b)) (\alpha) \cdot \pi^* b. \end{equation}

\begin{lemma}
\label{lemma:quantumademdiagram}
The following diagram commutes:

\begin{equation} \label{quantumademdiagram}
\xymatrix{
H^*(M)
\ar@{->}^-{\hat{qq}}[r]
\ar@{->}_-{=}[d]
&
QH^*(M) \otimes H^*(BS_4)_{red}
\ar@{->}^-{id_{H^*(M)} \otimes \pi^*}[d]
\\
H^*(M)
\ar@{->}^-{qq}[r]
&
QH^*(M) \otimes H^*(BD_8)_{red}
}
\end{equation}
\end{lemma}
\begin{proof}
	Observe that $$(id \otimes \pi^*) \hat{qq}(\alpha) = \sum_{\pi^* b \in \pi^* \mathcal{B}} q_{S_4} ((m_1 + g m_1 + g^2 m_1) \otimes \Psi (b)) (\alpha) \cdot \pi^* b.$$ Then $$qq (\alpha) = \sum_{\pi^* b \in \pi^* \mathcal{B}} q_{S_4}((m_1 + g m_1 + g^2 m_1) \otimes \pi_* \Psi (\pi^* b)) (\alpha) \cdot \pi^* b$$ using Equation \eqref{equation:qq} and Lemma \ref{lemma:adem1}. For $b \in \mathcal{B}$, let $D = \Psi (\pi^* b)$. Then $\langle \pi^* b, D \rangle = 1$ and $\langle \hat{b}, D \rangle = 0$ for all $\hat{b} \in \hat{\mathcal{B}} - \pi^* b$, specifically for $\hat{b} = \pi^* d$ with $d \in \mathcal{B} - b$. Hence $\langle d, \pi_* D \rangle = 0$ for $d \in \mathcal{B} - b$ and $\langle b, \pi_* D \rangle = 1$, so $\pi_* \Psi (\pi^* b) = \Psi (b)$ by definition of $\Psi$.
\end{proof}

			We now pick a different basis $\tilde{\mathcal{B}}$ for $H^*(BD_8)_{red}$ (i.e. different from $\hat{B}$) consisting of elements of the form $e^i \sigma_2^j$ (see Section \ref{subsec:classadem} to recall the notation). Let \begin{equation} \label{equation:qqijdef}qq_{i,j}(\alpha) := q_{D_8}((m_1 + g m_1 + g m_1) \otimes \Psi (e^{i} \sigma_{2}^{j})), \end{equation} the coefficient of $e^{i} \sigma_{2}^{j}$ in $qq(\alpha)$. 

			\begin{proof}[Proof of Theorem \ref{thm:QAR}, The Adem Relations]
			The theorem follows immediately by Lemma \ref{lemma:quantumademdiagram} and the combinatorial argument in Theorem \ref{thm:car}.
			\end{proof}

			We relate this to a composition of quantum Steenrod squares. Firstly, we observe that instead of using $BD_8$ as our parameter space, we will use the spaces $E$ and $B = E / D_8$ as defined in Appendix \ref{subsec:ed8}. Recall from Appendix \ref{subsec:es4}, there is a map $\rho: ES_4 \rightarrow \mathbb{S}_2 = E$, a space that is an $ED_8$ (i.e. a contractible space with a free $D_8$-action). 

Recall then that $E$ is stratified by finite dimensional $D_8$-invariant submanifolds $E^{i,j} = S^i \times (S^j \times S^j)$. Let $D^{j,+}$ be the upper $i$-dimensional hemisphere as usual. It is immediate that $D^{i,+} \times D^{j,+} \times D^{j,+} \subset E$ represents a closed cycle in $H_*(BD_8)$. 

\begin{lemma}
\label{lemma:submandidjdj}
The submanifold $D^{i,+} \times D^{j,+} \times D^{j,+}$ represents $\Psi (e^i \sigma_2^j)$
\end{lemma}
\begin{proof}
Consider the projections $$k_1: E \rightarrow S^{\infty}, \ (x,(x_1,x_2)) \mapsto x$$ and $$k_2: E \rightarrow S^{\infty} \times S^{\infty}, \ (x,(x_1,x_2)) \mapsto (x_1,x_2),$$ which are respectively $\mathbb{Z}/2 \cong \langle (13)(24) \rangle$- and $\mathbb{Z}/2 \times \mathbb{Z}/2 \cong \langle (12), (34) \rangle$-equivariant. Indeed, they induce respectively $\mathbb{Z}/2$- and $\mathbb{Z}/2 \times \mathbb{Z}/2$-equivariant homotopy equivalences for the same reason as the map at the end of Appendix \ref{subsec:es4}. We abusively denote by $k_1,k_2$ the maps after quotienting by the free $\mathbb{Z}/2$-action. Combining these with the quotient maps $$l_1: E /( \mathbb{Z}/2) \rightarrow B$$ and $$l_1: E / (\mathbb{Z}/2 \times \mathbb{Z}/2) \rightarrow B,$$ we obtain $i_p$ from Diagram \ref{commutativediagramofgroups} as the composition $i_p = l_p^* \circ (k_p^*)^{-1}$, for $p=1,2$. 

If $j=0$ then observe that $i_1(e^i) = e^i$ using Diagram \ref{commutativediagramofgroups}. Notice that for the homogeneous choice of homology basis, $\Psi(e^i)$ in $B \mathbb{Z}/2$ is represented by $D^{i,+} \subset S^{\infty}$. Letting $(i_{p})_{*}: H_*(B \mathbb{Z}/2) \rightarrow H_*(B)$ be the pushforward, observe that $$\Psi(e^i) = (i_{1})_{*} \Psi (i_{1})^{*}(e^i) = (i_{1})_{*} \Psi(e^i) = (i_{1})_{*}([D^{i,+}]) = [D^{i,+} \times D^{0,+} \times D^{0,+}] \in H_*(B).$$ Hence the result holds for $j=0$. Similarly the result holds for $i=0$, using the homomorphism $i_2$ and replacing $e$ by $\sigma_2^j$ and $D^{i,+}$ by $D^{j,+} \times D^{j,+}$, we see that $\Psi(\sigma_2^j)$ is respresented by $[D^{0,+} \times D^{j,+} \times D^{j,+}]$. 

Observe, via the K\"unneth isomorphism, that elements of $H^*(B)$ may be represented as $D_8$-equivalence classes of cochains $x \otimes y \otimes z$, where $x \in C^*(S^{\infty})$. By the previous paragraph, we know that $e^i \sigma_2^j$ is represented by $x_i \otimes x_j \otimes x_j$, where $x_i$ is the indicator homomorphism for a simplex representing $D^{i,+}$. Then $\Psi([x_i \otimes x_j \otimes x_j]) = [D^{i,+} \times D^{j,+} \times D^{j,+}]$ as required.
\end{proof}

We reinterpret the operation $qq$ in terms of the parameter space $B = E / D_8$ (as opposed to $BD_8 = ES_4 / D_8$). Observe that the space $E$ does not have an action of any element of $S_4 - D_8$. Hence, in Definition \ref{defn:s4operators} we no longer ask for our Morse function $f_{v,s}$ to have invariance under $(23),(24)$ (as this is meaningless). Further, we choose the $f_{v,s}$ that we use for incidence conditions to be respectively $f_{v}, f_{(12) \cdot v}, f_{(13)(24) \cdot v}, f_{(14)(23) \cdot v}$. (Observe that when we used the parameter space $ES_4$, the invariance conditions imply that $f_{(13)(24) \cdot v} = f_{(13) \cdot v}$ and $f_{(14)(23) \cdot v} = f_{(14) \cdot v}$). 

For the following proof, we fix $\alpha \in \text{crit}(f)$.

			\begin{proof}[Proof of Corollary \ref{corollary:QAR}]
				The corollary follows from Theorem \ref{thm:QAR} if we prove that for each $i,j \in \mathbb{Z}_{\ge 0}$ \begin{equation} \label{equation:corollaryequation} q_{D_8}(m_1  \otimes \Psi (e^{i} \sigma_{2}^{j}))(\alpha) = \sum_{b,d} Q\mathcal{S}^{i,b} \circ Q\mathcal{S}^{j,d}(\alpha). \end{equation} Henceforth we will fix some choice of $b,d$, and count those contributions to $q_{D_8}$, ostensibly denoted $q_{D_8,b,d}$, which arise from counting configurations with nodal curves with three components, corresponding to a nodal sphere comprised of a sphere of Chern number $bN$ attached to two spheres of Chern number $dN$ at $1$ and $\infty$. As we have now fixed $b,d$, we abusively exclude them from the notation.

To prove Equation \eqref{equation:corollaryequation}, we use a similar idea to proving the Cartan relation, as is illustrated in Figure \ref{fig:ademmodulispace}. Specifically, recall the $1$-dimensional space of graphs $T^c$ from Section \ref{subsec:Cartan}. Recall from Lemma \ref{lemma:lem1} that for each $t \in [0,\infty]$ there is a space $|t|_Q$, consisting of three copies of $S^2$ with semi-infinite or finite lines attached at $0,1,\infty$ (with the length of the finite edges being $t$ for $t \in [0,\infty)$). Associated to each $t \in [0,\infty)$, we define $f_{v,s,t}$ and $g_{v,s,t}$ for $v \in E$ and for $s \in [0,\infty)$ or $[0,t]$ respectively, and $t \in T$. We also choose a perturbation $f_s$ for $s \in (-\infty, 0]$. We choose these such that:
				\begin{enumerate}
					\item $f_{s} = f$ for $s \le -1$.
					\item $f_{(14)(23) \cdot v, s, t} = f_{(13)(24) \cdot v, s, t} = f_{v,s,t}$ for $t \ge 1$,
					\item $g_{(12) \cdot v,s,t} = g_{(34) \cdot v,s,t} = g_{v,s,t}$ for all $v,s,t$.
					\item $f_{v,s,t} = f$ for $s \ge 1$, for all $v,t$.
					\item $g_{v,s,t} = f$ when $t \ge 1$ and $s \ge 1$.
				\end{enumerate}

These conditions are analogues of those made in Definition \ref{defn:s4operators} (here adapted to the $D_8$ case).

		\begin{figure}
			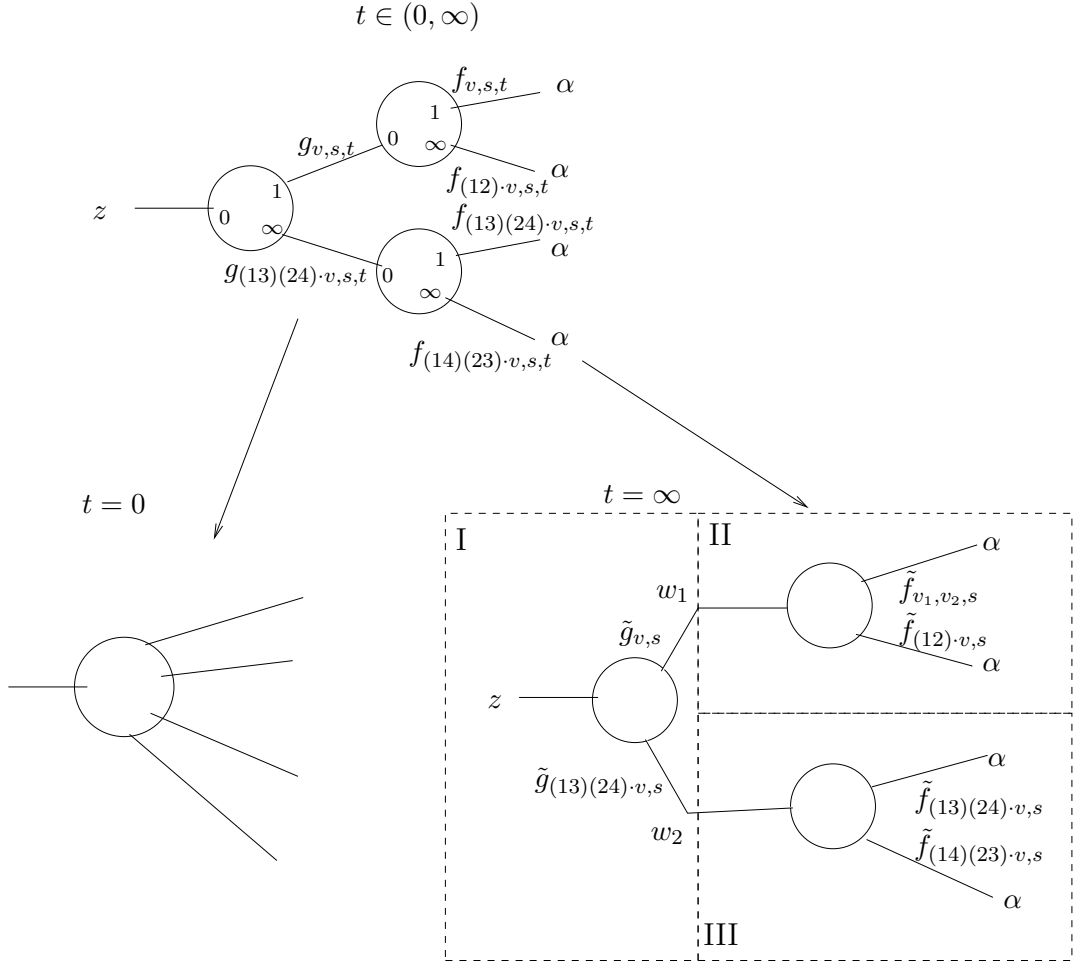
			\caption{Moduli space for the Adem relations.}
			\label{fig:ademmodulispace}
		\end{figure}

		Similarly to the case of the quantum Cartan relation, we define for each $t \in [0,\infty]$ a moduli space $\tilde{\mathcal{M}}_t(\alpha, z)$, consisting of pairs $(v,u)$ such that $v \in S^{i,+} \times S^{j,+} \times S^{j,+}$ (see Lemma \ref{lemma:submandidjdj}) and $u: |t|_Q \rightarrow M$ such that $u$ is $J$-holomorphic, with the Chern number on each sphere being as fixed at the beginning of the proof, and edge and asymptotic conditions as in Figure \ref{fig:ademmodulispace}. There is the previously given $D_8$-action on $E^{i,j} = S^{i,+} \times S^{j,+} \times S^{j,+}$, and $D_8$ also acts by permutations on $\overline{M}_{0,5}$ (which induces an action on the moduli space of $u: |t|_Q \rightarrow M$). Together these yield a $D_8$-action on $\tilde{\mathcal{M}}_t(\alpha, z)$, and we write $\mathcal{M}_t(\alpha, z) = \tilde{\mathcal{M}}_t(\alpha, z) / D_8$. We let $\mathcal{M}(\alpha, z) = \sqcup_{t \in T^c} \mathcal{M}_t(\alpha, z)$, which is a smooth $1$-dimensional manifold (establishing transversality is a modification of Appendix \ref{subsec:bordismquantumcartantrans} using the considerations of Appendix \ref{subsec:quantumademrels}, so we omit a restatement here).

Note that it is immediate from the definition that the $t=0$ boundary corresponds to $q_{D_8}(m_1  \otimes \Psi (e^{i} \sigma_{2}^{j}))(\alpha)$ (i.e. when the output is $z \in \text{crit}(f)$, this yields the coefficient of $z$ in $q_{D_8}(m_1  \otimes \Psi (e^{i} \sigma_{2}^{j}))(\alpha)$). Hence it remains to prove that the $t= \infty$ boundary yields the coefficient of $z$ in $Q\mathcal{S}^{i,b} \circ Q\mathcal{S}^{j,d}(\alpha)$. 

We observe that from our choice of $g_{v,s,t}$, we may ensure that $g_{v,s,\infty}$ depends only on the first summand $S^{\infty}$ of $E = S^{\infty} \times S^{\infty} \times S^{\infty}$, which we denote $\tilde{g}_{v,s}$ for $v \in S^{\infty}$. Similarly, we may ensure that $f_{v,s,\infty}$ depends only on the second two summands, denoted $\tilde{f}_{v_1,v_2,s}$ for $(v_1,v_2) \in S^{\infty} \times S^{\infty}$. Then let $S$ be the space obtained by attaching $(-\infty,0]$ to $S^2$ at $0$, and two copies of $[0,\infty)$ to $S^2$ at $1$ and $\infty$ respectively. Let $R : S \rightarrow S$ be the involution that is $z \mapsto z/(z-1)$ on $S^2$, swapping the positive half-lines and fixing the negative half-line. 

The configurations for the $t=\infty$ end then decouple. Specifically, if $v = (v_0,(v_1,v_2)) \in E$ then pairs $(v,u) \in \mathcal{M}(\alpha,z)$ correspond to two tuples as follows:
\begin{itemize}
\item A pair $(v_0,u)$ such that $v_0 \in S^{i}$, $u: S \rightarrow M$ such that $u$ is $J$-holomorphic of Chern number $bN$, satisfying conditions in Figure \ref{fig:ademmodulispace}$(I)$.
\item A four-tuple $(v_1,v_2,u_1, u_2)$ such that $(v_1,v_2) \in S^{j} \times S^{j}$, and $u_p: S \rightarrow M$ for $p=1,2$ such that $u_p$ is $J$-holomorphic of Chern number $dN$, satisfying conditions in Figure \ref{fig:ademmodulispace}$(II), (III)$ respectively.
\end{itemize}
The $D_8$ action on these pairs is as follows (recalling that the invariance conditions on $\tilde{f}$ ensures that $\tilde{f}_{v_1,v_2,t} = \tilde{f}_{-v_1,-v_2,t}$ for any $v_1,v_2 \in S^{\infty}$):
\begin{itemize}
\item $\begin{array}{l} (12) \cdot (v_0,u) = (v_0,u), \\ (34) \cdot (v_0,u)  = (v_0,u), \\ (13)(24) \cdot (v_0,u) = (-v_0, u \circ R). \end{array} $ \\
\item $\begin{array}{l} (12) \cdot (v_1,v_2,u_1,u_2) = (-v_1,v_2, u_1 \circ R, u_2 \circ R), \\ (34) \cdot (v_1,v_2,u_1, u_2)  = (v_1,-v_2,u_1 \circ R, u_2 \circ R), \\ (13)(24) \cdot (v_1,v_2,u_1, u_2) = (v_2,v_1, u_2, u_1). \end{array}$
\end{itemize}

To begin with, we see that counting the pair $(v_0, u)$ such that $v_0 \in S^{\infty}$, modulo the $D_8$ action, is exactly the coefficient of $(Q \mathcal{S}')^{i,b}(A \otimes B(x))$, where the operation $A \otimes B: QH^*(M) \rightarrow QH^*(M) \otimes QH^*(M)[h]$ is determined by counting configurations corresponding to the $(v_1,v_2,u_1,u_2)$ above  (with $(v_1, u_1)$ determining the $A$ component and $(v_2,u_2)$ determining the $B$ component), and $Q \mathcal{S}'$ is recalled from Definitions \ref{defn:mss} and \ref{defn:mqss}.

In fact, we only need to count solutions where $v_1 = v_2$ and $u_1 = u_2$. Firstly, (from Figure \ref{fig:ademmodulispace}) consider contributions from using the intermediate critical points $w_1 \neq w_2$. Then if $(v_0, u), (v_1, v_2,u_1,u_2)$ contributes to a term of the form $(Q \mathcal{S}')^{i,b}(w_1 \otimes w_2 T^d)$, it must be that $(v_0, u), (v_2, v_1,u_2,u_1)$ contributes to $(Q \mathcal{S}')^{i,b}(w_2 \otimes w_1 T^d)$. Hence together counting all such contributions, one will attain a summand of the form $$(Q \mathcal{S}')^{i,b}(n \cdot (w_1 \otimes w_2 + w_2 \otimes w_1) T^d),$$ for some $n \in \mathbb{Z}/2$. We know this to be zero by an argument as in Proposition \ref{propn:propositionftw}. Hence the only contributions we must count occur when $w_1 = w_2$, in which case if $v_1 \neq v_2$ or $u_1 \neq u_2$ then solutions come in pairs $(v_0, u_0), (v_1, v_2, u_1, u_2)$ and $(v_0, u_0), (v_2, v_1, u_2, u_1)$ which are not related by the $D_8$-action (hence are counted separately).

In particular, it is immediate that (with asymptotic conditions as given in Figure \ref{fig:ademmodulispace}) the number of pairs $(v_0,u),(v_1,u_1)$ up to the action of $D_8$, is the coefficient of $w_1$ in $Q \mathcal{S}^{j,d}(\alpha)$ multiplied by the coefficient of $z$ in $Q \mathcal{S}^{i,b}(w_1)$. Summing over all $w_1 \in \text{crit}(f)$, we get that the count of the moduli spaces of maps is then exactly the coefficient of $z$ in $Q \mathcal{S}^{i,b} \circ Q \mathcal{S}^{j,d}(\alpha)$ as required.
			\end{proof}

		\begin{rmk}
			Observe that the coefficients of $zT^0$ in $q_{D_{8}}(g m_1 \times \Psi b)(x)$ and in $q_{D_{8}}(g^2 m_1 \times \Psi b)(x)$ (corresponding to constant spheres) are the same: specifically, we are counting exactly the same moduli space in both cases. Consider the contributions of $q_{D_{8}}(g m_1 \times \Psi b)(x)$ and $q_{D_{8}}(g^2 m_1 \times \Psi b)(x)$ to Equation \eqref{equation:QAR} for constant spheres. These then cancel out modulo $2$, and so we are left with only $q_{D_8}( m_1 \times \Psi b)(x)$, which for constant $J$-holomorphic spheres is $Sq \circ Sq(x)$.
		\end{rmk}

			\begin{rmk}

			The term $qq_{j,0}(\alpha) \in QH^*(M)$ is the $h^{j}$ coefficient in $Q\mathcal{S}(\alpha) * Q\mathcal{S}(\alpha)$. This is one of the correction terms that can be computed, e.g. the $p=|\alpha|$ term in Corollary \ref{corollary:QAR}. 
			\end{rmk}

\section{$Q\mathcal{S}$ for blow-ups}
\label{sec:blowups}
Denote by $Q \mathcal{S}_{i,j}(x)$ the coefficient of $h^i T^j$ in $Q \mathcal{S}$ (where $|T|$ is the minimal Chern number of $M$). We will demonstrate calculations of $QS_{1,1}$ in two cases. The setup in both cases will be similar to the setup in \cite[Section 8]{blaier}, where Blaier computes the quantum Massey product. 

\subsection{$\mathbb{CP}^3$}

Fix two generic quadric hypersurfaces in $X = \mathbb{CP}^3$ . Their intersection $Y$ is an elliptic curve, hence a torus. We let $M = Bl_Y X$, equipped with the blowdown $\rho: M \rightarrow X$. Recall that there is a $\mathbb{CP}^1$-bundle $\pi: E \rightarrow Y$ over the torus and an inclusion $i: E \rightarrow M$ of the exceptional divisor $E$. Specifically $E$ is the projectivisation of the normal bundle of $y: Y \xhookrightarrow{} X$.

Consider the continuous $3$-disc bundle $\pi': DY \rightarrow Y$ such that $E \xhookrightarrow{} DY$ is an inclusion of the subbundle $E$, with the maps of fibres being inclusion of the boundary $S^2 \xhookrightarrow{} D^3$. Locally, let $U \subset Y$ is a trivialising neighbourhood for $Y$, so there is a homeomorphism $\phi_U: U \times S^2 \xrightarrow{\cong} \pi^{-1}(U)$. Then we define $DY$ to have the same trivialising neighbourhoods as $E$, i.e. $\phi'_U: U \times D^3 \cong \pi'^{-1}(U)$. Further, we require that the transition functions $\psi'_{U_1,U_2} :=(\phi'_{U_1})^{-1} \circ \phi'_{U_2}: (U_1 \cap U_2) \times D^3 \rightarrow (U_1 \cap U_2) \times D^3$ are defined by $$\psi'_{U_1,U_2}((r,\theta),x) = ((r, \psi_{U_1,U_2}(\theta)), x),$$ where here we use polar coordinates on $D^3$ and $\psi_{U_1,U_2}$ is the transition function on $Y$. One can use the Mayer-Vietoris sequence, by observing that $M \cup_{E} DY$ is homotopy equivalent to $X$, to write down the {\it long exact sequence of a blow-up},
\begin{equation}\label{lesblowup}
\xymatrix{
\ldots \ar@{->}[r]
&
H_*(E) \ar@{->}^-{i_* \oplus \pi_*}[r]
&
H_*(M) \bigoplus H_*(Y) \ar@{->}^-{\rho_* - y_*}[r]
&
H_*(X) \ar@{->}^-{\delta}[r]
&
H_{*-1}(E) \ar@{->}[r]
&
\ldots
}
\end{equation}
One can also use the homological Gysin sequence for the bundle $DY \rightarrow Y$ to get an exact sequence (after applying the Thom isomorphism and observing that $DY \rightarrow Y$ is a homotopy equivalence):
\begin{equation}\label{homgysin}
\xymatrix{
\ldots \ar@{->}[r]
&
H_*(E) \ar@{->}^-{\pi_*}[r]
&
H_*(Y) \ar@{->}^-{g}[r]
&
H_{*-3}(Y) \ar@{->}[r]^{\phi}
&
H_{*-1}(E) \ar@{->}[r]
&
\ldots
}
\end{equation}

Where $\delta, \phi$ are the connecting homomorphisms in the long exact sequence and the maps $i_*, \pi_*, \rho_*$ and $y_*$ are induced by the continuous maps above. The map $g$ is induced by the composition of $H_*(Y) \cong H_*(DY) \rightarrow H_*(DY, DY - Y)$ (where $DY - Y$ denotes the removal of the $0$-section) with the Thom isomorphism. Putting these together, we see that:

\begin{equation} \label{equation:eq111} H_2(M) \cong H_2(X) \oplus H_2(E)/H_2(Y) \end{equation} \begin{equation} \label{equation:eq222} i_*: H_3(E) \xrightarrow{\cong} H_3(M) \end{equation} \begin{equation} \label{equation:eq333} \phi: H_1(Y) \xrightarrow{\cong} H_3(E) .\end{equation} In particular $\dim H_3 (E) = 2$. The class of a sphere lifted from $\mathbb{P}^3$ to $M$ and the class of a fibre $\pi^{-1}(y)$ of $\pi$ over $y \in Y$ generate $H_2(M)$.

We calculate $c_1(TM)$. There is a natural embedding $j: M \rightarrow \mathbb{P}^3 \times \mathbb{P}^1$ as a complex hypersurface of bidegree $(2,1)$, with respect to the generators of $H_2(M)$ in the previous paragraph. By bidegree, we mean that the number of points of intersection of $M$ with a general curve $A$ of $\mathbb{P}^3 \times \mathbb{P}^1$ is the following:
\begin{itemize} 
\item if $A = \mathbb{P}^1 \times \mathbb{P}^0$, then we get an intersection number of $2$ because we are counting the number of solutions of a general quadric equation (the blow-up is defined as the set of $(t,[r:s]) \in \mathbb{CP}^3 \times \mathbb{CP}^1$ such that $r f(t) + s g(t) = 0$, where $f$ and $g$ are the quadric equations defining $Y$).
\item if $A = \mathbb{P}^0 \times \mathbb{P}^1$, then we obtain an intersection number of $1$ because a general point $p \in \mathbb{P}^3$ is such that there is only one point $q$ such that $(p,q) \in M.$
\end{itemize}

Functorality and the Whitney sum formula imply that $$c_{1}(TM) + c_1(v M) = c_1(T(\mathbb{P}^3 \times \mathbb{P}^1)|_M)$$ where $v M$ is the normal bundle of $M$ in $\mathbb{P}^3 \times \mathbb{P}^1$. Recall $c_1(T(\mathbb{P}^3 \times \mathbb{P}^1)|_M) = (4,2)$, and $c_1(v M) = (2,1)$ because it is the same as the degree (here we note that the Euler class can be reinterpreted as $j^* PD([M])$, represented by the self intersection of $M$, where $[M] \in H_6(\mathbb{P}^3 \times \mathbb{P}^1)$). Hence $c_1(TM) = (2,1)$. Therefore, when calculating $Q\mathcal{S}_{1,1}$ we only need to consider the spheres in the fibre class of $M$ as these are the only $J$-holomorphic spheres of Chern number $1$, which are confined to be in $E$.

Consider $Q\mathcal{S}_{1,1} : H^3(M) \rightarrow H^3(M)$. We will show that $Q\mathcal{S}_{1,1}|_{H^3(M) } = id$. First we show that calculating $Q\mathcal{S}_{1,1}|_{H^3(M)}$ reduces to calculating $Q\mathcal{S}_{1,1}: H^3(E) \rightarrow H^1(E)$. We use $i^{!}$ on cohomology to mean $PD \circ i_* \circ PD^{-1}$, and let $i_! = PD^{-1} \circ i^* \circ PD$ on homology.

\begin{lemma}
\label{lemma:SqE2M}
Fix $a \in H^3(M)$. Then
$$Q\mathcal{S}_{1,1} \circ i^* (a) = i^! \circ Q\mathcal{S}_{1,1}(a)$$ for $a \in H^3(M),$ i.e. \eqref{squareEMcommutes} commutes:

\begin{equation}\label{squareEMcommutes}
\xymatrix{
H^3(M)
\ar@{->}^-{Q\mathcal{S}_{1,1}}[r]
\ar@{->}_-{i^*}[d]
&
H^3(M)
\ar@{<-}^-{i^!}[d]
\\
H^3(E)
\ar@{->}^-{Q\mathcal{S}_{1,1}}[r]
&
H^1(E)
}
\end{equation}
\end{lemma}
\begin{proof}
		\begin{figure}
			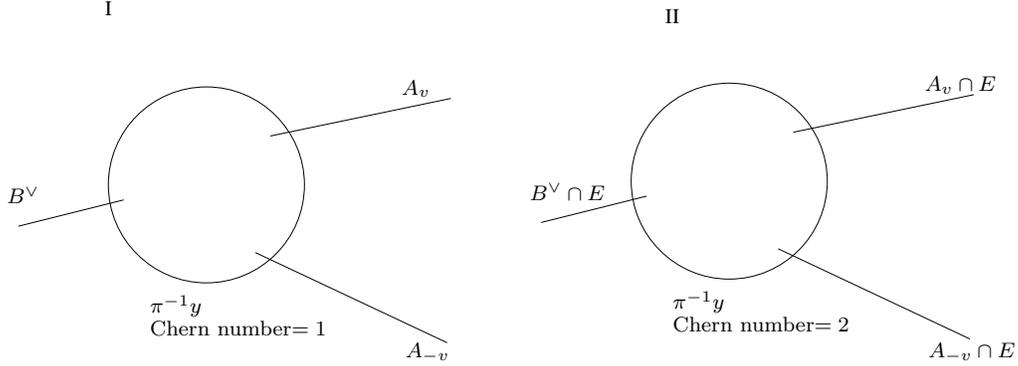
			\caption{Configurations for $Q\mathcal{S}_{1,1}$ on $M$ and $Q\mathcal{S}_{1,1}$ on $E$.}
			\label{fig:p_3setup}
		\end{figure}

Fix a generator $a \in H^3(M)$ whose Poincar\'e dual is represented by a smooth submanifold $A$. To compute the coefficient of $b$ in $Q \mathcal{S}_{1,1}(a)$, we choose $A_v$ for $v \in S^{\infty}$ (such that $A_v$ represents $PD(a)$ for each $v$) and $B^{\vee}$ that represents $PD(b^{\vee})$. To simplify the notation, $^{\vee}$ acts on $H_*(M)$ as the conjugation of the cohomological intersection dual by Poicar\'e duality, so $PD(b)^{\vee} := PD(b^{\vee})$. We therefore see that the coefficient of $b$ in $Q \mathcal{S}_{1,1}(a)$ is determined by counting setups as in Figure \eqref{fig:p_3setup}I.

In general, if $a$ is represented by $A$ then a representative of $i^*(a)$ is obtained by choosing some small perturbation $A'$ of $A$, such that $A'$ intersects $E$ transversely. Then the intersection $A' \cap E$, a submanifold of $E$, represents $i^*(a)$. Hence, supposing that we perturb $A$ to intersect $E$ transversely, if we choose $A_v$ to be a sufficiently small perturbation of $A$ for each $v$ then $A_v \cap E$ is transverse for each $v \in S^{\infty}$. By this procedure, we obtain representatives of $i^*A$ for each $v$ that we denote $A_v \cap E$, and we similarly obtain a representative $B^{\vee} \cap E$ of $i^*(B^{\vee})$. This can be done in such a way as to ensure that the space of setups as in Figure \eqref{fig:p_3setup}II are transverse: in particular any perturbation of $A_v \cap E$ may be extended to yield a perturbation of $A_v$ in a small neighbourhood of $E$. By making this perturbation sufficiently small, we ensure that the intersection between $A_v$ and $E$ remains transverse.

Observe that using the choice of basis induced from $H_1(Y)$ by Equations \eqref{equation:eq222} and \eqref{equation:eq333}, a direct computation shows that for any $b \in H^3(M)$, \begin{equation} \label{Bvees} i_! PD(b)^{\vee} = (i_*^{-1} PD(b))^{\vee}. \end{equation} Hence in particular $B^{\vee} \cap E$ represents $((i^!)^{-1}b)^{\vee}$. With all of this in mind, the coefficient of $(i^!)^{-1}b$ in $Q \mathcal{S}_{1,1}(i^* a)$ is determined by counting setups as in Figure \eqref{fig:p_3setup}II.

Hence, to show that the diagram commutes we need to show that setups of type I and II biject. This is immediate, however, because every $J$-holomorphic curve $u$ of Chern number $1$ (in $M$) is contained in $E$. Hence if $(v,u)$ is a setup of type I (i.e. $u$ intersects $A_{\pm v}$ and $B^{\vee}$) then $u$ will automatically intersect with $E \cap A_{\pm v}$ and $E \cap B^{\vee}$ (hence $(v,u)$ will be a setup of type II), and vice versa. 
\end{proof}

\begin{rmk}
In Lemma \ref{lemma:SqE2M}, in order to show that $E \cap A_v$ is of the correct form for Definition \ref{defn:singqss}, we need to demonstrate that in fact $$\bigsqcup_{v \in S^i} (E \cap A_{v}) \times \{ v \}$$ is of the form $A' \times S^i$ for some smooth manifold $A'$. We make two observations:
\begin{itemize}
\item The space $\sqcup_{v \in S^i} A_{v} \times \{ v \}$ is a smooth manifold by assumption, transversely intersecting $E \times S^i$. Hence $\sqcup_{v \in S^i} (E \cap A_{v}) \times \{ v \} = (\sqcup_{v \in S^i} A_{v} \times \{ v \}) \cap (E \times S^i)$ is a smooth manifold. Further, the map $\sqcup_{v \in S^i} E \cap A_{v} \times \{ v \} \rightarrow S^i$ induced by projection to the second factor is a proper surjective submersion between two smooth manifolds, hence it is a fibre bundle by Ehresmann's Lemma. %proper: Suppose v varies in a compact set. given an open cover, for each $(a,v)$ \in the fibre F_v over \{ v \}, there are finitely many open sets \{ U_v \} covering F_v x \{ v \}. Consider the (finite) intersection of the second projection of these open sets around v. This is an open set. Repeat this for each v, and a finite number of these finite intersections then covers this compact set of v. \\ submersion: we know it is a smooth manifold of a certain dimension, and the fibrewise direction has the appropriate dimension, so the direction transverse to the fibre has dimension i, so this is a submersion.

\item This fibre bundle is trivial, because we know that it must extend to a fibre bundle over $D^{i+1,+}$, the upper hemisphere in $S^i$, which is a contractible base. 
\end{itemize}

A similar lemma to \ref{lemma:SqE2M} would hold if we replaced embedded submanifolds by pseudocycles.
\end{rmk}

Note that the index of the codomain of $Q \mathcal{S}_{1,1}$ changes by $2$, between $H^3(M)$ and $H^1(M)$, in Diagram \eqref{squareEMcommutes}. This comes from the fact that $i^!$ changes cohomological degree by 2 (and is also to be expected because the minimal Chern number of $E$ is $2$, whereas the minimal Chern number of $M$ is $1$). 

Henceforth, we will use the Morse theoretic definition of the quantum square. Observe that $E = Y \times \mathbb{P}^1$ so we may pick the Morse function on $E$ to be $f+g$ where $f: Y \rightarrow \mathbb{R}$ and $g: \mathbb{P}^1 \rightarrow \mathbb{R}$, such that $g$ has two critical points of index $0,2$, which we call $a_0, a_2$, and $f$ has critical points $b_0, b_1, b_1^{'}, b_2$ (whose indices are the subscripts). Recalling that $\pi: E \rightarrow Y$ is the projection map, $$(\pi^!)^{-1} (b_1)  = (b_1, a_0) \text{ and } \pi^*(b_1) = (b_1,a_2).$$ 

\begin{lemma}
\label{lemma:SqE2Y}
Let $Sq_i(x)$ be the coefficient of $h^i$ in $Sq(x)$. Then $Q\mathcal{S}_{1,1} =\pi^* \circ Sq_1 \circ \pi^!$.
\end{lemma}
\begin{proof}
Recall that input elements of $H^3(E)$ correspond to $(b_1,a_2)$ or $(b^{'}_1,a_2)$, which project down under $\pi$ to $b_1$ or $b_1^{'}$ respectively. Output elements of $H^1(E)$ correspond to $(b_1,a_0)$ or $(b^{'}_1,a_0)$. 

We will show that pairs $(\tilde{v},\tilde{u})$ in the moduli space used to calculate the coefficient of $c$ in $Q\mathcal{S}_{1,1}(x)$ correspond to pairs $(v,u)$ in the moduli space yielding the coefficient of $\pi^* c$ in $Sq_{1}(\pi^{!} x)$. For clarity we will fix $x = (b_1,a_2)$ and $c = (b_1,a_0)$ (hence $\pi^{*} (b_1) = x$ and $\pi^! (c) = b_1$), although the argument follows identically for any choice of $x$ and $c$. For conciseness we denote the moduli spaces of pairs respectively as $\mathcal{M}_Q$ and $\mathcal{M}$ for $Q \mathcal{S}$ on $E$ and $Sq$ on $M$.

Consider a pair $(v,u) \in \mathcal{M}_Q$, as on the right hand side of Figure \eqref{fig:sqtosqQblowup}. We observe that, using the projection $\pi: E \rightarrow Y$, the setup $(v, \pi u)$ is one that is counted when calculated the coefficient of $b_1$ in $Sq_{1}(b_1)$, i.e. $(v, \pi u) \in \mathcal{M}$. This is because a fibre sphere in $E$ lives above a point $y \in Y$, hence the incidence condition of the flowlines attaching to $S^2$ at the points $0,1,\infty$ translates under this projection to the three flowlines coinciding at the point $y \in Y$. As the transversality condition is generic, we may choose the perturbations $f_{v,s}$ and $g_{v,s}$ of $f$ and $g$ in such a way that both the moduli space $\mathcal{M}$ and $\mathcal{M}_Q$ are transverse.

We show that every pair $(v,u) \in \mathcal{M}$ arises uniquely from the a pair $(\tilde{v}, \tilde{u}) \in \mathcal{M}_Q$ as in the previous paragraph. Consider such a pair $(v,u)$: specifically, the image of $u$ consists of three perturbed half-flowlines meeting at some point $y$. Consider the $- \nabla f_{v,s}$ flowline $l$, which is the image of $u$ restricted to one of the two positive halflines. This flowline is asymptotic to $b_1$ in $Y$, hence it lifts uniquely to a $- \nabla (f_{v,s} + g_{v,s})$ flowline that is asymptotic to $(b_1,a_2)$ in $E$. The uniqueness is because $a_2$ is the maximum of $g$, and hence there is a unique $- \nabla g_{v,s}$-flowline $L$ asymptotic to $a_2$. Specifically this is the flowline $L: [0,\infty) \rightarrow S^2$ such that $L(s) = a_2$ for $s \gg 0$.  See Figure \ref{fig:sqtosqQblowup}. Likewise the output flowline on $Y$, which is a $- \nabla f_s$-flowline, lifts uniquely to a $-\nabla (f_s+g_s)$-flowline on $E$, asymptoting to $(b_1,a_0)$, which is unique because $a_0$ is the minimum of $g$.

		\begin{figure}
			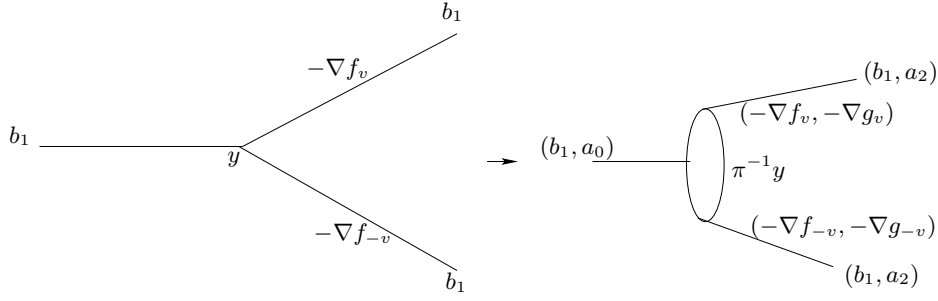
			\caption{Lifting configurations of $Sq$ on $Y$ to $Q\mathcal{S}$ on $M$.}
			\label{fig:sqtosqQblowup}
		\end{figure}

Moreover, because this setup lifts from $Y$ we know that the three flowlines will all intersect the $J$-holomorphic sphere $\pi^{-1} y$ at $0$. We also know where each lifted flowline intersects $J$-holomorphic sphere $\pi^{-1} y$, as this is determined by the gradient of $f+g$ on each of the flowlines at $s=0$. Hence there is a unique $J$-holomorphic sphere that fits into the lifted setup, giving a unique configuration on $E$ corresponding to the configuration on $Y$.
\end{proof}

\begin{proof}[Proof of Equation \eqref{equation:blowupID1} in Theorem \ref{thm:blowupID}]
	Note that $Sq|_{H^1(Y)}: H^1(Y) \rightarrow H^1(Y)$ is the identity, which is known by the definition of $Sq$. Lemmas \ref{lemma:SqE2M} and \ref{lemma:SqE2Y} imply that Diagram \eqref{squareMEYcommutes} commutes.

\begin{equation}\label{squareMEYcommutes}
\xymatrix{
H^3(M)
\ar@{->}^-{Q\mathcal{S}_{1,1}}[r]
\ar@{->}_-{i^*}[d]
&
H^3(M)
\ar@{<-}^-{i^!}[d]
\\
H^3(E)
\ar@{->}^-{Q\mathcal{S}_{1,1}}[r]
\ar@{->}_-{\pi^!}[d]
&
H^1(E)
\ar@{<-}_-{\pi^*}[d]
\\
H^1(Y)
\ar@{->}^-{Sq_{1}}[r]
&
H^1(Y)
}
\end{equation}

The abelian group $H^i(E)$ is generated by $\{ (b_1,a_{i-1}), (b'_1,a_{i-1}) \}$ for $i=1,3$. From the axioms of $Sq$, we know that $Sq_1 = id$. Then from Lemma \ref{lemma:SqE2Y}: \begin{equation} \label{equation:eq211} Q\mathcal{S}_{1,1}(b_1,a_0) = (b_1, a_2) \text{ and } Q\mathcal{S}_{1,1}(b^{'}_1,a_0) = (b^{'}_1, a_2) \end{equation} 

We apply the isomorphism between Morse and classical cohomology and then Poincar\'e duality to the Morse cocycles $(b_1,a_2) \in H^3(E)$ and $(b_1,a_0) \in H^1(E)$. This yields cycles $B_1 \in H_1(E)$ and $B_3 \in H_3(E)$. Likewise we define $B'_i \in H_i(E)$ for $(b'_1,a_{3-i})$ for $i=1,3$. We recall that $^{\vee}$ is the intersection dual on homology (defined as the conjugation by Poincar\'e duality of the duality on cohomology, for our given basis). Note that \begin{equation} \label{equation:eq212} B^{\vee}_3 = B'_1, \end{equation} and so on. In this notation $Q\mathcal{S}_{1,1} (B_1) = B_3$. Observe that $B_3 \cap B_3 = \emptyset$ so we see that $i_* B_3 \cap i_* B_3 = \emptyset$ (which is immediate if one chooses a generic submanifold representative: then nonintersection in $E$ implies nonintersection in $M$). As $H_3(M)$ is generated by $i_* B_3$ and $i_* B'_3$, this implies
\begin{equation} \label{istarB} (i_* B_3)^{\vee} = i_* B'_3.\end{equation} 

By Equation \eqref{istarB}, \begin{equation} \label{equation:bbb1} i_* \circ Q\mathcal{S}_{1,1} \circ i_! ((i_* B_3)^{\vee}) =  i_* \circ Q\mathcal{S}_{1,1} \circ i_! ((i_* B'_3)).\end{equation}

By Equation \eqref{Bvees}, then Equations \eqref{equation:eq211} and \eqref{equation:eq212}, \begin{equation} \label{equation:bbb2} i_* \circ Q\mathcal{S}_{1,1} \circ i_! ((i_* B_3)^{\vee}) = i_* \circ Q\mathcal{S}_{1,1} (B^{\vee}_3) = i_* B'_3. \end{equation}

From Equations \eqref{equation:bbb1} and \eqref{equation:bbb2}, along with identical calculations for the other generators, plus the fact that $i_*$ is an isomorphism, we deduce that $i_* \circ Q\mathcal{S}_{1,1} \circ i^* = id.$ Diagram \eqref{squareMEYcommutes} then implies that $$Q \mathcal{S}_{1,1} = i_* \circ Q\mathcal{S}_{1,1} \circ i^* = id.$$
\end{proof}

\subsection{$\mathbb{CP}^1 \times \mathbb{CP}^1 \times \mathbb{CP}^1$}

Now let $X = \mathbb{CP}^1 \times \mathbb{CP}^1 \times \mathbb{CP}^1$, with $Y \subset X$ defined by the intersection of two generic linear hypersurfaces. ``{\it Linear}" means that we require that the defining equation of the hypersurfaces are linear in the coordinates of each $\mathbb{CP}^1$ (when the other coordinates are treated as constants). The subvariety $Y$ is in fact a torus, which one can see by using the adjunction formula: in particular, one proves that $K_Y = 0$. Specifically, $K_X = (-2,-2,-2)$ and the two linear hypersurfaces are $(1,1,1)$ by definition, and hence $K_Y = (1,1,1) + (1,1,1) + (-2,-2,-2) = 0$. Then the genus $g$ of $Y$ satisfies $g = 1 + (\deg K_Y)/2 = 1$, hence $Y$ is a surface of genus $1$.

Define $M = Bl_Y X$. Using a similar method to the $\mathbb{CP}^3$ case, we can show that the Chern class of $M$ is $(1,1,1,1)$, where the first three entries correspond to lifting the $J$-holomorphic spheres on each of the $3$ coordinates of $X$, and the final entry corresponds to a fibrewise $J$-holomorphic sphere in the exceptional divisor. Hence, when calculating $Q\mathcal{S}_{1,1}: H^3(M) \rightarrow H^3(M)$ there are contributions from the fibre direction plus those from $J$-holomorphic spheres in $X$ that have been lifted to $M$. The fibrewise contributions are calculated in exactly the same way as for $\mathbb{CP}^3$, so we turn our attention to the spheres lifted from $\mathbb{CP}^1 \times \mathbb{CP}^1 \times \mathbb{CP}^1$.

\begin{proof}[Proof of Equation \eqref{equation:blowupID2} in Theorem \ref{thm:blowupID}]
Suppose the defining linear equations for $Y$ are $P_1(x,y,z)$ and $P_2(x,y,z)$ in local coordinates on $\mathbb{P}^1 \times \mathbb{P}^1 \times \mathbb{P}^1$. Fixing $x$ and $y$, there is at most one solution $z$ such that $P_1(x,y,z) = P_2(x,y,z)=0$. Hence, let $S = \{ x \} \times \{ y \} \times \mathbb{CP}^1$ be a $J$-holomorphic curve in $X$, and $\tilde{S}$ its lift to $M$. By the previous, $\tilde{S} \cap E$ is at most $1$ point. If $A \in H_3(M)$ then recall from Equation \eqref{equation:eq222}, we may assume that $A = i_* A_E$ for some $A_E \in H_3(E)$. Recall that to calculate $Q \mathcal{S}_{1,1}(a)$, where $A = PD (a)$, we need to choose some $A_v$ satisfying the transversality conditions. We may pick the representatives $A_v$ to be some $D_v \times \mathbb{P}^1$, where $D_v$ is a representative of $D \in H_1(Y)$. Then assuming $S$ is not contained in $Y$, there are no solutions to Figure \ref{fig:p_1_3setup}. For such a solution, we would need that $\tilde{S}$ intersects $E$ in at least $2$ points (as $A_v \subset E$ for all $v$), which we know is impossible. Hence the space of such setups is transverse, because it is empty.

		\begin{figure}
			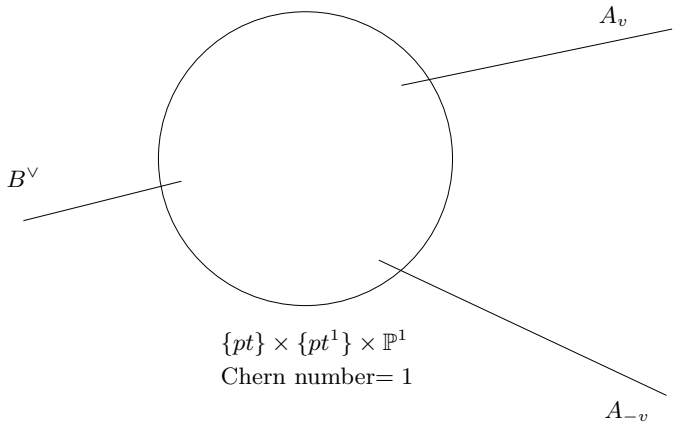
			\caption{Configurations for contributions to $Q\mathcal{S}_{1,1}$ from lifts on $\mathbb{P}^1 \times \mathbb{P}^1 \times \mathbb{P}^1$.}
			\label{fig:p_1_3setup}
		\end{figure}

The case $S \subset Y$ is not possible, because there is no degree $1$ holomorphic map $\mathbb{P}^1 \rightarrow Y$. 
\end{proof}

\appendix
\section{Equivariant Compactification}
\label{sec:equivariantcompact}
We give a more in-depth treatment of Equation \eqref{equation:Sq'chainmap} in Definition \ref{defn:mss}. 

Consider a $1$-dimensional moduli space $\mathcal{M}_{i}(a_1,a_{2}, a_3)$. Specifically, we require that $|a_1| -|a_2| - |a_3| + i = 1$. Here, the notation is as in Section \ref{subsec:msss}. One characterises the different possible limits $(v_{\infty}, u_{\infty})$ of some sequence of pairs $(v_j,u_j)$ in $\mathcal{M}_{i}(a_1,a_{2}, a_3)$. There are two possibilities for each: 
\begin{itemize}
\item either $v_{\infty} \in D^{i,+}$ or $v_{\infty} \in \partial D^{i,+} = D^{i-1,+} \cup D^{i-1,-}$, and
\item either $u_{\infty}$ is a map taking as its domain the $Y$-shaped graph with a single broken edge (broken) or $u_{\infty}$ is a map with the $Y$-shaped graph as its domain (unbroken).
\end{itemize}

No codimension $1$ boundary setups are present when either:
\begin{itemize}
\item $u_{\infty}$ is unbroken and $v_{\infty} \in D^{i,+}$, or
\item $u_{\infty}$ is broken and $v_{\infty} \in \partial D^{i,+}$.
\end{itemize}
In the case where $v_{\infty} \in D^{i-1,+}$ and $u_{\infty}$ is unbroken, one obtains the coefficient of $a_1$ in $Sq'_{i-1}(a_2 \otimes a_3) = Sq'_{i}((a_2 \otimes a_3)h)$. Similarly, when $v_{\infty} \in D^{i-1,-}$ and $u_{\infty}$ is unbroken, one obtains the coefficient of $a_1$ in $Sq'_{i-1}(a_3 \otimes a_2) = Sq'_{i}((a_3 \otimes a_2)h)$. When $v_{\infty} \in D^{i,+}$ and $u_{\infty}$ is unbroken, this is the standard compactification of the $Y$-shaped graph, and thus one obtains the coefficient of $a_1$ in $dSq'_i(a_2 \otimes a_3) + Sq'_i((d a_2) \otimes a_3) + Sq'_i(a_2 \otimes (d a_3))$. These are all of the terms in Equation \eqref{equation:Sq'chainmap}.

\section{Notes on Transversality}
Following are more detailed treatments of the relevant transversality considerations for the sections of this paper.

\subsection{The Morse Steenrod Square: Section \ref{subsec:msss}}
\label{subsec:mssrmks}
In the second of the conditions at the beginning of Section \ref{subsec:msss}, for $a_1, a_2 \in \text{crit}(f)$, we required that our moduli spaces were transverse. More specifically, denote by $$W^s(a_2, f^2_{v,s}) = \{ p \in M | {\displaystyle \lim_{s \rightarrow \infty} } \phi_{v,s}(p) = a_2 \},$$ where $\phi_{v,t}$ is the $1$-parameter family of diffeomorphisms defined for $v \in S^i$ and $t \ge 0$ by \begin{equation} \label{equation:backint} \dfrac{d \phi_{v,t}}{dt}(s) = - \nabla f^2_{v,s} \text{ and } \phi_{v,0} = id. \end{equation} We recall first that if $a$ is in fact a sum of critical points of $f$ such that $da = 0$, then there is a (partial) compactification $\overline{W^s(a,f)}$ of $W^s(a,f)$, the stable manifold of the critical point $a$ for the Morse function $f$. The compactification $\overline{W^s(a,f)}$ comes equipped with an evaluation map $E: \overline{W^s(a,f)} \rightarrow M$, so that $\overline{W^s(a,f)}$ becomes a pseudocycle (i.e. \cite{schwarzmorsesingiso}). We consider $$\tilde{\euscr{W}}^{i} := \overline{W^u(a_1, f)} \times \bigsqcup_{v \in S^i} \overline{W^s(a_2, f^2_{v,s})} \times \overline{W^s(a_2, f^2_{-v,s})} \times \{ v \},$$ and $\euscr{W}^{i} = \tilde{\euscr{W}}^{i} / (\mathbb{Z}/2)$ (i.e. the pseudocycle descending to the quotient). The $\mathbb{Z}/2$-action is induced from $M \times M \times M \times S^i$, fixing the first $M$ factor, swapping the second and third $M$ factors and acting antipodally on $S^i$. This induces a smooth pseudocycle representative of a cycle in $M \times ((M \times M) \times_{\mathbb{Z}/2} S^i)$, which we can see for example by precomposing the smooth $\mathbb{Z}/2$-equivariant map \begin{equation} \label{equation:pseudohere}\begin{array}{l} \nu: M \times M \times M \times S^i \rightarrow M \times M \times M \times S^i, \\ (x_0, x_1, x_2,v) \mapsto (x_0, \phi_{v,-1}(x_1), \phi_{-v,-1}(x_2), v). \end{array} \end{equation} with the pseudocycle $$E \times E \times E \times id: \overline{W^u(a_1,f)} \times \overline{W^s(a_2,f)} \times \overline{W^s(a_2,f)} \times S^i \rightarrow M \times M \times M \times S^i.$$ Here $\phi_{v,-1}$ is the diffeomorphism induced by backwards integrating along $f^2_{v,s}$ from $s=1$ to $s=0$: specifically, $\dfrac{d \phi_{v,t}}{dt}(-s) = \nabla f^2_{v,1-s} \text{ and } \phi_{v,0} = id$ for $s \in [0,1]$, analogously to Equation \eqref{equation:backint}. Indeed, $\phi_{v,-1} = \phi_{v,1}^{-1}$. We abusively denote by $\euscr{W}^{i}$ the pseudocycle associated to the quotient by $\mathbb{Z}/2$ of $\nu \circ (E \times E \times E \times id)$. 

For transversality, we require that for $$\zeta: M \times \mathbb{RP}^i \rightarrow M \times( (M \times M) \times_{\mathbb{Z}/2} S^i), \text{ such that } (x, [v]) \mapsto (x,[x,x,v]),$$ the pseudocycle $\euscr{W}^{i}$ is transverse to $\zeta$. Indeed, the coefficient $n_{a_1,a_2,a_3,i}$ in Definition \ref{defn:mss} is the intersection number of these two pseudocycles. More generally in that definition, we considered moduli spaces $\mathcal{M}_i(a_1,a_2,a_3)$ for general $a_1,a_2,a_3 \in \text{crit}(f)$. This does not yield a pseudocycle (as there may be codimension $1$ boundary strata). Nonetheless, such an intersection may still be defined at the chain level, and the operation $Sq$ itself is only well defined on cohomology (i.e. exactly when the codimension $1$ strata of these moduli spaces may be glued so as to define a smooth manifold).

To demonstrate that we can choose such an $f_{v,s}$ (and indeed such a choice is generic in a reasonable sense), we first assume that the $f_{v,s}$ are constrained to $U_f$ from Section \ref{subsec:prelimbcncon}. We next, as in the same section, assume that $f_{v,s} =\beta(s) f_{v,0} + (1-\beta(s)) f$ for the monotonic bump function $\beta : \mathbb{R} \rightarrow [0,1]$. Observe that for $v \in S^0$, there is always a choice of $f_{v,0}$ and $f^1_s$ such that $f_{v,0}$, $f_{-v,0}$ and $f^1_0$ yield transverse moduli spaces $\mathcal{M}'_i(a_1,a_2,a_3)$ of dimension $|a_1| - |a_2| - |a_3| + i$, as in Definition \ref{defn:mss}. The above pseudocycle description using $\zeta$ and $\euscr{W}^{i}$ reveals that this is simply a classical transversality question. We wish to ensure transversality for all possible incoming and outgoing critical points $a_1,a_2,a_3$, which requires only a finite number of choices. Having shown a base case, we then proceed inductively. Next suppose that we have made a choice of $f_{v,0}$ for $v \in S^{i-1}$. Choose a small open collar neighbourhood $N^{i-1}$ of $S^{i-1}$ in $S^i$. As the setup is transverse when we let $v \in S^{i-1}$ vary, if we pick some small perturbation of $f_{v,0}$ for $v \in N^{i-1}$, then the pseudocycle intersection remains transverse when we let $v$ vary in $N^{i-1}$. We then choose an extension of $f_{v,0}$ to $S^i$. Observe that we may freely (i.e. without requiring any consistency conditions on our $f_{v,0}$) pick a small perturbation away from a (potentially smaller than $N^{i-1}$) collar neighbourhood of $S^{i-1}$. This ensures that the intersection is now transverse when we vary $v$ in $S^{i}$, and remains transverse when restricting to $S^{i-1}$. Further, as we make a countable number of choices, this condition is generic. This is an example of family Morse homology, as covered in \cite{hutchingsfamilies} (the Floer theoretic example of this is made in \cite[Section 4c]{seidel}).

\begin{rmk}
The difference between the case considered here and that in \cite{seidel} is that our parameter space is $\mathbb{RP}^i$, and not the space of flowlines of some fixed Morse function $h: \mathbb{RP}^i \rightarrow \mathbb{R}$. In our instance there is no technical difference between these two options, because the underlying chain complexes have a trivial action of $\mathbb{Z}/2$. The technical importance of using a parametrisation by the space of flowlines appears when considering gluing and compactness of equivariant flowlines in the equivariant Morse or Floer complex: see \cite[4b]{seidel}.
\end{rmk}
%It is then classical that for a generic choice of $f_{\cdot,0}: S^{\dim M} \times M \rightarrow \mathbb{R}$, the moduli space $\mathcal{M}_{\dim M}(b,a,a)$ is transverse.

\begin{rmk}
Note that the transversality condition required here is weaker than requiring $W^s(a_2, f^2_{v,s}) \pitchfork W^s(a_2,f^2_{-v,s})$ and $(W^s(a_2, f^2_{v,s}) \cap W^s(a_2,f^2_{-v,s})) \pitchfork W^u(a_1,f)$ intersect transversely for all $v$ (which is impossible in general with our given conditions). Indeed, it is the failure of transversality for particular isolated $v$ that ensures we obtain interesting Steenrod squares.
\end{rmk}

\subsection{The Cartan Relation: Section \ref{subsec:Cartan}}
\label{subsec:appendcartrel}
We will, for convenience, assume that $f^i_{v,s,t} = f$ for $s \ge 1$ when $i=3,4,6,7$.

In Section \ref{subsec:Cartan}, when constructing the moduli spaces that we used to prove the Cartan relation, we asked that we chose our $f^i_{v,s,t}$ to be ``generic at each vertex". To illustrate what is meant by this, observe that for each vertex of the graph corresponding to $t \in T$, there are either three or five edges meeting at this vertex. The idea is that for each $p$, the smooth functions $f^p_{v,0,t}$ and $f^p_{v,t,t}$ are chosen in such a way that all of the intersections occur transversely within the space $M^{\times 9} \times_{\mathbb{Z}/2} S^{\infty}$ (denoting the $9$-fold Cartesian product of $M$ by $M^{\times 9}$). Here, labelling the copies of $M$ by $1,...,9$, the $\mathbb{Z}/2$-action is denoted as the transposition $(56)(89)$.

More specifically, we first define (analogously to Equation \eqref{equation:backint}) some $1$-parameter families of diffeomorphisms $\phi^p_{v,s,t}: M \rightarrow M$ (with parameter $s$), via $$\dfrac{d \phi^p_{v,s,t}}{ds}(s_0) = - \nabla f^p_{v,s_0,t} \text{ and } \ \phi_{v,0,t} = id \text{ for all } v,t.$$ For the moduli space denoted $\mathcal{M}_1(x,y,z)$ we define a map $$\zeta_1: M \times M \times M \times \mathbb{RP}^i \rightarrow ( M^{\times 9}) \times_{\mathbb{Z}/2} S^i,$$ $$(x_1, x_2, x_3, [v]) \mapsto [x_1, x_1, x_1, x_2, x_2, x_2,x_3, x_3, x_3,v],$$ recalling that in this instance $f^2_{v,s,t}$ and $f^5_{v,s,t}$ (hence $\phi^2_{v,s,t}$ and $\phi^5_{v,s,t}$) are independent of $v$. We also define: $$\nu_1: M^{\times 7} \times S^i \times T_K \rightarrow M^{\times 9} \times S^i,$$ where $T_K = [0,K]$ is an interval for any $K \in \mathbb{R}_{>0}$, $$\left( \begin{array}{l} m_1 \\ m_2 \\ m_3 \\ m_4 \\ m_5 \\ m_6 \\ m_7 \\ v \\ t \end{array} \right) \mapsto \left( \begin{array}{l} \phi^1_{v,1,t}(m_1) \\ m_2 \\ m_3 \\ \phi^2_{v, t,t}(m_2) \\ \phi^3_{v,-1, t}(m_4) \\ \phi^3_{-v,-1, t}(m_5) \\ \phi^5_{v,t,t}(m_3) \\ \phi^6_{v,-1, t}(m_6) \\ \phi^6_{-v,-1, t}(m_7) \\ v \end{array} \right),$$ to construct a pseudo-cycle bordism when composing with $$\left( \begin{array}{l} E \\ id_M \\ id_M \\ E \\ E  \\ E \\ E  \\ id_{S^i} \\ id_{T_K} \end{array} \right) {\Huge :}\left( \begin{array}{l} \overline{W^u(a_3,f)} \\ M \\ M \\ \overline{W^s(a_1,f)} \\ \overline{W^s(a_1,f)} \\ \overline{W^s(a_2,f)} \\ \overline{W^s(a_2,f)} \\ S^i \\ T_K \end{array} \right) \rightarrow M^{\times 7} \times S^i \times T_K,$$ where $E$ is the evaluation map. Transversality in this instance means that we want to make a choice of $f^i_{v,0,t}$ (for all $i$) and $f^j_{v,t,t}$ (for $j=2,5$) such that the two pseudocycles intersect transversely for a generic choice of $t$, including $t=0, 1$. To do this we appeal once again to genericity for Morse homology, observing as previously that we may always perturb any choice of $f^i_{v,0,t}$ and $f^j_{v,t,t}$ in a generic way so that these pseudocycles intersect transversely. This ensures that the moduli space $\mathcal{M}'_1(a_1,a_2,a_3)$ is a manifold of dimension $1$ if $|a_3| - 2|a_1| - 2 |a_2| + i =0$ (intuitively there is a limiting process as $K \rightarrow \infty$, in practice one appeals to a gluing theorem). There is a similar argument for $\mathcal{M}'_2(a_1,a_2,a_3)$.

\begin{rmk}
\label{rmk:EZ/2}
			Some of the conditions require a choice of $f^p_{v,s,t}$ that becomes independent of $v$ for large enough $t$. The reason why this is not an unrealistic request is that we need to retain a nontrivial $v$ dependency on the moduli space in some form, but it need not be everywhere. In essence we are constructing a copy of $E \mathbb{Z}/2$: there is a tuple $(f^{1}_{v,s,t},...,f^{7}_{v,s,t})$ for each $v \in S^{\infty}$, along with the associated $\mathbb{Z}/2$ action $\iota$. Our choices ensure that $\iota: v \mapsto -v$ always acts freely on tuples, hence the union of tuples over all $v \in S^{\infty}$ is an $E \mathbb{Z}/2 \subset (C^{\infty}(M))^7$. Nonetheless, each pair $(f^{i}_{v,s,t},f^{j}_{v,s,t})$ for $(i,j)=(2,5),(3,4),(6,7)$ also defines an $E \mathbb{Z}/2$, as long as the action remains free on that pair (which will be the case for a generic choice of $v$-dependence). So as long as $\iota$ acts freely on some pair, the set of these tuples is an $E \mathbb{Z}/2$. 
		\end{rmk}

\begin{rmk}
Intuitively, the $M^{\times 9}$ corresponds to the $9$ possible finite ends of a flowline in the domain of $\mathcal{M}'_1(a_1,a_2,a_3)$ (as each vertex has valence $3$). The $M^{\times 7}$ corresponds to the $7$ edges in the domain of $\mathcal{M}'_1(a_1,a_2,a_3)$. Then the condition of intersecting with $\zeta_1$ is exactly identifying the edges at the vertexes at their finite endpoints.
\end{rmk}

\subsection{Transversality for the quantum Steenrod square: Section \ref{sec:SqQviaMorse}}
\label{subsection:transvholspheres}

The argument that we can ensure that the moduli spaces $\mathcal{M}_{i,j}(b,a)$ are carved out transversely in Section \ref{sec:SqQviaMorse} is identical to that in Appendix \ref{subsec:mssrmks}, after turning this problem into a suitable intersection problem involving pseudocycles. 

In particular, fixing $a,b$ we consider $$\tilde{\euscr{W}} := \bigsqcup_{v \in S^i} W^u(b,f) \times W^s(a, f^2_{v,s}) \times W^s(a, f^2_{-v,s}) \times \{ v \} \subset M \times M \times M \times S^i,$$ which similarly to the previous appendices when passing to the partial compactification (if $a$ is a closed sum of critical points) induces (via the evaluation map) a pseudocycle that we denote $\euscr{W}$ in $M \times ((M \times M) \times_{\mathbb{Z}/2} S^i)$. We require that this intersects transversely with the pseudocycle defined by the evaluation map, $$ev: \overline{\mathcal{M}}(j,J) \times S^i \rightarrow M \times M \times M \times S^i,$$ defined by $$ev(u,v) = (u(0), u(1), u(\infty), v).$$ We observe that this is classically a pseudocycle, such as in \cite[Theorem 6.6.1]{jholssympl} where we denote by $\overline{\mathcal{M}}(j,J)$ the compactification of the smooth moduli space of $u: S^2 \rightarrow M$ such that $u_*[S^2]$ is a homology class of Chern number $jN$ and $u$ is $J$-holomorphic. 

Once we quotient the two pseudocycles by the given $\mathbb{Z}/2$-action, our requirement is then that we can perturb the choice of Morse functions in such a way that the Morse pseudocycle is transverse to the pseudocycle $ev$. This is the same as the previous transversality problems.

\subsection{The quantum Cartan relation: Section \ref{sec:quancar}}
\label{subsec:bordismquantumcartantrans}

In the proof of Lemma \ref{lemma:lem1}, and in Appendix \ref{subsec:appendcartrel}, we constructed a bordism. We will make this somewhat clearer, using the language of the previous appendices.

Strictly, for each $t$ the setup as given can be shown to be identical to the intersection of two pseudocycles, as usual. The first is:
$$ev: \mathcal{M}_j(J) \times \mathcal{M}_j(J) \times \mathcal{M}_j(J) \times  S^i \rightarrow M^9 \times S^i,$$ acting by $$ev: (u_1,u_2,u_3,v) \mapsto (u_1(0),u_1(1), u_1(\infty), u_2(0),u_2(1), u_2(\infty),u_3(0),u_3(1), u_3(\infty),v).$$ 

The second is (for $\chi_i : X_i \rightarrow M \times S^i$ and $\gamma_i: Y_i \rightarrow M \times S^i$ being the ``input" pseudocycles, and $\zeta: Z \rightarrow M$ the ``output" pseudocycle) the following map: $$g: M \times M \times X \times X \times Y \times Y \times Z \times S^i \times [0,K] \rightarrow M^9 \times S^i,$$ for some large $K$, such that \begin{equation} \label{equation:qcarpseudo} g: \left( \begin{array}{l}a_1 \\ a_2 \\ x_1 \\ x_2 \\ y_1 \\ y_2 \\ z \\ v \\ t \end{array} \right) \mapsto \left( \begin{array}{l} \zeta(z) \\  a_1 \\ a_2 \\ \phi^2_{v,t,t}(a_1,v) \\ \phi^3_{v,-1,t} \circ \chi_v(x_1,v) \\ \phi^4_{v,-1,t} \circ \chi_{-v}(x_2,-v) \\ \phi^5_{v,t,t}(a_2, v) \\ \phi^6_{v,-1,t} \circ \gamma_v(y_1,v) \\ \phi^7_{v,-1,t} \circ \gamma_{-v}(y_2,-v) \\ v\end{array} \right) \end{equation} We observe that this is a pseudocycle bordism, identical to the construction for the classical Cartan relation of Appendix \ref{subsec:appendcartrel}. For each fixed $t$, the pseudocycle is then obtained by using the parameter $t \in [0,K]$. Observe that for sufficiently large $K$, an (equivariant) gluing theorem (in this case it is nothing more than a standard gluing theorem: see Appendix \ref{sec:equivariantgluing}) shows that the $K$-end of the bordism corresponds to a count using broken trajectories. Further, the entire bordism $g$ above shows that the pseudocycles corresponding to the $t=0$ and $t=K$ ends are bordant, hence have the same intersection with the pseudocycle $ev$. 

\subsection{Transversality for the quantum Adem relations: Section \ref{sec:QAR}}
\label{subsec:quantumademrels}
We give only brief details for this case, as one just needs to edit the previous transversality arguments using the discussion in this Appendix.

We recall that in our previous constructions, we proved transversality by observing that we may assume that $v$ varies in some fundamental domain $D$ of the $\mathbb{Z}/2$-action on $S^{\infty}$, and then we may freely choose $f_{v}$ for $v \in D$. We would like to make a similar claim for the more general finite groups that we consider.

Suppose that we are defining operations as in Section  \ref{sec:QAR}, using $ES_4$ as our parameter space. One constructs a pseudocycle as in Equation \eqref{pseudoaaa2}, and intersects this with the evaluation pseudocycle in Equation \eqref{pseudoaaa1}. In order to ensure transversality, one requires that the $f_v$ may be chosen (where $v \in E S_4$) sufficiently generically. In particular, we must ensure that the invariance conditions are not troublesome. Suppose that $D$ were a fundamental domain of the $S_4$-action on $ES_4$. The only way that the invariance conditions could prove to be a problem for transversality is if one of the invariance conditions were to in some way relate $(1 p) \cdot D$ with $(1 q) \cdot D$ for some $p,q = 1,2,3,4$. If this were the case then we could not freely choose our Morse function on $f_{(1 q) \cdot v}$ with $v \in D$ (because then this would influence our choice of $f_{(1 p ) \cdot D)}$). Recall that the invariance conditions required that $f_{(23) \cdot v} = f_{(24) \cdot v} = f_v$. The pseudocycle that one must write down involves $f_v$, $f_{(12) \cdot v}$, $f_{(13) \cdot v}$ and $f_{(14) \cdot v}$. It is then sufficient to demonstrate that the cosets of $G:= \langle (23), (24) \rangle \subset S_4$ are exactly $\{ G, (12) \cdot G, (13) \cdot G, (14) \cdot G \}$, which is a straightforward verification.

\section{Equivariant Gluing}
\label{sec:equivariantgluing}
We will not repeat the classical gluing argument, as pertinent details are well known for example as in \cite[Chapter 10]{jholssympl}. Instead, for completeness we mention gluing for the equivariant case (i.e. as necessary for Section \ref{subsec:Cartan}). A more general equivariant gluing argument, as in for example \cite[Section 4c]{seidel}, is not necessary for this paper. In the case of Morse flowlines, observe that the Morse-Smale condition is open, so for a fixed metric $g$ we may assume that each pair $(f_{v,s},g)$ is Morse-Smale. The equivariant gluing theorem for a broken pair of flowlines with parameter $v$ simply works applying a standard gluing argument to a $- \nabla f_{v,s}$-flowline concatenated with a $- \nabla f_{v,\pm \infty}$-flowline. The conditions on the functions $f^p_{v,s,t}$ for $p=1,...,7$ for Section \ref{subsec:Cartan} were specifically chosen for this to be the case. Specifically, for large $t$, one chooses $f^p_{v,s,t}$ such that $f^p_{v,s,t}$ is independent of $v$ and $s$ when $|s|$ is sufficiently large. 

The gluing theorem for holomorphic spheres holds likewise: one may apply a gluing theorem on $\mathcal{M}(J,j)$ of stable genus $0$ $J$-holomorphic maps of Chern number $j$, and this immediately provides a gluing theorem for $\mathcal{M}(J,j) \times S^i$. However, an important point must be highlighted: when descending to $\mathcal{M}(J,j) \times_{\mathbb{Z}/2} S^i$, recalling the action of $\mathbb{Z}/2$ on $\mathcal{M}(J,j)$ is by $$(z \mapsto z/(z-1)) \in PSL(2,\mathbb{C}),$$ suppose that $u$ is a nodal $J$-holomorphic map (with at least one node). Then if $[u,v] \in \mathcal{M}(J,j) \times_{\mathbb{Z}/2} S^i$, in general there is an ambiguity in how one defines a glued solution (which depends on a choice of lift to $\mathcal{M}(J,j) \times S^i$). In particular, if the domain of $u$ is $m_1$ (in the notation of Figure \ref{fig:m05elmts}) then there is no way to coherently glue the domain while respecting the $\mathbb{Z}/2$-action. This is because $m_1$ is an isolated point in the fixed point set of $\overline{M}_{0,5}$.

\section{Constructions for $ED_8$ and $ES_4$}
		\label{sec:ed8es4}
In this appendix, we construct $ED_8$ and $ES_4$ as the union of a countable nested family of submanifolds (with respective $D_8$ and $S_4$ actions).

\subsection{The construction for $ED_8$}
		\label{subsec:ed8}
Consider the contractible space $S^{\infty} \times (S^{\infty} \times S^{\infty})$, along with an action of $D_8 = \langle r, a | r^2 = a^4 = 1, r \cdot a = a^3 \cdot r \rangle \subset S_4$ (where $r=(13)(24)$ and $a= (1324)$) as follows:

			\begin{itemize}
			\item $r \cdot (z,( z_1, z_2)) = (-z, (z_2,z_1))$.
			\item $a \cdot (z,(z_1,z_2)) = (-z, (-z_2, z_1))$.
			\end{itemize}

			This is a free action of $D_8$ on the contractible space $S^{\infty} \times (S^{\infty} \times S^{\infty})$, which can be checked by verifying all possibilities. We call this space $E$, and it is a the model for $ED_8$. We denote $B = E / D_8$, to contrast with the $BD_8$ as constructed in the next section. Observe that there are finite dimensional submanifolds $S^i \times (S^j \times S^j)$, for each $(i,j)$, which are invariant under the $D_8$-action. These will be referred to as $E^{i,j}$, and let $B^{i,j} = E^{i,j} / D_8$.

\subsection{The construction for $ES_4$}
		\label{subsec:es4}

Recall from Appendix \ref{subsec:ed8}, the space $$\mathbb{S}_2 := S^{\infty} \times (S^{\infty} \times S^{\infty}),$$ defined for $D_8$ using $r=(12)(34)$ and $a= (1324)$. We may define similarly $\mathbb{S}_3, \mathbb{S}_4$ for the conjugates $(23) \cdot D_8 \cdot (23)$ and $(24) \cdot D_8 \cdot (24)$ (corresponding to copies of $D_8$ with $r=(13)(24)$ and $r=(14)(23)$) respectively. 

Let $S^{\infty}_{\mathbb{C}} = \cup_{i \ge 1} S^{2i-1}$ denote the infinite dimensional sphere taken as the union of the odd dimensional spheres $S^{2i-1} \subset \mathbb{C}^{i}$. Let $$\mathbb{S}_{\mathbb{C}} = S^{\infty}_{\mathbb{C}} \times \mathbb{S}_{2}  \times \mathbb{S}_{3}  \times \mathbb{S}_{4}.$$ This remains contractible, but now has a free $S_4$ action as follows (to define the action, it is sufficient to define it on the generators $(12),(13),(14)$ and show that all relations are satisfied). Firstly, these act on $$\mathbb{S}_{2}  \times \mathbb{S}_{3}  \times \mathbb{S}_{4},$$ by:

$$(12): \left( \begin{array}{ccc} a &  a_1 & a_2 \\ b &  b_1 & b_2 \\ c &  c_1 & c_2 \end{array} \right) \mapsto  \left( \begin{array}{ccc} -a &  -a_1 & a_2 \\ c &  c_2 & c_1 \\ b &  b_2 & b_1 \end{array}\right)$$
$$(13): \left( \begin{array}{ccc} a &  a_1 & a_2 \\ b &  b_1 & b_2 \\ c &  c_1 & c_2 \end{array} \right)  \mapsto 
\left( \begin{array}{ccc} c &  -c_2 & c_1 \\ -b &  -b_1 & b_2 \\ a &  a_2 & -a_1 \end{array}\right) $$
$$(14): \left( \begin{array}{ccc} a &  a_1 & a_2 \\ b &  b_1 & b_2 \\ c &  c_1 & c_2 \end{array} \right)\mapsto 
\left( \begin{array}{ccc} b &  -b_2 & -b_1 \\ a &  -a_2 & -a_1 \\ -c &  -c_1 & c_2 \end{array}\right). $$
Here we denote an element of this space as a matrix $$\left( \begin{array}{ccc} a &  a_1 & a_2 \\ b &  b_1 & b_2 \\ c &  c_1 & c_2 \end{array}\right),$$ where for example $(a,(a_1,a_2)) \in \mathcal{S}_2$ and so on.

Further, to decide the action of these transpositions $(12),(13),(14)$ on $S^{\infty}_{\mathbb{C}}$, we pick three reflections of $S^{\infty}$ (i.e. involutions that restrict to reflections in each $S^{2i-1} \subset S^{\infty}$), such that the composition of any two of these reflections are a rotation by $2 \pi /3$ in each complex coordinate, using the identification with the unit sphere $S^{2i-1} \subset \mathbb{C}^i$. Alternatively, this is multiplication by the third root of unity $\zeta \in \mathbb{C}$. To see that we can do this, recall that any element of $O(2n)$ may be block diagonalised into $2 \times 2$ blocks. It is then sufficient to show that this may be done for $S^1$ (and then extending likewise for each $2 \times 2$ block). Observe that if one takes the diameters of $S^1$ through $0$, $\zeta$ and $\zeta^2$ then reflection through any two of these three lines suffices.

It remains to show that this is indeed a well defined free action of $S_4$, which we will not prove here but can be checked by exhaustion. Note further that the inclusion of $D_8 \subset S_4$ from Appendix \ref{subsec:ed8} demonstrates this $ES_4$ as an $ED_8$. We denote by $ES_4^{i,j,k}$ the $S_4$-invariant submanifold of $ES_4$ corresponding to $$S^{2i-1} \times (S^j \times (S^k \times S^k)) \times (S^j \times (S^k \times S^k)) \times (S^j \times (S^k \times S^k)).$$

Choosing any point $p \in ES_4$ that is fixed by the reflection associated to $(12)$, there is a homotopy equivalence between our $ES_4$ and $E$ (from Appendix \ref{subsec:ed8}), via the projection $$\rho: ES_4 \rightarrow \mathbb{S}_2,$$ which is in fact a $D_8$-equivariant homotopy equivalence with respect to the action of $D_8 \subset S_4$. To see that we can do this, recall the following argument that $S^{\infty}_{\mathbb{C}}$ is contractible: there is a shift map $$T: S^{\infty}_{\mathbb{C}} \rightarrow S^{\infty}_{\mathbb{C}}, \ (z_0,z_1,\ldots) \mapsto (0, z_0, z_1,\ldots),$$ where each $z_i \in \mathbb{C}$. Pick some $p \in S^{\infty}_{\mathbb{C}} - \text{Im}(T)$. Then $T$ is both homotopic to the identity map on $S^{\infty}_{\mathbb{C}}$ and the constant map at $p$ via a (normalised) linear interpolation (the key being that this interpolation is never zero). A similar argument shows that $S^{\infty}$ is contractible. We then observe that $\rho$ is $D_8$-equivariant, and the $ES_4 / D_8 \rightarrow E / D_8$ is a homotopy equivalence because it is a fibration over $E / D_8$ with a contractible fibre. Using the long exact sequence of homotopy for a fibration, this implies that $\rho$ is a weak equivalence. By construction, our choices of $ES_4$ and $E$ are CW-complexes, hence this is indeed a homotopy equivalence.

For notational purposes, we will denote $BD_8 = ES_4 / D_8$ (specifically, using the contractible space $ES_4$ with a free $S_4$ action as constructed in Appendix \ref{subsec:es4}, but only quotienting by $D_8$). We will denote $B = E / D_8$ (specifically, using the contractible space $E$ from Appendix \ref{subsec:ed8}), noting that by standard theory of classifying spaces $B$ and $BD_8$ are both homotopy equivalent, as in the previous paragraph.

\bibliographystyle{plain}
\bibliography{bibliothesis}

\end{document}